\documentclass[11pt]{article}\textwidth 160mm\textheight 235mm
\usepackage{amsfonts} \usepackage{amssymb} \usepackage{graphicx}
\usepackage{amsmath} \usepackage{amsthm} \usepackage{mathrsfs}
\usepackage{tikz} \usepackage{cancel} \usepackage{todonotes}
\usetikzlibrary{matrix} \usetikzlibrary{arrows,shapes}
\usepackage{cite} \usepackage{hyperref}
\hypersetup{colorlinks=true,urlcolor=blue}
\expandafter\let\expandafter\oldproof\csname\string\proof\endcsname
\let\oldendproof\endproof \renewenvironment{proof}[1][\proofname]{
\oldproof[\ttfamily\scshape \bf #1.] }{\oldendproof}
\def\ve{\varepsilon}  
\def\emp{\emptyset}  
\def\dom{{\rm dom}\,}  \def\epi{{\rm epi\,}}
\def\rge{{\rm rge\,}}  \def\min{\mbox{\rm minimize}}
 
   \def\B{\mathbb
B}   
\def\ox{\overline{x}} \def\oy{\overline{y}} \def\oz{\overline{z}}
  \def\disp{\displaystyle}
\def\tto{\rightrightarrows} \def\Hat{\widehat}
\def\Tilde{\widetilde} \def\Bar{\overline} \def\ra{\rangle}
\def\la{\langle} \def\ve{\varepsilon} \def\epsilon{\varepsilon}
\def\ox{\bar{x}} \def\oy{\bar{y}} \def\oz{\bar{z}}
 \def\ov{\bar{v}} 
   
\def\co{\mbox{\rm co}\,} 
 
\def\gph{\mbox{\rm gph}\,} \def\epi{\mbox{\rm epi}\,}
 \def\dom{\mbox{\rm dom}\,}
\def\ker{\mbox{\rm ker}\,}

 \def\dn{\downarrow} 
\def\ph{\varphi} \def\emp{\emptyset} \def\st{\stackrel}
\def\oR{\Bar{\R}} \def\lm{\lambda} \def\gg{\gamma} \def\dd{\delta}
\def\al{\alpha}  \def\N{{\rm I\!N}}
\def\R{\mathbb{R}}  \def\vt{\vartheta}

\begin{document}
\newtheorem{Theorem}{Theorem}[section]
\newtheorem{Proposition}[Theorem]{Proposition}
\newtheorem{Remark}[Theorem]{Remark}
\newtheorem{Lemma}[Theorem]{Lemma}
\newtheorem{Corollary}[Theorem]{Corollary}
\newtheorem{Definition}[Theorem]{Definition}
\newtheorem{Example}[Theorem]{Example}
\newtheorem{Algorithm}[Theorem]{Algorithm}
\newtheorem{Problem}[Theorem]{Problem}
\renewcommand{\theequation}{{\thesection}.\arabic{equation}}
\renewcommand{\thefootnote}{\fnsymbol{footnote}} \begin{center}
{\bf\Large A Generalized Newton Method for Subgradient Systems}\\[1ex]
PHAM DUY KHANH\footnote{Department of Mathematics, Ho Chi Minh City University
of Education, Ho Chi Minh City, Vietnam E-mails:
khanhpd@hcmue.edu.vn; pdkhanh182@gmail.com}, BORIS S.
MORDUKHOVICH\footnote{Department of Mathematics, Wayne State
University, Detroit, Michigan, USA. E-mail: boris@math.wayne.edu.}
and VO THANH PHAT \footnote{Department of Mathematics, HCMC
University of Education, Ho Chi Minh City, Vietnam and Department of
Mathematics, Wayne State University, Detroit, Michigan, USA.
E-mails: phatvt@hcmue.edu.vn; phatvt@wayne.edu.} \end{center}

\small{\bf Abstract}. This paper proposes and develops a new Newton-type algorithm to solve subdifferential inclusions defined by subgradients of extended-real-valued prox-regular functions. The proposed algorithm is formulated in terms of the second-order subdifferential of such functions that enjoys extensive calculus rules and can be efficiently computed for broad classes of extended-real-valued functions. Based on this and on metric regularity and subregularity properties of subgradient mappings, we establish verifiable conditions ensuring well-posedness of the proposed algorithm and its local superlinear convergence. The obtained results are also new for the class of equations defined by continuously differentiable functions with Lipschitzian gradients (${\cal C}^{1,1}$ functions), which is the underlying case of our consideration. The developed algorithms for prox-regular functions and  its extension to a structured class of composite functions   are formulated in terms of proximal mappings and  forward-backward envelopes.  Besides numerous illustrative examples and comparison with known algorithms for ${\cal C}^{1,1}$ functions and generalized equations, the paper presents applications of the proposed algorithms to  regularized least square problems   arising in statistics, machine learning, and related disciplines\\[1ex]
\vspace*{-0.2in}

\normalsize \section{Introduction and
Overview}\label{intro}\vspace*{-0.05in} \setcounter{equation}{0}

Recall that, given a function $\ph\colon\R^n\to\R$, which is twice
	continuously differentiable (${\cal C}^2$-smooth) around some point
	$\ox\in\R^n$, the {\em classical Newton method} to solve the
	nonlinear {\em gradient system} \begin{equation}\label{gra}
		\nabla\ph(x)=0 \end{equation} constructs the iterative procedure
	\begin{equation}\label{clas-newton} x^{k+1}:=x^k+d^k\;\mbox{ for
			all}\;k\in\mathbb{N}:=\big\{1,2,\ldots\big\}, \end{equation} where
	$x^0\in\mathbb{R}^n$ is a given starting point, and where $d^k$ is a
	solution to the linear system \begin{equation}\label{newton-iter}
		-\nabla\varphi(x^k)=\nabla^2\varphi(x^k)d^k,\quad k=0,1,\ldots,
	\end{equation} written in terms of the Hessian matrix
	$\nabla^2\varphi(x^k)$ of $\varphi$ at $x^k$. As known in classical
	optimization, Newton's algorithm in \eqref{clas-newton} and
	\eqref{newton-iter} is well-defined (i.e., the equations in
	\eqref{newton-iter} are solvable for $d^k$), and the sequence of its
	iterations $\{x^k\}$ superlinearly (even quadratically) converges to
	a solution $\ox$ of \eqref{gra} if $x^0$ is chosen sufficiently
	close to $\ox$ and if the Hessian matrix $\nabla^2\varphi(\ox)$ is
	positive-definite. Note also that, besides being a necessary
	condition for local minimizers of $\ph$, the gradient system
	\eqref{gra} is important for its own sake and holds not only for
	local minimizers and local maximizers of $\ph$, but in essentially
	larger settings. Furthermore, a counterpart of the classical Newton
	method has been developed for solving more general nonlinear
	equations of the type $f(x)=0$, where
	{$f\colon\R^n\to\R^n$} is a continuously
	differentiable (${\cal C}^1$-smooth) mapping, and where
	$\nabla^2\varphi(x^k)$ in \eqref{newton-iter} is replaced by the
	Jacobian matrix of $f$ at the points in question. We are not going
	to deal with the latter method and its extensions in this paper
	while being fully concentrated on the gradient systems \eqref{gra}
	and their appropriate subgradient counterparts.
	
	Concerning the gradient systems of type \eqref{gra} where $\ph$ may
	not be ${\cal C}^2$-smooth around $\ox$, we mention that the
	enormous literature has been devoted to developing various versions
	of the (generalized) Newton method; see, e.g., the books by Dontchev
	and Rockafellar \cite{Donchev09}, Facchinei and Pang \cite{JPang},
	Izmailov and Solodov \cite{Solo14}, Klatte and Kummer \cite{Klatte},
	Ulbrich \cite{Ul}, and the references therein. The vast majority of
	such extensions deals with functions $\ph$ in \eqref{gra} of class
	${\cal C}^{1,1}$ (or ${\cal C}^{1+}$ in the notation of Rockafellar
	and Wets \cite{Rockafellar98}) around $\ox$, which consists of
	continuously differentiable functions with locally Lipschitzian
	derivatives. The most popular generalized Newton method to solve
	\eqref{gra} for functions of this type is known as the {\em
		semismooth Newton method} initiated independently by Kummer
	\cite{Kummer} and by Qi and Sun \cite{LQi}. In the semismooth Newton
	method, the Hessian matrix of $\ph$ in \eqref{newton-iter} is
	replaced by the (Clarke) {\em generalized Jacobian} (collection of
	matrices) of the gradient mapping $\nabla\ph$. Then the
	corresponding Newton iterations are well-defined around
	$\ox$ and exhibit a local superlinear convergence to the solution
	$\ox$ of \eqref{gra} provided that each matrix from the generalized
	Jacobian is {\em nonsingular} and the gradient mapping
	$\nabla\ph$ is {\em semismooth} around $\ox$ in the sense of Mifflin
	\cite{Mifflin}. The latter property has been well investigated and
	applied in variational analysis and optimization, not only in
	connection with the semismooth Newton method. Besides the
	aforementioned books and papers, we refer the reader to Burke and Qi
	\cite{Burke}, Henrion and Outrata \cite{ho}, and Meng et al.
	\cite{msz} among many other publications on the theory and
	applications of such functions.\vspace*{0.03in}
	
	In the case of ${\cal C}^{1,1}$ functions $\ph$, our Newton-type
	algorithm proposes replacing \eqref{newton-iter} by
	\begin{equation}\label{newtonC11}
		-\nabla\varphi(x^k)\in\partial^2\varphi(x^k)(d^k),\quad
		k=0,1,\ldots, \end{equation} where $\partial^2\ph$ stands for {\em
		second-order subdifferential/generalized Hessian} of $\ph$
	introduced by Mordukhovich \cite{m92} for arbitrary
	extended-real-valued functions. This construction reduces to the
	classical Hessian operator for ${\cal C}^2$-smooth functions while
	maintaining key properties of the latter for important classes of
	functions {in broad generality}; see below. In what
	follows we obtain efficient conditions ensuring the solvability of
	the inclusions in \eqref{newtonC11} and superlinear convergence of
	iterates $\{x^k\}$ if the starting point $x^0$ is sufficiently close
	to $\ox$. As shown in this paper, the obtained conditions allow us
	to use the proposed algorithm \eqref{newtonC11} to solve systems
	\eqref{gra} with ${\cal C}^{1,1}$ functions $\ph$ in the situations
	where the semismooth Newton method cannot be applied.

	Observe that algorithm \eqref{newtonC11} has been recently
	introduced and developed, in an equivalent form, in the paper by
	Mordukhovich and Sarabi \cite{BorisEbrahim} to find {\em tilt-stable
		local minimizers} for functions $\ph$ of class ${\cal C}^{1,1}$.
	We'll discuss tilt stability of local minimizers, the notion
	introduced by Poliquin and Rockafellar \cite{Poli}, in the
	corresponding place below with a detailed comparison of the results
	obtained in this paper and in the paper by Mordukhovich and Sarabi.
	Note that here we do not assume that $\ox$ is a local minimizer of
	$\ph$, not even talking about its tilt stability. Observe also that
	both the latter paper and the current one employ the {\em
		semismooth$^*$ property} of $\nabla\ph$ that has been recently
	introduced and developed by Gfrerer and Outrata \cite{Helmut}
	for set-valued mappings as an improvement of the
	standard semismoothness for locally Lipschitzian functions
	used before.
	
	The main thrust of this paper is on developing a generalized Newton
	method to solve {\em subgradient inclusions} of the following type:
	\begin{equation}\label{subgra-inc} 0\in\partial\varphi(x),
	\end{equation} where $\ph\colon\R^n\to\oR:=(-\infty,\infty]$ is an
	extended-real-valued function belonging to a broad class of {\em
		prox-regular} and subdifferentially continuous, which overwhelmingly
	appear in variational analysis and optimization. The subdifferential
	operator used in \eqref{subgra-inc} is understood as the
	(Mordukhovich) {\em limiting subdifferential} of
	extended-real-valued functions that agrees with the classical
	gradient for ${\cal C}^1$-smooth functions and the subdifferential
	of convex analysis when $\ph$ is convex. In very general settings,
	the limiting subdifferential enjoys comprehensive calculus rules
	that can be found in the books by Mordukhovich
	\cite{Mordukhovich06,Mor18} and by Rockafellar and Wets
	\cite{Rockafellar98}. {If $\ph$ is a
		$\mathcal{C}^1$-smooth function, then
		$\partial\varphi=\nabla\varphi$ and that the inclusion
		\eqref{subgra-inc} reduces to the gradient system \eqref{gra}.}
	
	{One of the} generalized Newton algorithms, which is designed in this paper to solve the
	subgradient inclusion \eqref{subgra-inc}, is also based on the
	second-order subdifferential $\partial^2\ph$ with replacing
	\eqref{newtonC11} by \begin{equation}\label{subgra-prox}
		-v^k\in\partial^2\ph(x^k-\lm v^k,v^k)(\lm v^k+d^k)\;\mbox{ with
		}\;v^k:=\frac{1}{\lm}\Big(x^k-{\rm Prox}_\lm\ph(x^k)\Big),
	\end{equation} where ${\rm Prox}_\lm\ph(x)$ stands for the {\em
		proximal mapping} of $\ph$ corresponding to a constructive choice of
	the parameter $\lm>0$. This form is shown to be closely related to
	the Newton-type algorithm developed by Mordukhovich and Sarabi
	\cite{BorisEbrahim}, in terms of Moreau envelopes with somewhat
	different choice of parameters, to find tilt-stable minimizers of
	prox-regular and subdifferentially continuous functions $\ph$. As
	mentioned,  the latter is not an ultimate framework of
	\eqref{subgra-inc}.
	
	Here we develop a new approach to solvability of systems
	\eqref{subgra-prox} with respect to the directions $d^k$ and to
	local superlinear convergence of iterates $x^k\to\ox$ under certain
	{\em metric regularity} and {\em subregularity} properties of the
	subdifferential mapping $\partial\ph$. In particular, all our
	assumptions hold if $\partial\ph$ is semismooth$^*$ at the reference
	point and {\em strongly metrically regular} around it. As shown by
	Drusvyatskiy and Lewis \cite{dl}, for the class of prox-regular and
	subdifferentially continuous functions the latter property is
	equivalent to tilt stability of $\ox$ required by Mordukhovich and
	Sarabi \cite{BorisEbrahim} {\em provided that} $\ox$ is a local
	minimizer of $\ph$, which is not assumed in this paper.
	
	{We also extend our generalized Newton method to
		solve the following structured class of {\em composite optimization
			problems} given in the form \begin{equation}\label{nocvcomposite}
			\min \; \varphi(x):= f(x) + g(x) \quad \text{subject to }\; x \in
			\R^n, \end{equation} where $f:\R^n\to \R$ is $\mathcal{C}^2$-smooth
		while the {\em regularizer} $g:\R^n \to \overline{\R}$ is
		prox-regular and subdifferentially continuous. Problems written in
		format \eqref{nocvcomposite} frequently arise in many applied areas
		such as machine learning, compressed sensing, image processing, etc.
		Since $g$ is generally extended-real-valued, the unconstrained
		format \eqref{nocvcomposite} encompasses problems of {\em
			constrained optimization}. If, in particular, $g$ is the indicator
		function of a closed set, then \eqref{nocvcomposite} becomes an
		optimization problem with geometric constraints. In our algorithm to
		solve problems of type \eqref{nocvcomposite} we employ the machinery
		of {\em forward-backward envelopes}.}
	
	The developed generalized Newton method for subgradient inclusions
	is finally applied to solving  {regularized least
		square problems} that appear in practical models of statistics,
	machine learning, etc. For such problems, we compute the
	second-order subdifferential and the proximal mapping from
	\eqref{subgra-prox} entirely in terms of the given data, derive
	explicit calculation formulas, and then provide
	{solvability and local convergence results for our
		algorithms applied to problems of this class}.\vspace*{0.03in}
	
	The rest of the paper is organized as follows. Section~\ref{prel}
	presents and discusses those notions of variational analysis and
	generalized differentiation, which are broadly used in the
	formulations and proofs of the main results obtained below.
	Section~\ref{sec:solv} is devoted to {\em solvability} of the {\em
		generalized equations} \begin{equation}\label{ge-cod} -\ov\in
		D^*F(\ox,\ov)(d)\;\mbox{ for }\;d\in\R^n, \end{equation} where
	$D^*F(\ox,\ov)(\cdot)$ is the {\em coderivative} of a set-valued
	mapping $F\colon\R^n\tto\R^n$ that is associated with the limiting
	subdifferential $\partial\ph$ as defined in Section~\ref{prel}. The
	framework of \eqref{ge-cod} encompasses all the versions
	\eqref{newton-iter}, \eqref{newtonC11}, and \eqref{subgra-prox} of
	the Newton-type algorithms discussed above. Using well-developed
	calculus rules of limiting generalized differentiation allows us to
	prove that \eqref{ge-cod} is solvable for $d$ if the mapping $F$ is
	{\em strongly metrically subregular} at the reference point
	$(\ox,\ov)$. Furthermore, the {\em strong metric regularity} of $F$,
	in particular, ensures the solvability of \eqref{ge-cod} and the
	{\em compactness} of the  {(generalized) Newton
		directions} therein {\em around} the point in question. The
	solvability results established in Section~\ref{sec:solv} for
	coderivative inclusions are then applied in Section~\ref{sec:solvN}
	to solvability issues for {\em generalized Newton systems} of types
	\eqref{newtonC11} and \eqref{subgra-prox} involving the second-order
	subdifferential $\partial^2\ph$. In this way we identify broad
	classes of functions $\ph\colon\R^n\to\oR$ for which the required
	assumptions on $F=\partial\ph$ are satisfied.
	
	Section~\ref{sec:newtonC11} presents a generalized Newton algorithm
	to solve the gradient equations \eqref{gra} with functions $\ph$ of
	class ${\cal C}^{1,1}$ according to the iteration procedure
	\eqref{newtonC11}. The main result of this section establishes a
	local {\em superlinear convergence} of iterates \eqref{newtonC11} to
	a designated solution $\ox$ of \eqref{gra} under the
	semismoothness$^*$ of the gradient mapping $\nabla\ph$ at $\ox$ and
	under merely its {\em metric regularity} around this point. Even in
	the case of tilt-stable local minimizers of $\ph$, the obtained
	result improves the one from Mordukhovich and Sarabi
	\cite{BorisEbrahim}. We also compare the new algorithm with some
	other generalized Newton methods and show, in particular, that our
	algorithm is well-defined and exhibits a superlinear convergence of
	iterates when the semismooth Newton algorithm cannot be even
	constructed.
	
	Section~\ref{sec:prox} {and
		Section~\ref{sec:Newcomposite}} are the culmination of the paper.
	They describe and justify new Newton-type algorithms to solve the
	subgradient inclusions \eqref{subgra-inc}, where
	$\varphi\colon\R^n\to\oR$ is a {\em prox-regular} function,
	and where $\varphi\colon\R^n\to\oR$ is the cost
	function of the {\em composite optimization problem}
	\eqref{nocvcomposite}, respectively. Note that both
	these extended-real-valued frameworks of $\ph$ incorporates problems
	of {\em constrained optimization} for which inclusion
	\eqref{subgra-inc} provides a necessary condition for local
	minimizers. The results obtained here justifies a constructive and
	well-defined algorithm, with a verifiable choice of the starting
	point, that superlinearly converges to the solution $\ox$ of
	\eqref{subgra-inc} under about the same assumptions as in
	Section~\ref{sec:newtonC11}, but being now addressed to
	$\partial\ph$ instead of $\nabla\ph$. In fact, the proofs of the
	main results in this section are based on the reduction to the
	${\cal C}^{1,1}$ case by using Moreau envelopes,
	{forward-backward envelopes}, and the machinery of
	variational analysis taken from Mordukhovich \cite{Mor18} and
	Rockafellar and Wets \cite{Rockafellar98}.
	
	In Section~\ref{lasso} we present applications of the developed
	Newton-type algorithms to solving  some nonsmooth,
	convex and nonconvex {\em regularized least square problems},
	where we completely calculate all the algorithm
	parameters in terms of the given problem data.
	
	The concluding Section~\ref{conc} summarizes the major contributions
	of the paper and discusses some topics of the future research. Our
	notation is standard in variational analysis and optimization and
	can be found in the aforementioned books \cite{Mor18} and
	\cite{Rockafellar98}. Recall that
	$\mathbb{B}_r(\bar{x}):=\big\{x\in\R^n\;\big|\;\|x-\bar{x}\|\le
	r\big\}$ stands for the closed ball with center $\bar{x}$ and radius
	$r>0$.\vspace*{-0.1in}

\section{Variational Analysis: Preliminaries and
Discussions}\label{prel}\vspace*{-0.05in}\setcounter{equation}{0}

Here we present the needed background material from variational
	analysis and generalized differentiation by following the books of
	Mordukhovich \cite{Mordukhovich06,Mor18} and Rockafellar and Wets
	\cite{Rockafellar98}.
	
	Given a set $\Omega\subset\mathbb{R}^s$ with $\oz\in\Omega$, the
	(Bouligand-Severi) {\em tangent/contingent cone} to $\Omega$ at $\oz$ is
	\begin{equation}\label{tan}
		T_\Omega(\oz):=\big\{w\in\R^s\;\big|\;\exists\,t_k\dn 0,\;w_k\to
		w\;\mbox{ as }\;k\to\infty\;\mbox{ with }\;\oz+t_k w_k\in\Omega\big\}.
	\end{equation} The (Fr\'echet) {\em regular normal cone} to $\Omega$
	at $\bar{z}\in\Omega$ is defined by \begin{equation}\label{rnc}
		\widehat{N}_\Omega(\bar{z}):=\Big\{v\in\mathbb{R}^s\;\Big|\;\limsup_{z\overset{\Omega}{\rightarrow}\bar{z}}\frac{\langle
			v, z-\bar{z}\rangle}{\|z-\bar{z}\|}\le 0\Big\}, \end{equation} where
	the symbol $z\overset{\Omega}{\rightarrow}\bar{z}$ indicates that
	$z\to\bar{z}$ with $z\in\Omega$. It can be equivalently described
	via a duality correspondence with \eqref{tan} by
	\begin{equation}\label{dua} \Hat
		N_\Omega(\oz)=T^*_\Omega(\oz):=\big\{v\in\R^s\;\big|\;\la v,w\ra\le
		0\;\mbox{ for all }\;w\in T_\Omega(\oz)\big\}. 
	\end{equation} 
	The
	(Mordukhovich) {\em limiting normal cones} to $\Omega$ at
	$\bar{z}\in\Omega$ is defined by \begin{equation}\label{lnc}
		N_\Omega(\bar{z}):=\big\{v\in\mathbb{R}^s\;\big|\;\exists\,z_k\st{\Omega}{\to}\bar{z},\;v_k\to
		v\;\text{ as }\;k\to\infty\;\text{ with
		}\;v_k\in\widehat{N}_\Omega(z_k)\big\}. \end{equation} Note that the
	regular normal cone \eqref{rnc} is always convex, while the limiting
	normal cone \eqref{lnc} is often nonconvex (e.g., for the graph of
	$|x|$ at $\oz=(0,0)$), and hence it cannot be obtained by the
	duality correspondence of type \eqref{dua} from any tangential
	approximation of $\Omega$ at $\oz$. Nevertheless, the normal cone
	\eqref{lnc}, as well as the coderivative and subdifferential
	constructions for mappings and functions generated by it and
	described below, enjoy comprehensive {\em calculus rules} that are
	based on {\em variational/extremal principles} of variational
	analysis.\vspace*{0.03in}
	
	Given further a set-valued mapping $F\colon\R^n\tto\R^m$ with the
	\textit{graph} \begin{equation*} \gph
		F:=\big\{(x,y)\in\mathbb{R}^n\times\mathbb{R}^m\;\big|\;y\in
		F(x)\big\}, \end{equation*} the {\em graphical derivative} of $F$ at
	$(\ox,\oy)\in\gph F$ is defined via \eqref{tan} by
	\begin{equation}\label{gra-der}
		DF(\ox,\oy)(u):=\big\{v\in\R^m\;\big|\;(u,v)\in T_{{\rm
				gph}\,F}(\ox,\oy)\big\},\quad u\in\R^n. \end{equation}
	Consider also the \textit{domain, kernel,} and
	\textit{range} of $F$ denoted, respectively, by \begin{equation*}
		\dom F:=\big\{x \in \R^n\;\big|\;F(x) \ne \emp\big\}, \quad  \ker
		F:=\big\{x\in \R^n\;|\; 0 \in F(x)\big\}, \end{equation*}
	\begin{equation*} \rge F:=\big\{y \in \R^m\;\big|\;\exists x \in
		\R^n\quad \text{with }\; y\in F(x)\big\}. \end{equation*}
	The inverse mapping of $F$ is the set-valued mapping $F^{-1}:\R^m \rightrightarrows \R^n$ given by
	$$
	F^{-1}(y):= \{x \in \R^n|\; y \in F(x)\}, \quad y \in \R^m. 
	$$
	The coderivative constructions for $F$ at $(\ox,\oy)\in\gph F$ are
	defined via the regular normal cone \eqref{rnc} and the limiting
	normal cone \eqref{lnc} to the graph of $F$ at this point. They are,
	respectively, the {\em regular coderivative} and the {\em limiting
		coderivative} of $F$ at $(\ox,\oy)$ given by
	\begin{equation}\label{reg-cod} \Hat
		D^*F(\ox,\oy)(v):=\big\{u\in\R^n\;\big|\;(u,-v)\in\Hat N_{{\rm
				gph}\,F}(\ox,\oy)\big\},\quad v\in\R^m, \end{equation}
	\begin{equation}\label{lim-cod}
		D^*F(\ox,\oy)(v):=\big\{u\in\R^n\;\big|\;(u,-v)\in N_{{\rm
				gph}\,F}(\ox,\oy)\big\},\quad v\in\R^m. \end{equation} In the case
	where $F(\bar{x})$ is the singleton $\{\bar{y}\}$, we omit $\oy$ in
	the notation of \eqref{gra-der}--\eqref{lim-cod}. Note that if
	$F\colon\R^n\to\R^m$ is ${\cal C}^1$-smooth around $\ox$, then
	\begin{equation*} DF(\bar{x})=\nabla F(\bar{x})\;\mbox{ and
		}\;\widehat{D}^*F(\bar{x})=D^*F(\bar{x})=\nabla F(\bar{x})^*,
	\end{equation*} where $\nabla F(\bar{x})^*$ is the adjoint/transpose
	matrix of the Jacobian $\nabla F(\bar{x})$.\vspace*{0.05in}
	
	Before considering the {first- and second-order}
	subdifferential constructions for extended-real-valued functions,
	{which are employed} in this paper and are closely
	related to the limiting normals and coderivatives, we formulate the
	{\em metric regularity} and {\em subregularity} properties of
	set-valued mappings that are highly recognized in variational
	analysis and optimization. These properties are
	{frequently used} below.\vspace*{0.03in}
	
	To proceed, recall that the {\em distance function} associated with
	a set $\Omega\subset\R^s$ is \begin{equation*} {\rm
			dist}(x;\Omega):=\inf\big\{\|w-x\|\;\big|\;w\in\Omega\big\},\quad
		x\in\mathbb{R}^s. \end{equation*} A mapping
	$\widehat{F}\colon\R^n\tto\R^m$ is a {\em localization} of
	$F\colon\mathbb{R}^n\tto\mathbb{R}^m$ at $\bar{x}$ for $\bar{y}\in
	F(\bar{x})$ if there exist neighborhoods $U$ of $\bar{x}$ and $V$ of
	$\bar{y}$ such that we have \begin{equation*} \gph\widehat{F}=\gph
		F\cap(U\times V). \end{equation*}
	
	\begin{Definition}[\bf metric regularity and subregularity of mappings]\label{met-reg} Let $F\colon\mathbb{R}^n\tto\mathbb{R}^m$ be a set-valued mapping, and let $(\bar{x},\bar{y})\in\gph F$. We say that:\\[1ex]
		{\bf(i)} $F$ is {\sc metrically regular} around $(\bar{x},\bar{y})$
		with modulus $\mu>0$ if there exist neighborhoods $U$ of $\bar{x}$
		and $V$ of $\bar{y}$ providing the estimate \begin{equation*} {\rm
				dist}\big(x;F^{-1}(y)\big)\le\mu\,{\rm dist}\big(y;F(x)\big)\;\text{
				for all }\;(x,y)\in U\times V. \end{equation*}
		If in addition $F^{-1}$ has a single-valued localization around $(\bar{y},\bar{x})$, then $F$ is {\sc strongly metrically regular} around $(\bar{x},\bar{y})$ with modulus $\mu>0$.\\[1ex]
		{\bf(ii)} $F$ is {\sc metrically subregular} at $(\bar{x},\bar{y})$
		with modulus $\mu>0$ if there exists a neighborhood $U$ of $\bar{x}$
		such that we have \begin{equation*} {\rm
				dist}\big(x;F^{-1}(\bar{y})\big)\le\mu\,{\rm
				dist}\big(\bar{y};F(x)\big)\;\text{ for all }\;x\in U.
		\end{equation*} If $F^{-1}(\bar{y})\cap U=\{\bar{x}\}$, then
		$F$ is {\sc strongly metrically subregular} at $(\bar{x},\bar{y})$
		with modulus $\mu>0$. \end{Definition}
	
	{Observe that the metric regularity properties in
		Definition \ref{met-reg}(i) are \textit{stable/robust} with respect
		to small perturbations of the reference point $(\ox,\oy)$, i.e.,
		they  are preserved in a neighborhood of the point $(\ox,\oy)$.
		It is not always the case for metric subregularity.} The next remark
	summarizes relationships between the above metric
	regularity/subregularity properties of mappings and
	presents {\em generalized differential characterizations} of the
	major ones that are broadly used in this paper.
	
	\begin{Remark}[\bf on metric regularity and subregularity]\label{Morcri} {\rm Observe the following:\\[1ex]
			{\bf(i)} We obviously have that the strong metric regularity of a
			set-valued mapping implies that both its metric regularity and
			strong metric subregularity properties hold. Furthermore, a strongly
			metrically subregular mapping is metrically subregular at the corresponding point. However, metric regularity and strong metric subregularity are generally {\em incomparable}. For example, the mapping $F(x):=|x|$ for all $x\in\mathbb{R}$ is strongly metrically subregular at $(0,0)$, but it is not metrically regular around this point. On the other hand, $F\colon\mathbb{R}\tto\mathbb{R}$ defined by $F(x):=[x,\infty)$ is metrically regular around $(0,0)$ while not being strongly metrically subregular at this point.\\[1ex]
			{\bf(ii)} Simple examples show that a mapping may exhibit the strong
			metric subregularity property at some point while not being strongly
			metrically regular and even merely metrically regular around the
			reference point. Indeed, consider the simplest nonsmooth function
			$F(x):=|x|$ with $(\ox,\oy)=(0,0)$ discussed in (i).
			{The relationships between metric subregularity,
				metric regularity, strong metric subregularity, and strong metric
				regularity are illustrated in the following diagram, where the arrow
				$(\longrightarrow)$ reads ``implies"}. {
				\begin{center} \begin{tikzpicture}[line cap=round,line
						join=round,>=triangle 45,x=1.0cm,y=1.0cm]
						\clip(-2.33889817617,-0.380294618771) rectangle
						(11.787179127,4.02436778191); \draw (3.388467508,3.76778550614)
						node[anchor=north west] {$\text{metric regularity}$}; \draw
						(2.91831062348,0.731561909556) node[anchor=north west]
						{$\text{strong metric subregularity}$}; \draw
						(-2.06107819896,2.22829185154) node[anchor=north west]
						{$\text{strong metric regularity}$}; \draw
						(8.303744028,2.24967370785) node[anchor=north west] {$\text{metric
								subregularity}$}; \draw [->] (5.26909504609,3.16909352935) --
						(9.26542856452,2.35658298942); \draw [->]
						(0.118740083826,1.60821801843) -- (4.15781513722,0.731561909556);
						\draw [->] (5.37594888348,0.795707478498) --
						(9.35091163443,1.62959987474); \draw [->]
						(0.118740083826,2.27105556416) -- (4.09370283478,3.10494796041);
			\end{tikzpicture} \end{center}} {\bf(iii)} A major advantage of the
			generalized differential constructions defined above is the
			possibility to get in their terms complete {\em pointwise
				characterizations} of the metric regularity and strong metric
			subregularity properties of general set-valued mappings. Namely, a
			mapping $F\colon\R^n\tto\R^m$, which graph is locally closed around
			$(\ox,\oy)\in\gph F$, is {\em metrically regular} around this point
			{\em if and only if} we have the kernel coderivative condition
			\begin{equation}\label{cod-cr} \ker
				D^*F(\ox,\oy):=\big\{v\in\R^m\;\big|\;0\in
				D^*F(\ox,\oy)(v)\big\}=\{0\} \end{equation} established by
			Mordukhovich \cite[Theorem~3.6]{Mordu93} via his limiting
			coderivative \eqref{lim-cod} and then labeled as the {\em
				Mordukhovich criterion} in Rockafellar and Wets \cite[Theorems~7.40
			and 7.43]{Rockafellar98}. Broad applications of this result are
			based on {\em robustness} and {\em full calculus} available for the
			limiting coderivative; see the books by Mordukhovich
			\cite{Mordukhovich06,Mor18} and by Rockafellar and Wets
			\cite{Rockafellar98} for more details and references.
			
			A parallel characterization of the (nonrobust) {\em strong metric
				subregularity} property of a (locally) closed-graph mapping
			$F\colon\R^n\tto\R^m$ at $(\ox,\oy)\in\gph F$ is given by $\ker
			DF(\ox,\oy)=\{0\}$ via the (nonrobust) graphical derivative
			\eqref{gra-der} of $F$ at $(\ox,\oy)$ and is known as the {\em Levy-Rockafellar criterion}; see the book by Dontchev and Rockafellar
			\cite[Theorem~4E.1]{Donchev09} with the references and discussions
			therein.} \end{Remark}\vspace*{-0.05in}
	
	Finally in this section, we recall the limiting first-order and
	second-order subdifferential constructions for extended-real-valued
	functions  that are used for describing the subgradient inclusions
	\eqref{subgra-inc} and the Newton-type algorithms \eqref{newtonC11}
	and \eqref{subgra-prox} to compute their solutions. Given an
	extended-real-valued function $\ph\colon\R^n\to\oR$, consider its
	{\em effective domain} and {\em epigraph} defined by, respectively,
	\begin{equation*}
		\dom\ph:=\big\{x\in\R^n\;\big|\;\ph(x)<\infty\big\}\;\mbox{ and
		}\;\epi\ph:=\big\{(x,\al)\in\R^{n+1}\;\big|\;\al\ge\ph(x)\big\}.
	\end{equation*} Then for a fixed point $\ox\in\dom\ph$ we define the
	{\em basic/limiting subdifferential} and  {the {\em
			singular/horizon subdifferential}} of $\ph$ at $\ox$ by,
	respectively, \begin{equation}\label{lim-sub}
		\partial\varphi(\ox):=\big\{v\in\mathbb{R}^n\;\big|\;(v,-1)\in
		N_{{\rm\small
				epi}\,\varphi}\big(\bar{x},\varphi(\bar{x})\big)\big\},
	\end{equation} \begin{equation}\label{sin-sub}
		\partial^\infty\varphi(\ox):=\big\{v\in\mathbb{R}^n\;\big|\;(v,0)\in
		N_{{\rm\small epi}\,\varphi}\big(\bar{x},\varphi(\bar{x})\big)\big\}
	\end{equation} via the limiting normal cone \eqref{lnc} to the
	epigraph of $\ph$ at $(\ox,\ph(\ox))$. For simplicity we use here
	the geometric definitions of the subdifferentials \eqref{lim-sub}
	and \eqref{sin-sub} while referring the reader to the aforementioned
	monographs on variational analysis for equivalent analytic
	representations. Recall that the basic subdifferential
	$\partial\varphi(\ox)$ reduces to the gradient $\{\nabla\ph(\ox)\}$
	if $\ph$ is ${\cal C}^1$-smooth around $\ox$ (or merely strictly
	differentiable at this point), and that $\partial\ph(\ox)$ is the
	subdifferential of convex analysis if $\ph$ is convex. On the other
	hand, the singular subdifferential $\partial^\infty\varphi(\ox)$ of
	a lower semicontinuous (l.s.c.) function $\ph$ reduces to $\{0\}$
	{\em if and only if} $\ph$ is locally Lipschitzian around $\ox$.
	Both constructions \eqref{lim-sub} and \eqref{sin-sub} enjoy in
	parallel {\em full subgradient calculi} in very general settings.
	Let us also mention the {\em scalarization formula}
	\begin{equation}\label{scal} D^*F(\bar{x})(v)=\partial\langle
		v,F\rangle(\bar{x})\;\mbox{ for all }\;v\in\R^m, \end{equation}
	which holds whenever $F\colon\R^n\to\R^m$ is locally Lipschitzian
	around $\ox$.\vspace*{0.03in}
	
	Now we are ready to define the {\em second-order subdifferential} of
	$\ph\colon\R^n\to\oR$ at $\ox\in\dom\ph$ for
	$\ov\in\partial\ph(\ox)$ in the sense of Mordukhovich \cite{m92} as
	the mapping $\partial^2\ph(\ox,\ov)\colon\R^n\tto\R^n$ such that
	\begin{equation}\label{2nd}
		\partial^2\ph(\ox,\ov)(u):=\big(D^*\partial\ph\big)(\ox,\ov)(u)\;\mbox{
			for all }\;u\in\R^n, \end{equation} i.e., by applying the
	coderivative \eqref{lim-cod} to the first-order subgradient mapping
	\eqref{lim-sub}. This second-order subdifferential  (or {\em
		generalized Hessian}) \eqref{2nd} appears in the Newton-type
	iterations \eqref{newtonC11} and \eqref{subgra-prox} for ${\cal
		C}^{1,1}$ and prox-regular functions, respectively, which both go
	back to the classical Newton algorithm \eqref{clas-newton} for
	${\cal C}^2$-smooth functions due to the relationship
	\begin{equation}\label{C2_Case}
		\partial^2\ph(\ox)(u)=\big\{\nabla^2\ph(\ox)u\big\}\;\mbox{ whenever
		}\;u\in\R^n \end{equation} in the ${\cal C}^2$-smooth case. If $\ph$
	is of class ${\cal C}^{1,1}$ around $\ox$, then the computation of
	$\partial^2\ph(\ox)$ reduces to the computation of the limiting
	subdifferential \eqref{lim-sub} of the gradient mapping $\nabla\ph$
	by the scalarization formula \eqref{scal}. Besides the
	well-developed second-order calculus for \eqref{2nd}, variational
	analysis achieves {\em constructive computations} of the
	second-order subdifferential, entirely in terms of the given problem
	data, for major classes if nonsmooth functions arising in important
	problems of constrained optimization, bilevel programming, optimal
	control, operations research, mechanics, economics, statistics,
	machine learning, etc. Among many other publications, we refer the
	reader to Colombo et al. \cite{chhm}, Ding et al. \cite{dsy},
	Dontchev and Rockafellar \cite{dr}, Henrion et al. \cite{hmn,hos},
	Mordukhovich \cite{Mordukhovich06,Mor18}, Mordukhovich and Outrata
	\cite{BorisOutrata}, Mordukhovich and Rockafellar \cite{mr}, Outrata
	and Sun \cite{os}, Yao and Yen \cite{yy}, and the bibliographies
	therein. The new computation of this type is provided in
	Section~\ref{lasso} for the practically important  {\em
		regularized least square problems}   arising in
	statistics and machine learning.\vspace*{-0.1in}

\section{Solvability of Coderivative
Inclusions}\label{sec:solv}\vspace*{-0.05in}\setcounter{equation}{0}

A crucial step in the design and justification of numerical
	algorithms is to establish their {\em well-posedness}, i.e., the
	{\em solvability} of the corresponding iterative systems. In the
	case of the classical Newton method to solve $\nabla\ph(x)=0$, we
	have the equation for $d\in\R^n$ written as
	\begin{equation}\label{newton-eq}
		-\nabla\varphi(x)=\nabla^2\varphi(x)d, \end{equation} which is
	solvable if the Hessian matrix $\nabla^2\varphi(x)$ is invertible.
	In the case of the generalized Newton algorithm for nonsmooth
	functions discussed in Section~\ref{intro}, we extend
	\eqref{newton-eq} in the following way. Given a subgradient
	$v\in\partial\varphi(x)$, consider the inclusion
	\begin{equation}\label{newton-inc} -v\in\partial^2\varphi(x,v)(d)
	\end{equation} and find conditions ensuring the solvability of
	\eqref{newton-inc} with respect to $d\in\R^n$. Due to \eqref{2nd},
	the second-order inclusion \eqref{newton-inc} can be written as
	\begin{equation}\label{cod-inc} -v\in D^*F(x,v)(d) \end{equation}
	with $F:=\partial\varphi$, where $D^*$ stands for the limiting
	coderivative \eqref{lim-cod}.\vspace*{0.05in}
	
	The major goal of this section is to investigate the {\em
		solvability} of the {\em coderivative inclusion} \eqref{cod-inc}
	with respect to $d\in\R^n$. In the next section, we proceed with the
	study of solvability of the generalized Newton systems
	\eqref{newton-inc} and establish appropriate conditions on functions
	$\ph$, which allow us to efficiently apply the solvability results
	obtained for \eqref{cod-inc} to the case of systems
	\eqref{newton-inc} of our main interest.\vspace*{0.03in}
	
	The first theorem here verifies the solvability of \eqref{cod-inc}
	for $d$ at any point $(\ox,\ov)\in\gph F$ where the mapping $F$ is
	{\em strongly metrically subregular}. The proof of this result is
	based on major calculus rules for the limiting generalized
	differential constructions.\vspace*{-0.05in}
	
	\begin{Theorem}[\bf solvability of coderivative
		inclusions]\label{mainre} Let $F\colon\R^n\tto\R^n$ be a set-valued
		mapping  {whose graph} is locally closed around a given point
		$(\ox,\ov)\in\gph F$. If $F$ is strongly metrically subregular at
		$(\ox,\ov)$, then there exists $d\in\R^n$ satisfying the inclusion
		\begin{equation}\label{cod-inc1} -\ov\in D^*F(\ox,\ov)(d).
		\end{equation} 
	\end{Theorem} 
	\begin{proof} It follows from
		Theorem~3I.3 in Dontchev and Rockafellar \cite{Donchev09} that the
		assumed strong metric subregularity of $F$ around $(\ox,\ov)$ is
		equivalent to the existence of neighborhoods $U$ of $\ox$ and $V$ of
		$\ov$ together with a constant $\ell>0$ such that $F^{-1}(\ov)\cap
		U=\{\ox\}$ and \begin{equation}\label{isolate}
			\|x-\ox\|\le\ell\|v-\ov\|\;\mbox{ for all }\;x\in F^{-1}(v)\cap
			U\;\mbox{ and }\;v\in V. \end{equation} Passing to appropriate
		subsets, we suppose for convenience that the sets $U$ and $V$ are
		closed and bounded. Due to the local closedness of
		$F$ around $(\ox,\ov)$, we can assume further without loss of
		generality that the set $\gph G\cap (U\times V)$ is closed.
		Consider further the set-valued mapping $G\colon\R^n\tto\R^n$ given
		by \begin{equation*} G(v):=F^{-1}(v)\cap U\;\mbox{ for all
			}\;v\in\R^n \end{equation*} and then define the {\em
			marginal/optimal value function} $\mu\colon\R^n\to\oR$ and the
		corresponding argminimum mapping $M\colon\R^n\tto\R^n$ by,
		respectively, \begin{equation}\label{marg}
			\mu(v):=\inf\big\{\varphi(v,x)\;\big|\;x\in G(v)\big\}\;\mbox{ and
			}\;M(v):= \big\{x\in G(v)\;\big|\;\varphi(v,x)=\mu(v)\big\},\quad
			v\in\R^n, \end{equation} where
		$\varphi(v,x):=\langle\bar{v},x\rangle$ for all
		$(v,x)\in\mathbb{R}^n\times\mathbb{R}^n$. It is clear that
		$(\bar{v},\bar{x})\in\gph M$. We intend to verify
		that the marginal function \eqref{marg} is l.s.c.\ around $\ov$. To
		proceed, pick any $v \in \text{\rm int}V$ and check first that for $\epsilon>0$
		there exists $\delta >0$ ensuring the implication
		\begin{equation}\label{incG} \|u-v\|\leq \delta \Longrightarrow G(u)
			\subset G(v) + \epsilon \mathbb{B}
		\end{equation} Suppose on the contrary that \eqref{incG}
		does not hold and then find $\epsilon>0$ and a sequence $\{u_k\}$
		converging to $v$ such that $G(u_k) \not \subset G(v)+ \epsilon\mathbb{B}$. This means
		that for each $k \in \N$ there exists $x_k \in G(u_k)$ such that
		$x_k \notin G(v) + \epsilon\mathbb{B}$. Since we
		have $\{x_k\}\subset U$ where $U$ is compact, there exists a
		subsequence $\{x_{k_j}\}$ of $\{x_k\}$ converging to some $x^*$ as
		$j \to \infty$. It follows from the closedness of $\gph
		G\cap(U\times V)$ and the choice of $x_{k_j}\in G(u_{k_j})$ and
		$u_{k_j} \to v$ with $x_{k_j}\to x^*$ that $x^*\in G(v)$, which
		implies that $x_{k_j} \notin \mathbb{B}_\epsilon(x^*)$ for any $j
		\in\N$. This is a clear contradiction that verifies \eqref{incG}.
		Let us now show that for any $\epsilon>0$ there exists $\delta>0$
		such that \begin{equation}\label{lscmu} \mu(u) \geq  \mu(v)
			-\epsilon \quad \text{whenever }\; \|u - v\|\leq \delta.
		\end{equation} Indeed, using \eqref{incG} gives us $\delta>0$
		providing the implication \begin{equation}\label{incG2} \|u-v\|\le
			\delta \Longrightarrow G(u) \subset
			G(v) + \frac{\epsilon}{\|\ov\|+1}\mathbb{B}. \end{equation} 
		Pick $u \in \mathbb{B}_\delta(v)$. By \eqref{incG2}, for any $x'\in G(u)$, there is $x \in G(v)$ with $x'-x \in \epsilon/(\|\ov\|+1)\mathbb{B}$, and so
		$$ \langle \ov, x' \rangle - \langle \ov,x\rangle = \langle \ov, x'
		-x \rangle \geq -\|\ov\|\cdot\|x'-x\| \geq -\epsilon, $$ which
		implies in turn that
		$$
		\langle \ov , x' \rangle \geq \mu(v) - \epsilon \quad \text{for all }\; x' \in G(u).
		$$
		Therefore, we arrive at the conditions
		$$ \mu(u) = \inf_{x'\in G(u)} \langle \ov,
		x'\rangle  \geq \mu(v)
		-\epsilon, $$ and hence justifies \eqref{lscmu}, which means that the
		marginal function $\mu$ is l.s.c.\ around $\ov$.
		
		To show that the mapping $M$ is locally bounded around $\ov$,
		observe that for each $v\in V$ we have \begin{equation*} M(v)\subset
			G(v)= F^{-1}(v)\cap U\subset U, \end{equation*} which tells us that
		the image set $M(V)$ is bounded, i.e., the mapping $M$ is locally
		bounded around $\ov$. To evaluate now the limiting subdifferential
		\eqref{lim-sub} of the marginal function \eqref{marg} by
		Theorem~4.1(ii) from Mordukhovich \cite{Mor18}, it remains to check
		the qualification condition \begin{equation*}
			\partial^\infty\varphi(\bar{v},\bar{x})\cap\big(-N_{\text{gph}\,G}(\bar{v},\bar{x})\big)=\{0\}
		\end{equation*} therein, which automatically holds due to the
		Lipschitz continuity of the function $\ph$ around $(\ox,\ov)$. Since
		we have in \eqref{marg} that $M(\bar{v})=\{\bar{x}\}$,
		$\nabla_v\varphi(\bar{v},\bar{x})=0$, and
		$\nabla_x\varphi(\bar{v},\bar{x})=\bar{v}$ in \eqref{marg}, it follows that
		\begin{equation}\label{inc}
			\partial\mu(\bar{v})\subset\nabla_v\varphi(\bar{v},\bar{x})+D^*
			G(\bar{v},\bar{x})\big(\nabla_x\varphi(\bar{v},\bar{x})\big)=D^*G(\bar{v},\bar{x})(\bar{v}).
		\end{equation}
		
		To proceed further, consider the mapping $H(v)\equiv U$ on $\R^n$
		and observe that $N_{\text{gph}\,H}(\bar{v},\bar{x})=\{0\}$ by
		$(\bar{v},\bar{x})\in\text{int}(\text{gph}\,H)$. Thus
		$N_{\text{gph}H}(\bar{v},\bar{x})=\{0\}$ and the qualification
		condition \begin{equation*}
			N_{\text{gph}\,H}(\bar{v},\bar{x})\cap\big(-N_{\text{gph}\,F^{-1}}(\bar{v},\bar{x})\big)=\{0\}
		\end{equation*} is satisfied. This allows us to apply the normal
		cone intersection rule from Mordukhovich \cite[Theorem~2.16]{Mor18}
		and get the relationships \begin{equation*}
			N_{\text{gph}\,G}(\bar{v},\bar{x})\subset
			N_{\text{gph}\,F^{-1}}(\bar{v},\bar{x})+N_{\text{gph}\,H}(\bar{v},\bar{x})=N_{\text{gph}\,F^{-1}}(\bar{v},\bar{x}).
		\end{equation*} Combining the latter with \eqref{inc} gives us the
		inclusions \begin{equation}\label{mu-inc}
			\partial\mu(\bar{v})\subset D^*G(\bar{v},\bar{x})(\bar{v})\subset
			D^*F^{-1}(\bar{v},\bar{x})(\bar{v}). \end{equation}
		
		Let us now show that $\partial\mu(\ov)\ne\emp$. Recall from
		Mohammadi et al. \cite[Proposition~2.1]{mms} that if a function
		$\psi\colon\R^n\to\oR$ is l.s.c.\ around $\ov\in\dom\psi$ and
		satisfies the ``lower calmness" property
		\begin{equation}\label{calm}
			\psi(v)\ge\psi(\ov)-\eta\|v-\ov\|\;\mbox{ for all }\;v\in V
		\end{equation} with some $\eta\ge 0$ and a neighborhood $V$ of
		$\ov$, then $\partial\psi(\ov)\ne\emp$. To establish \eqref{calm}
		for $\psi:=\mu$, we verify in what follows the fulfillment of the
		estimate \begin{equation}\label{calmness}
			\mu(v)-\mu(\bar{v})\ge-\ell\|\bar{v}\|\cdot\|v-\bar{v}\|\;\mbox{
				whenever }\;v\in V. \end{equation} Observe first that if $G(v)=\emp$
		for some $v\in V$, then $\mu(v)=\infty$, and so estimate
		\eqref{calmness} is obviously satisfied. In the remaining case where
		$G(v)\ne\emp$ for a fixed vector $v\in V$, pick any $x\in G(v)$ and
		then check by \eqref{isolate} that \begin{equation*}
			\langle\bar{v},x\rangle-\langle\bar{v},\bar{x}\rangle\ge-\|\bar{v}\|\cdot\|x-\bar{x}\|\ge-\ell\|\bar{v}\|\cdot\|v-\bar{v}\|. \end{equation*} Indeed, the condition
		$M(\bar{v})=\{\bar{x}\}$ clearly yields \eqref{calmness}, and hence
		$\partial\mu(\ov)\ne\emp$. It justifies by \eqref{mu-inc} the
		existence of $u\in\R^n$ satisfying $u\in D^*F^{-1}
		(\bar{v},\bar{x})(\bar{v})$. To complete the proof of the theorem,
		recall that \begin{equation}\label{inverse} z\in
			D^*F(\bar{x},\bar{y})(u)\Longleftrightarrow-u\in\big(D^*F^{-1}\big)(\bar{y},\bar{x})(-z),
		\end{equation} which readily ensures the fulfillment of
		\eqref{cod-inc1} and thus finishes the proof. \end{proof}
	
	The following example shows that the strong metric subregularity
	assumption of Theorem~\ref{mainre} cannot be replaced by metric
	subregularity of $F$ at $(\ox,\ov)$, or even by metric regularity of
	$F$ around this point in order to guarantee the solvability of the
	coderivative inclusion \eqref{cod-inc1}.\vspace*{-0.05in}
	
	\begin{Example}[\bf insolvability of coderivative inclusions under
		metric regularity]\label{insolv}  {\rm  Consider the set-valued
			mapping $F\colon\R\tto\R$ defined by \begin{equation}\label{ex1}
				F(x):=[0,1]\;\mbox{ for all }\;x\in\mathbb{R}. \end{equation} Then
			the graphical set $\gph F=\R\times[0,1]$ is convex, and we easily
			calculate the limiting normal cone to the graph of \eqref{ex1} as
			follows: \begin{equation*} N_{\text{gph}\,F}(x,y)=\begin{cases}
					\{(0,0)\}&\text{if}\qquad y\in(0,1),\\
					\{0\}\times\mathbb{R}_{+}&\text{if}\qquad y=1,\\
					\{0\}\times\mathbb{R}_{-}&\text{if}\qquad y=0.\\
			\end{cases} \end{equation*} Pick
			$(\bar{x},\bar{v}):=\left(1,\frac{1}{2}\right)\in\gph F$ and show
			that $F$ is metrically regular around $(\bar{x},\bar{v})$. Indeed,
			taking $u\in\R^n$ with $0\in D^*F(\bar{x},\bar{v})(u)$, we readily
			have that \begin{equation*} (0,-u)\in
				N_{\text{gph}\,F}(\bar{x},\bar{v})=\{(0,0)\}, \end{equation*} and
			hence $u=0$. It follows from the Mordukhovich criterion
			\eqref{cod-cr} that $F$ is metrically regular around
			$(\bar{x},\bar{v})$. However, it is easy to see that there exists no
			$d\in\R$ which solves \eqref{cod-inc1}.} \end{Example}
	
	The next example shows that the strong metric subregularity, being a
	nonrobust property, does not ensure the {\em robust solvability} of
	the coderivative inclusion \eqref{cod-inc}, i.e., its solvability in
	a neighborhood of the reference point. Given $(x,v)\in\gph F$,
	consider the set of feasible solutions of to inclusion
	\eqref{cod-inc} at this point that is defined by
	\begin{equation}\label{newton-dir}
		\Gamma_F(x,v):=\big\{d\in\R^n\;\big|\;-v\in D^*F(x,v)(d)\big\}.
	\end{equation}\vspace*{-0.15in}
	
	\begin{Example}[\bf failure of robust solvability under strong
		metric subregularity]\label{ex-nonrob} {\rm Consider the set-valued
			mapping $F\colon\R\tto\R$ defined by \begin{equation}\label{ex2}
				F(x):=\begin{cases}
					\{0\}\cup[1,\infty)&\text{if}\quad x=0,\\
					[1,\infty)&\text{if}\quad x\ne 0. \end{cases} \end{equation} It is
			clear that $F$ is strongly metrically subregular at $(0,0)$ and that
			the graph of $F$ is closed around this point. Let us show that for
			any neighborhood $U\times V$ of the origin in $\R^2$ and any nonzero
			pair $(x,v)\in\gph F\cap(U\times V)$ there exists no $d\in\R$
			satisfying the coderivative inclusion \eqref{cod-inc}. To proceed,
			observe that  {$\gph F=(0,0)\cup\R\times[1,\infty)$}
			for $F$ from \eqref{ex2}. Then we readily get that
			\begin{equation*} N_{\text{gph}\,F}(x,v)=\begin{cases}
					\R^2&\text{if}\quad x =0,\;v =0,\\
					\big\{(0,\lambda)\;\big|\;\lambda\le 0\big\}&\text{if}\quad v\ge 1,
			\end{cases} \end{equation*} which yields the limiting coderivative
			expression \begin{equation*} D^*F(x,v)(u)=\begin{cases}
					\R&\text{if}\quad (x,v)=(0,0),\\
					\{0\}&\text{if}\quad(x,v)\ne(0,0),\;u\ge 0,\\
					\emp&\text{otherwise}. \end{cases} \end{equation*} It easily follows
			from this formula that there exists no $d\in\R$ satisfying
			\eqref{cod-inc} for any $(x,v)\in\gph F$ except $(x,v)=(0,0)$. Note
			also that the set $\Gamma_F(0,0)=\mathbb{R}$ from \eqref{newton-dir}
			is not bounded.} \end{Example}
	
	Now we are ready to show that the replacement of the strong metric
	subregularity of $F$ at $(\ox,\ov)$ in Theorem~\ref{mainre} by its
	{\em robust} counterpart, which is the {\em strong metric
		regularity} of $F$ around $(\ox,\ov)$, leads us to robust
	solvability of the coderivative inclusion \eqref{cod-inc} and thus
	ensures the well-posedness of generalized Newton iterations in the
	algorithms designed below in this paper.\vspace*{-0.05in}
	
	\begin{Theorem}[\bf robust solvability of coderivative
		inclusions]\label{strongsol} Let $F\colon\R^n\tto\R^n$ be a
		set-valued mapping {whose graph} is closed around
		$(\ox,\ov)\in\gph F$. If $F$ is strongly metrically regular around
		this point, then there is a neighborhood $U\times V$ of $(\ox,\ov)$
		such that for each $(x,v)\in\gph F\cap(U\times V)$ there exists a
		direction $d\in\R^n$ satisfying the coderivative inclusion
		\eqref{cod-inc}. Moreover, the set-valued mapping $\Gamma_F$ from
		\eqref{newton-dir} is compact-valued for all $(x,v)\in\gph
		F\cap(U\times V)$. 
	\end{Theorem} 
	\begin{proof} Since $F$ is strongly
		metrically regular around $(\bar{x},\bar{v})$, it follows from
		Definition~\ref{met-reg}(i) that the inverse mapping
		$F^{-1}\colon\R^n\tto\R^n$ admits a single-valued localization
		$\vt\colon V\to U$ around $(\bar{v},\bar{x})$, which is locally
		Lipschitzian around $\bar{v}$. This implies, together with the
		scalarization formula \eqref{scal}, that for each $(x,v)\in\gph
		F\cap(U\times V)$ we have the representations
		\begin{equation}\label{scar1}
			D^*F^{-1}(v,x)(v)=D^*\vartheta(v)(v)=\partial\langle
			v,\vartheta\rangle(v). \end{equation} It is well known (see, e.g.,
		Mordukhovich \cite[Theorem~1.22]{Mor18}) that the limiting
		subgradient set of a locally Lipschitzian function is nonempty and
		compact. Hence it follows from \eqref{scar1} that the set
		$D^*F^{-1}(v,x)(-v)$ is nonempty and compact in $\R^n$. Taking any
		$u\in D^*F^{-1}(v,x)(-v)$, we deduce from \eqref{inverse} that
		$-v\in D^*F(x,v)(d)$ with $d:=-u$, which verifies the robust
		solvability of \eqref{cod-inc} around $(\ox,\ov)$. The claimed
		compactness of \eqref{newton-dir} follows from the compactness of
		$\partial\langle v,\vartheta\rangle(v)$. 
	\end{proof}
	
	Note that the strong metric regularity in Theorem~\ref{strongsol} is
	{\em just a sufficient condition} for robust solvability of the
	coderivative inclusion \eqref{cod-inc} and its second-order
	subdifferential specifications studied in the next section. As we see below, the required robust solvability is exhibited even in the case of coderivative inclusions \eqref{newton-iter} generated by subdifferentially continuous functions $\ph$ without the strong metric regularity assumption on $\partial\ph$.\vspace*{-0.1in}

\section{Solvability of Generalized Newton
Systems}\label{sec:solvN}\vspace*{-0.05in}\setcounter{equation}{0}

In this section we consider the second-order subdifferential
	inclusions \eqref{newton-inc} generated by extended-real-valued
	functions $\ph\colon\R^n\to\oR$. Such systems appear in our
	generalized Newton algorithms, which were discussed in
	Section~\ref{intro} and will be fully developed in what follows. As
	mentioned, the second-order subdifferential systems
	\eqref{newton-inc} are specifications of the coderivative ones
	\eqref{cod-inc} for $F:=\partial\ph$, while the subdifferential
	structure of $F$ creates strong opportunities to efficiently
	implement and improve the assumptions of Theorems~\ref{mainre} and
	\ref{strongsol} for important classes of functions $\ph$ that are
	often encountered in finite-dimensional variational analysis and
	optimization.
	
	There are two groups of assumptions in both Theorems~\ref{mainre}
	and \ref{strongsol}: one on the closed graph of $F$, and the other
	on the strong metric subregularity and regularity properties. Let us
	start with the first one: {\em when is the limiting subgradient
		mapping $\partial\ph$ of $($locally$)$ closed
		graph}?\vspace*{0.03in}
	
	It is well known and easily follows from definitions \eqref{lnc} and
	\eqref{lim-sub} that the limiting subgradient mapping
	$\partial\ph\colon\R^n\tto\R^n$ is closed-graph around
	$(\ox,\ov)\in\gph\partial\ph$ if $\ph$ is {\em continuous} around
	$\ox$. However, this important and broad setting does not encompass
	functions that are locally extended-real-valued around $\ox$, while
	such functions are the most interesting for applications, e.g., to
	constrained optimization. This is the reason for the following
	definition taken from Rockafellar and Wets
	\cite[Definition~13.28]{Rockafellar98}.\vspace*{-0.05in}
	
	\begin{Definition}[\bf subdifferentially continuous
		functions]\label{sub-cont} A function $\ph\colon\R^n\to\oR$ is {\sc
			subdifferentially continuous} at $\ox\in\dom\ph$ for
		$\ov\in\partial\ph(\ox)$ if for any $(x_k,v_k)\to(\ox,\ov)$ with
		$v_k\in\partial\ph(x_k)$ we have $\ph(x_k)\to\ph(\ox)$ as
		$k\to\infty$. If this holds for all $\ov\in\partial\ph(\ox)$, then
		$\ph$ is subdifferentially continuous at $\ox$.
	\end{Definition}\vspace*{-0.05in}
	
	Note that $\ph$ is obviously subdifferentially continuous at any
	point $\ox\in\dom\ph$ where $\ph$ is continuous merely relative to
	its domain. It easily follows from the subdifferential construction
	of convex analysis that any convex extended-real-valued function is
	subdifferentially continuous at every $\ox\in\dom\ph$. As has been
	well recognized in variational analysis and optimization, the class
	of subdifferentially continuous functions is much broader and
	includes, in particular, {\em strongly amenable} functions, {\em
		lower-${\cal C}^2$} functions, etc.; see Rockafellar and Wets
	\cite{Rockafellar98}.\vspace*{0.05in}
	
	The next theorem on solvability and robust solvability of
	generalized Newton systems is a direct consequence of
	Theorems~\ref{mainre}, \ref{strongsol} and
	Definition~\ref{sub-cont}.\vspace*{-0.05in}
	
	\begin{Theorem}[\bf solvability and robust solvability of generalized Newton systems]\label{solvability}
		
		Let $\ph\colon\R^n\to\oR$ be subdifferentially continuous around some $\ox\in\dom\ph$. Then the following hold:\\[1ex]
		{\bf(i)} Given $\ov\in\partial\ph(\ox)$, assume that the mapping $\partial\ph\colon\R^n\tto\R^n$ is strongly metrically subregular at $(\ox,\ov)$. Then there exists $d\in\R^n$ satisfying the second-order subdifferential inclusion \eqref{newton-inc} for $(x,v):=(\ox,\ov)$.\\[1ex]
		{\bf(ii)} Given $\ov\in\partial\ph(\ox)$, assume that the
		subgradient mapping $\partial\ph$ is strongly metrically regular
		around $(\ox,\ov)$. Then there is a neighborhood $U\times V$ of
		$(\ox,\ov)$ such that for each $(x,v)\in\gph\partial\ph\cap(U\times
		V)$ there exists a direction $d\in\R^n$ satisfying the second-order
		subdifferential inclusion \eqref{newton-inc}. Moreover, the
		set-valued mapping $\Gamma_{\partial\ph}$ from \eqref{newton-dir} is
		compact-valued for all $(x,v)\in\gph\partial\ph\cap(U\times V)$.
	\end{Theorem}\vspace*{-0.1in} 
	\begin{proof} It is easy to check that
		the imposed subdifferential continuity assumption on $\ph$ ensures
		that the graph of $\partial\ph$ is locally closed around
		$(\ox,\ov)\in\gph\partial\ph$. Then the claimed assertions (i) and
		(ii) follow from Theorem~\ref{mainre} and Theorem~\ref{strongsol},
		respectively. 
	\end{proof} 
	
	Observe that Example~\ref{insolv}, which demonstrates that the
	strong metric subregularity assumption on the mapping $F$ in
	\eqref{cod-inc} cannot be replaced by the metric regularity and
	hence by the metric subregularity ones in the conclusion of
	Theorem~\ref{mainre}, still works for Theorem~\ref{solvability}(i)
	dealing with mappings $F=\partial\ph$ of the subdifferential type.
	Indeed, it is shown by Wang \cite[Theorem~4.6]{wang} that there
	exists a Lipschitz continuous function $\varphi\colon\R\to\R$ such
	that $\partial\varphi(x)=[0,1]$ for all $x\in\R$. Thus we get from
	Example~\ref{insolv} that the subgradient mapping $\partial\varphi$
	for this function is metrically regular around the point
	$(0,1/2)\in\gph\partial\ph$, while the second-order subdifferential
	inclusion \eqref{newton-inc} is not solvable for $d$ at this
	point.\vspace*{0.05in}
	
	Let us further reveal the class of functions where the {\em metric
		regularity} of $\partial\ph$ can replace the strong metric
	subregularity assumption in the solvability result of
	Theorem~\ref{solvability}(i). This issue is very appealing since
	metric regularity is a robust property, which is fully
	characterized---via the Mordukhovich criterion \eqref{cod-cr}---by
	the robust limiting coderivative (and hence by the second-order
	subdifferential \eqref{2nd} for the subgradient systems
	\eqref{subgra-inc}) enjoying comprehensive calculus rules and
	computation formulas discussed above. To proceed in this direction,
	we significantly use the subdifferential structure of
	\eqref{subgra-inc} with $F=\partial\ph$ in \eqref{cod-inc}. The
	following class of functions was introduced by Poliquin and
	Rockafellar \cite{Poliquin}.\vspace*{-0.05in}
	
	\begin{Definition}[\bf prox-regular functions]\label{prox-reg} A
		function $\varphi\colon\R^n\to\oR$ is {\sc prox-regular} at a point
		$x\in\dom\varphi$ for a subgradient $\ov\in\partial\varphi(\ox)$
		{with modulus $r>0$} if $\ph$ is l.s.c.\ around
		$\ox$ and there exists $\epsilon>0$  such that for all
		$x,u\in\mathbb{B}_\epsilon(\bar{x})$ with
		$|\varphi(u)-\varphi(\bar{x})|<\epsilon$ we have
		\begin{equation}\label{prox} \varphi(x)\ge\varphi(u)+\langle
			v,x-u\rangle-\frac{r}{2}\|x-u\|^2\;\mbox{ whenever
			}\;v\in\partial\varphi(u)\cap\mathbb{B}_\epsilon(\bar{v}).
		\end{equation} If this holds for all $\ov\in\partial\ph(\ox)$, $\ph$
		is said to be prox-regular at $\ox\in\dom\ph$.
	\end{Definition}\vspace*{-0.05in}
	
	In what follows we say that $\ph$ is {\em continuously prox-regular}
	at $\ox$ for $\ov$ (and just at $\ox)$ if it is simultaneously
	prox-regular and subdifferentially continuous according to
	Definitions~\ref{sub-cont} and \ref{prox-reg}. It is easy to see
	that if $\ph$ is continuously prox-regular at $\ox$ for
	$\ov\in\partial\ph(\ox)$, then the condition
	$|\varphi(u)-\varphi(\bar{x})|<\epsilon$ the definition of
	prox-regularity can be omitted. Furthermore, in this case the graph
	of $\partial\ph$ is locally closed around $(\ox,\ov)$. As discussed
	in the book by Rockafellar and Wets \cite{Rockafellar98}, the class
	of continuously prox-regular functions is fairly broad containing,
	besides $\mathcal{C}^2$-smooth functions, also functions of class
	$\mathcal{C}^{1,1}$, convex l.s.c.\ functions, lower-$\mathcal{C}^2$
	functions, strongly amenable functions, etc. This class plays a
	central role in second-order variational analysis and its
	applications; see the books by Rockafellar and Wets
	\cite{Rockafellar98} and by Mordukhovich \cite{Mor18} with the
	commentaries and references therein.\vspace*{0.03in}
	
	To establish the desired solvability theorem for the second-order
	subdifferential inclusions \eqref{newton-inc} with continuously
	prox-regular functions $\ph$, we need to recall yet another notion
	of generalized second-order differentiability taken from Rockafellar
	and Wets \cite[Chapter~13]{Rockafellar98}.
	
	Given $\ph\colon\R^n\to\oR$ with $\ox\in\dom\ph$, consider the
	family of second-order finite differences \begin{equation*}
		\Delta^2_\tau\varphi(\bar{x},v)(u):=\frac{\varphi(\bar{x}+\tau
			u)-\varphi(\bar{x})-\tau\langle v,u\rangle}{\frac{1}{2}\tau^2}
	\end{equation*} and define the {\em second subderivative} of
	$\varphi$ at $\ox$ for $v\in\R^n$ and $w\in\R^n$ by
	\begin{equation*} d^2\varphi(\ox,v)(w):=\liminf_{\tau\downarrow
			0\atop u\to w}\Delta^2_\tau\varphi(\ox,v)(u). \end{equation*} Then
	$\ph$ is said to be {\em twice epi-differentiable} at $\ox$ for $v$
	if for every $w\in\R^n$ and every choice $\tau_k\downarrow 0$ there
	exists a sequence $w^k\to w$ such that \begin{equation*}
		\frac{\varphi(\ox+\tau_k w^k)-\varphi(\ox)-\tau_k\langle
			v,w^k\rangle}{\frac{1}{2}\tau_k^2}\to d^2\varphi(\ox,v)(w)\;\mbox{
			as }\;k\to\infty. \end{equation*}
	
	Twice epi-differentiability has been recognized as an important
	property in second-order variational analysis with numerous
	applications to optimization; see the aforementioned monograph by
	Rockafellar and Wets and the recent papers by Mohammadi et al.
	\cite{mms,mms1,ms}. In particular, the latter papers develop a
	systematic approach to verify epi-differentiability via {\em
		parabolic regularity}, which is a major second-order property of
	extended-real-valued functions that goes far beyond the class of
	{\em fully amenable} functions investigated in Rockafellar and Wets
	\cite{Rockafellar98}.\vspace*{0.05in}
	
	Now we are ready to establish solvability of the second-order
	subdifferential inclusion \eqref{newton-inc} under merely metric
	regularity of the limiting subgradient mappings $\partial\ph$ for
	the class of continuously prox-regular and twice epi-differentiable
	functions.\vspace*{-0.05in}
	
	\begin{Theorem}[\bf solvability of generalized Newton systems under metric regularity]\label{twiceepi} Let $\ph\colon\R^n\to\oR$ be  {continuously  prox-regular} at $\ox$ for some $\ov\in\partial\ph(\ox)$. Suppose in addition that the subgradient mapping $\partial\varphi$ is metrically regular around $(\ox,\ov)$ and that one of two following properties holds:\\[1ex]
		{\bf(i)} $\varphi$ is a univariate function, i.e., $n=1$.\\[1ex]
		{\bf(ii)} $\varphi$ is twice epi-differentiable at $\ox$ for $\ov$.\\[1ex]
		Then there exists $d\in\R^n$ satisfying the second-order
		subdifferential system \eqref{newton-inc} at $(\ox,\ov)$.
	\end{Theorem}
	\begin{proof} As mentioned above, the
		subdifferential graph $\partial\ph$ is locally closed around
		$(\ox,\ov)$. Let us show now that the imposed assumptions ensure
		that $\partial\ph$ is strongly metrically subregular at $(\ox,\ov)$.
		We are going to employ the Levy-Rockafellar criterion for the latter
		property telling us that it is equivalent to
		\begin{equation}\label{lev-rock}
			0\in\big(D\partial\ph\big)(\ox,\ov)(u)\Longrightarrow u=0
		\end{equation} as discussed in Remark~\ref{Morcri}(iii). To verify
		\eqref{lev-rock}, pick any $u\in\R^n$ such that
		$0\in(D\partial\varphi)(\bar{x},\bar{v})(u)$. Then the continuous
		prox-regularity of $\ph$ and each of the imposed assumptions (i) and
		(ii) ensure the following relationship between the graphical
		derivative and limiting coderivative of $\partial\ph$:
		\begin{equation}\label{der-cod}
			\big(D\partial\varphi\big)(\ox,\ov)(u)\subset\big(D^*\partial\varphi\big)(\ox,\ov)(u)\;\mbox{
				for all }\;u\in\R^n. \end{equation} In the univariate case (i)
		inclusion \eqref{der-cod} was proved by Rockafellar and Zagrodny
		\cite[Theorem~4.1]{rock-zag}, while the twice epi-differentiable
		case (ii) was done in the equivalent form in Theorem~1.1 of the
		latter paper; see also Rockafellar and Wets
		\cite[Theorem~13.57]{Rockafellar98}. Since the subdifferential
		mapping $\partial\ph$ is assumed to be metrically regular at
		$(\ox,\ov)$, we get by using \eqref{der-cod} and the Mordukhovich
		criterion \eqref{cod-cr} that \begin{equation*}
			0\in\big(D\partial\varphi\big)(\bar{x},\bar{v})(u)\Longrightarrow
			0\in\big(D^*\partial\varphi\big)(\bar{x},\bar{v})(u)\Longrightarrow
			u=0, \end{equation*} which ensures by \eqref{lev-rock} that
		$\partial\ph$ is strongly metrically subregular at $(\ox,\ov)$.
		Using finally the result of Theorem~\ref{solvability}(i), we arrive
		at the claimed solvability and thus complete the proof.
	\end{proof} 
	
	Note that Theorem~\ref{twiceepi} concerns solvability of the
	second-order subdifferential inclusion \eqref{newton-inc} {\em at}
	the chosen point $(\ox,\ov)\in\gph\partial\ph$. What about {\em
		robust} solvability of \eqref{newton-inc} {\em around} the reference
	point in the line of Theorem~\ref{solvability}(ii)? This is
	discussed in the following remark.\vspace*{-0.05in}
	
	\begin{Remark}[\bf robust solvability under metric
		regularity]\label{robust-newton} {\rm Theorem~\ref{solvability}(ii)
			tells us that the {\em strong} metric regularity of $\partial\ph$
			around $(\ox,\ov)$ ensures the robust solvability of
			\eqref{newton-inc} around this point. But it has been recognized
			that the strong metric regularity of subgradient mappings
			$\partial\ph$ is {\em equivalent} to merely the metric regularity of
			them for major subclasses of continuously prox-regular functions
			$\ph\colon\R^n\to\oR$ with the {\em conjecture} that it holds for
			the {\em entire class} of such functions at {\em local minimizers}
			of $\ph$; see Drusvyatskiy {et al.}
			\cite[Conjecture~4.7]{dmn}. This is largely discussed in the
			mentioned paper by Drusvyatskiy  {et al.}, and now
			we recall some results from that paper. Indeed, the equivalence
			clearly holds (and not only for local minimizers of $\ph$) for
			${\cal C}^2$-smooth functions and for l.s.c.\ convex functions due
			the fundamental Kenderov theorem on maximal monotone operators
			\cite{kender}. The claimed equivalence is also valid for a broad
			class of functions given by $\ph(x)=\ph_0(x)+\dd_\Omega(x)$, where
			$\ph_0$ is a ${\cal C}^2$-smooth function, and where $\dd_\Omega$ is the
			indicator function of a polyhedral convex set; see Dontchev and
			Rockafellar \cite{dr}. Yet another large setting of such an
			equivalence is revealed in Drusvyatskiy et al.
			\cite[Theorem~4.13]{dmn} for continuously prox-regular functions
			$\ph$ with $0\in\partial\ph(\ox)$ under the additional condition
			that the second-order subdifferential $\partial^2\ph(\ox,0)$ is {\em
				positive-semidefinite} in the sense that \begin{equation*} \la
				v,u\ra\ge 0\;\mbox{ for all }\;v\in\partial^2\ph(\ox,0)(u),\quad
				u\ne 0. \end{equation*} Note that the requirement that $\ox$ is a
			local minimizer of $\ph$ is {\em essential} for the validity of this
			conjecture even for twice epi-differentiable functions of class
			${\cal C}^{1,1}$ with piecewise linear and directionally
			differentiable gradients; see Example~\ref{ex-kummer} in the next
			section.} \end{Remark} \vspace*{-0.2in}

\section{Generalized Newton Method for ${\cal C}^{1,1}$ Gradient
Equations}\label{sec:newtonC11}\vspace*{-0.05in}\setcounter{equation}{0}

	In this section we propose and justify a generalized Newton
	algorithm to solve gradient systems of type \eqref{gra}, where
	$\ph\colon\R^n\to\R$ is a function of class ${\cal C}^{1,1}$ around
	a given point $\ox$. To begin with, let us formulate the {\em
		semismooth$^*$ property} of set-valued mappings
	$F\colon\R^n\tto\R^m$ introduced recently by Gfrerer and Outrata
	\cite{Helmut}. This property is used here for the justification of
	local superlinear convergence of our Newton-type algorithm to solve
	gradient equations \eqref{gra} and then to solve subgradient
	inclusions \eqref{subgra-inc} in the subsequent sections of the
	paper.
	
	To formulate the semismooth$^*$ property of set-valued mappings,
	recall first the notion of the {\em directional limiting normal
		cone} to a set $\Omega\subset\R^s$ at $\oz\in\Omega$ in the direction
	$d\in\R^s$ introduced by Ginchev and Mordukhovich \cite{gin-mor} by
	implementing the limiting process \begin{equation}\label{dir-nc}
		N_\Omega(\oz;d):=\big\{v\in\R^s\;\big|\;\exists\,t_k\dn 0,\;d_k\to
		d,\;v_k\to v\;\mbox{ with }\;v_k\in\widehat{N}_\Omega(\oz+t_k
		d_k)\big\}. \end{equation} It is obvious that \eqref{dir-nc} reduces
	to the limiting normal cone \eqref{lnc} for $d=0$. Given a
	set-valued mapping $F\colon\R^n\tto\R^m$ and a point
	$(\ox,\oy)\in\gph F$, the {\em directional limiting coderivative} of
	$F$ at $(\ox,\oy)$ in the direction $(u,v)\in\R^n\times\R^m$ is
	defined by Gfrerer \cite{g} as \begin{equation*}
		D^*F\big((\ox,\oy);(u,v)\big)(v^*):=\big\{u^*\in\R^n\;\big|\;(u^*,-v^*)\in
		N_{\text{gph}\,F}\big((\ox,\oy);(u,v)\big)\big\}\;\mbox{ for all
		}\;v^*\in\R^m \end{equation*} by using the directional normal cone
	\eqref{dir-nc} to the graph of $F$ at $(\ox,\oy)$ in the direction
	$(u,v)$. The aforementioned semismooth$^*$ property of $F$ is now
	formulated as follows.\vspace*{-0.05in}
	
	\begin{Definition}[\bf semismooth$^*$ property of set-valued
		mappings]\label{semi*} A mapping $F\colon\R^n\tto\R^m$ is {\sc
			semismooth$^*$} at $(\bar{x},\bar{y})\in\gph F$ if whenever
		$(u,v)\in\R^n\times\R^m$ we have the equality \begin{equation*}
			\langle u^*,u\rangle=\langle v^*,v\rangle\;\mbox{ for all
			}\;(v^*,u^*)\in\gph D^*F\big((\ox,\oy);(u,v)\big) \end{equation*}
		via the graph of the directional limiting coderivative of $F$ at
		$(\ox,\oy)$ in all the directions $(u,v)$. \end{Definition}
	
	Semismooth$^*$ mappings were introduced and largely investigated in
	Gfrerer and Outrata \cite{Helmut}, where this property is verified
	for any mapping $F\colon\R^n\tto\R^m$ with the graph represented was
	a union of finitely many closed and convex sets, for normal cone
	mappings generated by convex polyhedral sets. Other equivalent
	descriptions and properties of semismooth$^*$ mappings are given in
	Mordukhovich and Sarabi \cite{BorisEbrahim}. If $F\colon\R^n\to\R^m$
	is locally Lipschitzian around $\ox$ and directionally
	differentiable at this point, then its semismooth$^*$ property
	reduces to the classical semismoothness. Although the
	standard semismooth property of locally Lipschitzian mappings
	involving directional derivatives has been conventionally used in
	the literature for the semismooth Newton method, some important
	results were obtained without the directional differentiability
	assumption; see, e.g., Kummer \cite{Kummer}, Meng et al. \cite{msz},
	and Sun \cite{Sun2001}. Such a relaxed semismooth property of
	single-valued locally Lipschitzian mappings is known as {\em
		$G$-semismoothness}. Note that, in contrast to $G$-semismoothness,
	the semismooth$^*$ property is defined for arbitrary set-valued
	mappings, and it is used for subgradient ones in this paper. But
	even for single-valued Lipschitzian mappings, the semismooth$^*$
	definition based on coderivatives may have some advantages in
	comparison with the $G$-semismooth one due to perfect coderivative
	calculus rules. \vspace*{0.05in}
	
	Now we are ready to present and discuss the major assumptions used in the rest of the paper for the design and justification of our generalized Newton algorithms to solve the gradient and subgradient systems. The following assumptions are formulated for the general subgradient inclusions \eqref{subgra-inc} at a reference point $\ox$ satisfying \eqref{subgra-inc}.\\[1ex]
	{\bf(H1)} Given a subgradient $\ov\in\partial\ph(\ox)$, the second-order subdifferential inclusion \eqref{newton-inc} is robustly solvable around $(\ox,\ov)$, i.e., there is a neighborhood $U\times V$ of $(\ox,\ov)$ such that for every $(x,v)\in\gph\partial\varphi\cap(U\times V)$ there exists a direction $d\in\R^n$ satisfying \eqref{newton-inc}.\\[1ex]
	{\bf(H2)} The subgradient mapping $\partial\varphi$ is metrically regular around $(\ox,\ov)$.\\[2ex]
	{\bf(H3)} The subgradient mapping $\partial\varphi$ is
	semismooth$^*$ at $(\ox,\ov)$.\vspace*{0.05in}
	
	Observe that in the case where $\ph$ is of class ${\cal C}^{1,1}$
	around $\ox$, we have $v=\nabla\ph(x)$ and the second-order
	subdifferential system \eqref{newton-inc} is written as
	\begin{equation}\label{2ndC11}
		-\nabla\varphi(x)\in\partial^2\varphi(x)(d). \end{equation} The
	robust solvability assumption (H1) has been discussed in
	Section~\ref{sec:solvN} with presenting sufficient conditions for
	its fulfillment; see Theorems~\ref{solvability}(ii), \ref{twiceepi}
	and Remark~\ref{robust-newton}. Note that the strong metric
	regularity of $\partial\ph$ around $(\ox,\ov)$ for subdifferentially
	continuous functions $\ph$ ensures that both assumptions (H1) and
	(H2) are satisfied. However, this is just a sufficient condition for
	the validity of (H1) and (H2). The following example borrowed from
	Klatte and Kummer \cite[Example~BE.4]{Klatte}, where it was
	constructed for different purposes, presents a function
	$\ph\colon\R^2\to\R$ of class ${\cal C}^{1,1}$ (i.e., certainly
	being continuously prox-regular), which is twice epi-differentiable
	on the entire space $\R^2$ with the semismooth, metrically regular,
	but not strongly metrically regular gradient mapping $\nabla\ph$
	around the point in question. It is worth mentioning that the given
	example illustrates that assuming $\ox$ to be a local minimizer of
	$\ph$ is {\em essential} to the validity of Conjecture~4.7 from
	Drusvyatskiy et al. \cite{dmn}; see
	Remark~\ref{robust-newton}.\vspace*{-0.05in}
	
	\begin{Example}[\bf all assumptions hold without strong metric
		regularity]\label{ex-kummer} {\rm Let $z:=(x,y)\in\R^2$ be written
			in the polar coordinates $(r,\theta)$ by \begin{equation*}
				z=r(\cos\theta+i\sin\theta). \end{equation*} We now describe the
			function $\varphi\colon\R^2\to\R$ and its partial derivatives on the
			eight cones \begin{equation*} C(k):=\left\{z:=(r\cos\theta,r\sin
				\theta)\;\Big|\;\theta\in\left[(k-1)\frac{\pi}{4},k\frac{\pi}{4}\right],\;r\ge
				0\right\},\quad k=1,\ldots,8. \end{equation*} The analytic
			expressions of $\ph$, $\nabla_x\ph$, and $\nabla_y\ph$ are collected
			in the table: 
			
			\medskip 
			\begin{center} \begin{tabular}{ || p{1em} | c | c| c|
						c| | } \hline
					$k$ & $C(k)$ & $\varphi(z)$ & $\nabla_x\varphi(z)$ &  $\nabla_y\varphi(z)$\\[0.5em] \hline\hline
					1 & $C(1)$ & $y(y-x)$ & $-y$ & $2y-x$\\
					2 & $C(2)$ & $x(y-x)$ & $-2x+y$ & $x$\\
					3 & $C(3)$ & $x(y+x)$ & $2x+y$  & $x$ \\
					4 & $C(4)$ & $-y(y+x)$ & $-y$ & $-2y-x$\\
					5 & $C(5)$ & $y(y-x)$ & $-y$  & $2y-x$\\
					6 & $C(6)$ & $x(y-x)$ & $-2x+y$ & $x$\\
					7 & $C(7)$ & $x(y+x)$ & $2x+y$ & $x$\\
					8 & $C(8)$ & $-y(y+x)$ & $-y$  & $-2y-x$\\
					\hline \end{tabular}
			\end{center}
			
			\medskip \noindent
			The function $\varphi$ and its gradient have the following properties:\\[1ex]
			{\bf(a)} The function $\varphi$ is of {\em class ${\cal C}^{1,1}$}
			with $\nabla\varphi$ being {\em piecewise linear} on $\R^2$. This is
			an obvious consequence of the definition. Thus $\ph$ is
			{\em continuously prox-regular} on $\R^2$.\\[1ex]
			{\bf(b)} The mapping $\nabla\varphi$ is {\em metrically regular} around $(0,0)$. Indeed, it is observed by Klatte and Kummer \cite[Example~BE.4]{Klatte} that the inverse mapping $(\nabla\ph)^{-1}$ is Lipschitz-like (pseudo-Lipschitz, Aubin) around $(\ox,\ov)$ with $\ox=(0,0)$ and $\ov=(0,0)$. As well known (see, e.g., Mordukhovich \cite[Theorem~3.2(ii)]{Mor18}), the latter property is equivalent to the metric regularity of the mapping $\nabla\ph$ around $\ox$.\\[1ex]
			{\bf(c)} The mapping $\nabla\varphi$ is {\em directionally differentiable} on $\R^2$, which follows from Facchinei and Pang \cite[Lemma~4.6.1]{JPang}. Hence $\varphi$ is {\em twice epi-differentiable} by Rockafellar and Wets \cite[Theorems~9.50(b), 13.40]{Rockafellar98}.\\[1ex]
			{\bf(d)} The mapping $\nabla\varphi$ is {\em semismooth} on $\R^2$ due to its piecewise linearity. This fact can be found, e.g., in  {Facchinei and Pang \cite[Proposition~7.4.6]{JPang}} and Ulbrich \cite[Proposition~2.26]{Ul}.\\[1ex]
			{\bf(e)} The mapping $\nabla\ph$ is {\em not strongly metrically
				regular} around $(\ox,\ov)$ with $\ox=(0,0)$ and $\ov=(0,0)$. To
			verify it, we proceed accordingly to Definition~\ref{met-reg}(i) and
			let $\vartheta$ be an arbitrary localization of
			$(\nabla\varphi)^{-1}$ around $\ov$ for $\ox$, i.e., such that
			\begin{equation*} \gph\vartheta=\gph(\nabla\varphi)^{-1}\cap(V
				\times U) \end{equation*} for some neighborhoods $V$ of $\ov$ and
			$U$ of $\ox$. Find $\epsilon,\gg>0$ with $\B_\epsilon(\ox)\subset U$
			and $\B_\gamma(\ov)\subset V$ and then pick $t\in\R$ such that
			$0<t<{\rm min}\{\gamma,\epsilon/\sqrt{5}\}$. This shows that
			$(t,0)\in\B_\gamma(\ov)\subset V$ with \begin{equation*}
				(0,t)=\left(r\cos\theta,r\sin\theta\right)\;\mbox{ for
				}\;r=t\;\mbox{ and }\;\theta=\frac{\pi}{2}, \end{equation*} and thus
			$\nabla\varphi(0,t)=(t,0)$. Furthermore, we have \begin{equation*}
				(-2t,-t)=\left(r\cos\theta,r\sin\theta\right)\;\mbox{for}\;r=\sqrt{5}t\;\mbox{and}\;\theta\in\Big[\pi,\frac{5\pi}{4}\Big],\;\cos\theta=-\frac{2}{\sqrt{5}},\;\sin\theta=-\frac{1}{\sqrt{5}},
			\end{equation*} which tell us that $\nabla\varphi(-2t,-t)=(t,0)$ and
			$(0,t),(-2t,-t)\in\B_\epsilon(\ox)\subset U$. This yields
			\begin{equation*}
				\big((t,0),(0,t)\big),\big((t,0),(-2t,-t)\big)\in\gph(\nabla\varphi)^{-1}\cap\big(\B_\gamma(\ov)\times\B_\epsilon(\ox)\big)\subset\gph\vartheta.
			\end{equation*} The latter means that there exists no localization
			$\vartheta$ of $(\nabla \varphi)^{-1}$ around $(\bar{v},\bar{x})$
			which is single-valued, and hence the mapping $\nabla\varphi$ is not
			strongly metrically regular around $(\ox,\ov)$.\vspace*{0.03in}
			
			Remembering finally that the metric regularity is a robust property
			and therefore holds for all points in some neighborhood of $(0,0)$,
			we deduce from (a)--(d) and Theorem~\ref{twiceepi} that all the
			imposed assumptions (H1)--(H3) are satisfied without the fulfillment
			of the strong metric regularity of $\nabla\ph$ around $(0,0)$ as
			shown in (e). It is easy to see that $\ox=(0,0)$ is a stationary
			point of $\ph$ while not its local minimizer. Thus this example does
			not contradict the conjecture from Drusvyatskiy et al. \cite{dmn}
			discussed in Remark~\ref{robust-newton}.}
	\end{Example}\vspace*{-0.05in}
	
	Now we are ready to formulate a generalized Newton algorithm to
	solve the gradient equation $\nabla\ph(x)=0$, labeled as \eqref{gra}
	in Section~\ref{intro}, where $\ph$ is of class ${\cal C}^{1,1}$
	around the reference point. This algorithm is based on the
	second-order subdifferential/generalized Hessian \eqref{2nd} of the
	function $\ph$ in question.\vspace*{-0.03in}
	
	\begin{Algorithm}[\bf Newton-type algorithm for ${\cal C}^{1,1}$ functions]\label{NM} {\rm Do the following:\\[1ex]
			{\bf Step~0:} Choose a starting point $x^0\in\R^n$ and set $k=0$.\\[1ex]
			{\bf Step~1:} If $\nabla\varphi(x^k)=0$, stop the algorithm. Otherwise move to Step~2.\\[1ex]
			{\bf Step~2:} Choose $d^k\in\R^n$ satisfying
			\begin{equation}\label{alC11} -\nabla\varphi(x^k)\in\partial\big\la
				d^k,\nabla\ph\big\ra(x^k). \end{equation} {\bf Step~3:} Set
			$x^{k+1}$ given by \begin{equation*} x^{k+1}:=x^k+d^k\;\mbox{ for
					all }\;k=0,1,\ldots. \end{equation*} {\bf Step~4:} Increase $k$ by
			$1$ and go to Step~1.} \end{Algorithm}
	
	The major step and novelty of Algorithm~\ref{NM} is the generalized
	Newton system \eqref{alC11} expressed in terms of the limiting
	subdifferential of the scalarized gradient mapping. Due to the
	coderivative scalarization formula \eqref{scal} and the second-order
	construction \eqref{2nd} we have \begin{equation*} \partial\big\la
		d^k,\nabla\ph\big\ra(x^k)=\big(D^*\nabla\ph\big)(x^k)(d^k)=\partial^2\ph(x^k)(d^k),
	\end{equation*} i.e., the iteration system \eqref{alC11} agrees with
	the second-order subdifferential  {inclusion
		\eqref{newton-inc} whose solvability} was discussed in
	Section~\ref{sec:solvN}; see Theorem~\ref{twiceepi}. The computation
	of the second-order subdifferential for ${\cal C}^{1,1}$ functions
	reduces to that of the (first-order) limiting subdifferential of the
	classical gradient mapping; {this significantly}
	simplifies the numerical implementation. Note also that, according
	to definition \eqref{lim-cod}, the direction $d^k$ in \eqref{alC11}
	can be found from \begin{equation*}
		\big(-\nabla\varphi(x^k),-d^k\big)\in
		N\big((x^k,\nabla\varphi(x^k));\gph\nabla\varphi\big).
	\end{equation*}
	
	The main goal for the rest of this section is to show that the
	metric regularity and the semismooth$^*$ properties of $\nabla\ph$
	imposed in (H2) and (H3) ensure the {\em convergence} of iterates
	$x^k\to\ox$ with {\em superlinear rate}. To proceed in this way, we
	present the following three lemmas of their own interest. The first
	lemma gives us a necessary and sufficient condition for the metric
	regularity of continuous single-valued mappings $F\colon\R^n\to\R^m$
	via their limiting coderivatives. Since $F$ is single-valued, we are
	talking about its metric regularity around $\ox$ instead of
	$(\ox,F(\ox))$.\vspace*{-0.05in}
	
	\begin{Lemma}[\bf yet another characterization of metric
		regularity]\label{metricforF} Let $F\colon\R^n\to\R^m$ be continuous
		around $\ox$. Then it is metrically regular around this point if and
		only if there exists $c>0$ and a neighborhood $U$ of $\ox$ such that
		we have the estimate \begin{equation}\label{H2metrical} \|v\|\ge
			c\|u\|\;\mbox{ for all }\;v\in D^*F(x)(u),\;x\in U,\;\mbox{ and
			}\;u\in\R^m. \end{equation} \end{Lemma} \begin{proof} If $F$ is
		metrically regular and continuous around $\ox$, then it follows from
		the book by Mordukhovich \cite[Theorem~1.54]{Mordukhovich06} that
		there are $c>0$ and an (open) neighborhood $U$ of $\bar{x}$ with
		\begin{equation}\label{Frechet} \|v\|\ge c\|u\|\;\mbox{ for all
			}\;v\in\widehat{D}^*F(x)(u),\;x\in U,\;\mbox{ and }\;u\in\R^m
		\end{equation} in terms of the regular coderivative \eqref{reg-cod}.
		Fix any $x\in U$, $u\in\R^m$, and an element $v\in D^*F(x)(u)$ from
		the limiting coderivative \eqref{lim-cod}. The continuity of $F$
		ensures that the graphical set $\gph F$ is closed around $(x,F(x))$.
		Then using Corollary~2.36 from the aforementioned book gives us
		sequences $x_k\to x$, $u_k\to u$, and $v_k\to v$ as $k\to\infty$
		such that $v_k\in\widehat{D}^*F(x_k)(u_k)$, and thus $x_k\in U$ for
		all $k\in\N$ sufficiently large. It follows from \eqref{Frechet}
		that \begin{equation*} \|v_k\|\ge c\|u_k\|\;\mbox {for large
			}\;k\in\N. \end{equation*} Letting $k\to\infty$, we arrive at the
		estimate $\|v\|\ge c\|u\|$. On the other hand, the fulfillment of
		\eqref{H2metrical} immediately yields the metric regularity of $F$
		around $\ox$ by the coderivative criterion \eqref{cod-cr}.
	\end{proof}\vspace*{0.03in}
	
	The next lemma presents an equivalent description of
	semismoothness$^*$ for Lipschitzian gradient mappings. This result
	follows from the combination of Proposition~3.7 in Gfrerer and
	Outrata \cite{Helmut} and Theorem~13.52 in Rockafellar and Wets
	\cite{Rockafellar98} due to the symmetry of generalized Hessian
	matrices; see Mordukhovich and Sarabi
	\cite[Proposition~2.4]{BorisEbrahim} for more
	details.\vspace*{-0.05in}
	
	\begin{Lemma}[\bf equivalent description of
		semismoothness$^*$]\label{equisemi} Let $\varphi\colon\R^n\to\R$ be
		of class ${\cal C}^{1,1}$ around $\bar{x}$. Then the gradient
		mapping $\nabla\varphi$ is semismooth$^*$ at $\ox$ if and only if
		\begin{equation*}
			\nabla\varphi(x)-\nabla\varphi(\ox)-\partial^2\varphi(x)(x-\ox)\subset
			o(\|x-\ox\|), \end{equation*} which means that for every
		$\epsilon>0$ there exists $\delta>0$ such that \begin{equation*}
			\|\nabla\varphi(x)-\nabla\varphi(\bar{x})-v\|\le\epsilon\|x-\ox\|\;\mbox{
				for all }\;v\in\partial^2\varphi(x)(x-\ox)\;\mbox{ and
			}\;x\in\B_\delta(\ox). \end{equation*} \end{Lemma}\vspace*{-0.03in}
	
	The last lemma establishes a useful subadditivity property of the
	limiting coderivative of single-valued and locally Lipschitzian
	mappings.\vspace*{-0.05in}
	
	\begin{Lemma}[\bf subadditivity of
		coderivatives]\label{phanracoderivative} If $F\colon\R^n\to\R^m$ is
		locally Lipschitzian around $\ox$, then we have the inclusion
		\begin{equation}\label{cod-subadd} D^*F(\ox)(u_1+u_2)\subset
			D^*F(\ox)(u_1)+D^*F(\ox)(u_2)\;\mbox{ for all }\;u_1,u_2\in\R^m.
	\end{equation} \end{Lemma} \begin{proof}
	It follows from the
		scalarization formula \eqref{scal} that \begin{equation*}
			D^*F(\ox)(u_1+u_2)=\partial\langle u_1+u_2,F\rangle(\ox).
		\end{equation*} The subdifferential sum rule from Mordukhovich
		\cite[Theorem~2.33]{Mordukhovich06} and the aforementioned scalarization formula
		ensure that
		\begin{equation*} \partial\langle
			u_1+u_2,F\rangle(\ox)\subset\partial\langle
			u_1,F\rangle(\ox)+\partial\langle
			u_2,F\rangle(\ox)=D^*F(\ox)(u_1)+D^*F(\ox)(u_2), \end{equation*}
		which therefore completes the proof of the lemma. \end{proof}
	
	The next theorem is the main result of this section that establishes
	the {\em local superlinear convergence} of Algorithm~\ref{NM} under
	the imposed assumptions.\vspace*{-0.05in}
	
	\begin{Theorem}[\bf local convergence of the Newton-type algorithm
		for ${\cal C}^{1,1}$ functions]\label{localconverge} Let $\ox$ be a
		solution to \eqref{gra} for which assumptions {\rm(H1)--(H3)} are
		satisfied. Then there exists a neighborhood $U$ of $\ox$ such that
		for all $x^0\in U$ Algorithm~{\rm\ref{NM}} is well-defined and
		generates a sequence $\{x^k\}$ that converges $Q$-superlinearly to
		$\ox$, i.e., we have \begin{equation*}
			\lim_{k\to\infty}\frac{\|x^{k+1}-\bar{x}\|}{\|x^k-\bar{x}\|}=0.
	\end{equation*} \end{Theorem} \begin{proof} Define the set-valued
		mapping $G_\ph\colon\R^n\times\R^n\tto\R^n$ by \begin{equation*}
			G_\ph(x,u):=-\big(D^*\nabla\varphi\big)(x)(-u)=-\partial^2\varphi(x)(-u)\;\mbox{
				for all }\;x,u\in\R^n. \end{equation*} Assumption (H1) allows us to
		construct the sequence of iterates $\{x^k\}$ in Algorithm~\ref{NM}.
		Using (H2) and its characterization from Lemma~\ref{metricforF}, we
		find $c>0$ and a neighborhood $U$ of $\ox$ with \begin{equation*}
			\|v\|\ge c\|u\|\;\mbox{ for all }\;v\in G_\varphi(x,u),\;x\in
			U,\;\mbox{ and }\;u\in\R^n. \end{equation*} Let us now verify the
		inclusion \begin{equation}\label{G-inc}
			\nabla\ph(x)-\nabla\ph(\ox)+G_\ph(x,u)\subset
			G_\ph(x,x+u-\ox)+o(x-\ox)\B \end{equation} for the above vectors
		$x,u$. Indeed, taking any $v_1\in G_\ph(x,u)$, i.e.,
		$-v_1\in\partial^2\ph(x)(-u)$, and using the subadditivity inclusion
		from Lemma~\ref{phanracoderivative} lead us to \begin{equation*}
			\partial^2\ph(x)(-u)\subset\partial^2\ph(x)(-x-u+\ox)+\partial^2\ph(x)(x-\ox)
		\end{equation*} and thus ensure the existence of
		$v_2\in\partial^2\ph(x)(-x-u+\ox)$ such that
		$-v_1-v_2\in\partial^2\ph(x)(x-\ox)$. The semismoothness$^*$
		assumption (H3) and its equivalent description in
		Lemma~\ref{equisemi} tell us that \begin{equation*}
			\lim_{x\to\ox}\frac{\|\nabla\ph(x)-\nabla\ph(\ox)+v_1+v_2\|}{\|x-\ox\|}=0,
		\end{equation*} which therefore verifies \eqref{G-inc}. All of this
		allows us to proceed similarly to the proof of Theorem~10.7 in
		Klatte and Kummer \cite{Klatte} and thus find a neighborhood $U$ of
		$\ox$ such that, whenever the starting point $x^0\in U$ is selected,
		Algorithm~\ref{NM} generates a well-defined sequence of iterates
		$\{x^k\}$, which converges $Q$-superlinearly to $\ox$. This
		therefore completes proof of the theorem. \end{proof}
	
	Assumption (H1) has been already discussed and obviously cannot be
	removed or relaxed; otherwise Algorithm~\ref{NM} is not
	well-defined. Now we present two examples showing that assumptions
	(H2) and (H3) are essential for the convergence (not even talking
	about its $Q$-superlinear rate) of Algorithm~\ref{NM}. Let us start
	with the semismoothness$^*$ assumption (H3).\vspace*{-0.05in}
	
	\begin{Example}[\bf semismooth$^*$ property is essential for
		convergence]\label{semi-conv} {\rm Consider the Lipschitz continuous
			function on $\R$ given by \begin{equation*}
				\psi(x):=\left\{\begin{array}{ll}
					\disp x^2\sin\frac{1}{x}+2x&{\rm if}\;\;x\ne 0,\\
					0&{\rm if}\;\;x=0, \end{array} \right. \end{equation*} which was
			used in Jiang et al. \cite{defeng} to show that the semismooth
			Newton method for solving the equation $\psi(x)=0$ fails to locally
			converge to $\ox:=0$. Consider further the ${\cal C}^{1,1}$ function
			\begin{equation*} \ph(x):=\int_{0}^{x}\psi(t)dt,\quad x\in\R,
			\end{equation*} with $\nabla\ph(x)=\psi(x)$ on $\R$, and hence
			$\nabla\ph(\ox)=0$. As shown in Mordukhovich and Sarabi
			\cite[Example~4.5]{BorisEbrahim}, the mapping $\nabla\ph$ is not
			semismooth$^*$ at $\ox$ and iterations \eqref{alC11} constructed to
			compute tilt-stable local minimizers of $\ph$ (see below) do not
			locally converge to $\ox$.} \end{Example}\vspace*{-0.05in}
	
	The next example reveals that assumption (H2) cannot be improved by
	relaxing the metric regularity property to the metric subregularity
	or even to the strong metric subregularity of the gradient mapping
	$\nabla\varphi$ at $\ox$ in order to keep the convergence of
	iterates in Algorithm~\ref{NM}.\vspace*{-0.05in}
	
	\begin{Example}[\bf convergence failure under strong metric
		subregularity]\label{smr-fail} {\rm Consider the function $\varphi\colon\R^2\to\R$ defined by
			by \begin{equation}\label{phC11}
				\varphi(x,y):=\frac{1}{2}x|x|+\frac{1}{2}y|y|\;\mbox{ for all
				}\;(x,y)\in\R^2. \end{equation} It is clear that the function $\ph$
			is of class ${\cal C}^{1,1}$ around $\oz:=(0,0)$ with
			$\nabla\varphi(x,y)=(|x|,|y|)$ for all $z:=(x,y)\in\R^2$ and
			$\nabla\varphi(\oz)=0$. The simple computation tells us that
			\begin{equation*}
				\big(D\nabla\varphi\big)(x,y)(u_1,u_2)=\begin{cases}
					\big\{\big(|u_1|,|u_2|\big)\;\big|\;u_1,u_2\in\R\big\}&\text{if}\quad x=0,\;y=0,\\
					\big\{\big(u_1\,\text{sgn}(x),|u_2|\big)\;\big|\;u_1,u_2\in\R\big\}&\text{if}\quad x\ne 0,\;y=0,\\
					\big\{\big(|u_1|,u_2\,\text{sgn}(y)\big)\big|\;u_1,u_2\in\R\big\}&\text{if}\quad x=0,\;y \ne 0,\\
					\big\{\big(u_1\,\text{sgn}(x),u_2\,\text{sgn}(y)\big)\;\big|\;u_1,u_2\in\R\big\}&\text{if}
					\quad x\ne 0,\;y\ne 0. \end{cases} \end{equation*} It follows from
			the Levy-Rockafellar criterion that the mapping $\nabla\varphi$ is
			strongly metrically subregular at any point $(x,y)\in\gph\nabla\ph$.
			Furthermore, $\nabla\varphi$ is semismooth$^*$ at $\oz$ since it is
			piecewise linear on $\R^2$, and thus assumption (H3) is satisfied.
			The fulfillment of (H1) is proved in Theorem~\ref{solvability}. Let
			us now show that the sequence of iterates $\{z^k\}$ generated by
			Algorithm~\ref{NM} does not converge to $\oz$. Indeed, fix any $r>0$
			and pick an arbitrary starting point $z^0$ in the form
			$z^0:=(0,r)\in\B_r(\oz)$. To run the algorithm, we need to find
			$d^0\in\R^2$ such that \begin{equation}\label{Step1Ex}
				-\nabla\varphi(z^0)\in\partial^2\varphi(z^0)(d^0). \end{equation}
			Using the second-order subdifferential \eqref{2nd} for the function
			$\ph$ from \eqref{phC11} gives us \begin{equation*}
				\partial^2\varphi(z^0)(u_1,u_2)=\begin{cases}
					\big\{(\alpha u_1,u_2)\;\big|\;\alpha\in[-1,1]\big\}&\text{if}\qquad u_1\ge 0,\;u_2\in\R,\\
					\big\{(\alpha u_1,u_2)\;\big|\;\alpha\in\{-1,1\big\}&\text{if}\qquad
					u_1<0,\;u_2\in\R. \end{cases} \end{equation*} This shows that the
			direction $d^0=(1,-r)$ satisfies inclusion \eqref{Step1Ex}. Put
			further $z^1:=z^0+d^0=(1,0)$ and find a direction $d^1\in\R^2$
			satisfying the inclusion \begin{equation}\label{Step2Ex}
				-\nabla\varphi(z^1)\in\partial^2\varphi(z^1)(d^1). \end{equation}
			Computing again the second-order subdifferential brings us to the
			expression \begin{equation*} \partial^2\varphi(z^1)(u_1,u_2)=
				\begin{cases}
					\big\{(u_1,\alpha u_2)\;\big|\;\alpha\in[-1,1]\big\}&\text{if}\qquad u_1\in\R,\;u_2\ge 0,\\
					\big\{(u_1,\alpha
					u_2)\;\big|\;\alpha\in\{-1,1\}\big\}&\text{if}\qquad
					u_1\in\R,\;u_2<0 \end{cases} \end{equation*} and then verifies that
			the direction $d^1=(-1,r)$ satisfies the inclusion in
			\eqref{Step2Ex}. Thus $z^2:=z^1+d^1=(0,r)$. Continuing this process,
			we construct the sequence of iterates $\{z^k\}$ such that
			$z^{2k}=z^0$ and $z^{2k+1}=z^1$ for all $k\in\N$. It is obvious that
			$\{z^k\}$ does not converge to $\bar{z}$.}
	\end{Example}\vspace*{-0.05in}
	
	There are several Newton-type methods to solve Lipschitzian
	equations $f(x)=0$ that apply to gradient systems \eqref{gra} with
	$f(x):=\nabla\ph(x)$, where $\ph$ is of class ${\cal C}^{1,1}$. Such
	methods are mainly based of various {\em generalized directional
		derivatives} and can be found, e.g., in Klatte and Kummer
	\cite{Klatte}, Pang \cite{p90}, Hoheisel et al. \cite{HungBoris},
	Mordukhovich and Sarabi \cite{BorisEbrahim}, and the references
	therein. It is beyond the scope of this paper to discuss their
	detailed relationships with Algorithm~\ref{NM}. However, let us
	compare the proposed algorithm with the {\em semismooth Newton
		method} to solve equation \eqref{gra}, which is based on the
	generalized Jacobian by Clarke \cite{cl} for the mapping
	$f=\nabla\ph$.\vspace*{-0.05in}
	
	\begin{Remark}[\bf comparing Algorithm~\ref{NM} with
		the semismooth Newton method and its variants]\label{semialg}
		{\rm In the setting of \eqref{gra}, the semismooth
			Newton method constructs the iterations
			\begin{equation}\label{iterationN}
				x^{k+1}=x^{k}-(A^k)^{-1}\nabla\varphi(x^k),\quad k=0,1,\ldots,
			\end{equation} where for each $k$ a nonsingular matrix $A^k$ is
			taken from the {\em generalized Jacobian} of $\nabla\ph$ at $x=x^k$,
			which is expressed for functions $\ph$ of class ${\cal C}^{1,1}$
			around $x$ via the convex hull $A^k\in{\rm co}\Bar\nabla^2\ph(x^k)$
			of the {\em limiting Hessian} matrices (or \textit{Bouligand
				Jacobian}) defined by \begin{equation}\label{lim-hes}
				\Bar\nabla^2\ph(x):=\Big\{\lim_{m\to\infty}\nabla^2\ph(u_m)\;\Big|\;u_m\to
				x,\;u_m\in Q_\ph\Big\},\quad x\in\R^n, \end{equation} where $Q_\ph$
			stands for the set on which $\ph$ is twice differentiable. It
			follows from the classical Rademacher theorem that \eqref{lim-hes}
			is a nonempty compact in $\R^{n\times n}$. Observe that
			\begin{equation*} {\rm co}\,\big[\partial\la
				u,\nabla\ph\ra(x)\big]=\co\partial^2\ph(x)(u)=\big\{A^*u\;\big|\;A\in{\rm
					co}\Bar\nabla^2\ph(x)\big\},\quad u\in\R^n, \end{equation*} which
			tells us that, in contrast to our algorithm \eqref{alC11}, the
			semismooth method \eqref{iterationN} employs the {\em convex hull}
			of the corresponding set. A serious disadvantage of
			\eqref{iterationN} is that it requires the {\em invertibility} of
			all the matrices from ${\rm co}\Bar\nabla^2\ph(x^k)$; otherwise the
			semismooth Newton algorithm \eqref{iterationN} is simply {\em not
				well-defined}. Observe that nothing like that is required to run our
			Algorithm~\ref{NM}. Indeed, the invertibility assumption is even
			more restrictive than the strong metric regularity of $\nabla\ph$
			around $\ox$, which is also not required in the imposed assumptions
			(H1)--(H3), that ensure the well-posedness and local superlinear
			convergence of Algorithm~\ref{NM}. {To overcome this
				disadvantage of the semismooth Newton method, one of the possible
				ideas developed by Sun \cite{Sun2001} is to take matrices from
				$\Bar\nabla^2\ph(x)$ instead of ${\rm co}\Bar\nabla^2\ph(x)$, i.e.,
				to construct the iterations as in \eqref{iterationN} with
				$A^k\in\Bar\nabla^2\ph(x^k)$. This method requires the
				nonsingularity of  $\Bar\nabla^2\ph(\ox)$, which is weaker than the
				nonsingularity of the Clarke generalized Jacobian. However, there
				are some drawbacks of this requirement in comparison with our
				assumption that $\nabla\varphi$ is metrically   regular around
				$\ox$. The first observation to make is that, to the best of our
				knowledge, calculus rules for the Bouligand Jacobian is more limited
				in comparison with the extensive ones available for coderivatives.
				Meanwhile, the metric regularity can be fully characterized by the
				Mordukhovich coderivative criterion for general set-valued mappings.
				Furthermore, the Bouligand-based semismooth Newton method clearly
				addresses just ${\cal C}^{1,1}$ functions, while Algorithm~\ref{NM}
				in our approach can be seen as a {\em bridge} to develop a generalized
				Newton method for solving subgradient inclusions $0\in
				\partial\varphi(x)$, where $\varphi$ is a prox-regular function as
				in Section \ref{sec:prox}, and where $\varphi= f + g$ with a
				$\mathcal{C}^2$-smooth function $f$ and a prox-regular function $g$
				as in Section~\ref{sec:Newcomposite}. Observe that our standing
				assumptions are formulated and used for {\em set-valued mappings without
					any directional differentiability requirements} or the like.}
			
			To conclude this remark}, we present a specific {\em one-dimensional example, where all the imposed assumptions (H1)--(H3) of Theorem~\ref{localconverge} are satisfied while the directional differentiability is not required. Indeed, consider the cost function $\varphi: [-1,1]\to \R$ defined by
			$$
			\varphi(x):= \int_{-1}^{x}H(t)dt, \quad x\in [-1,1],
			$$
			with the integrand $H:[-1,1]\to\R$ having the following properties:
			\begin{itemize}
				\item[\bf (i)] $H$ is Lipschitz continuous and metrically regular around $\ox:=0$ with $H(\ox)=0$.
				\item[\bf (ii)] $H$ is monotone on $(-1,1)$.
				\item[\bf (iii)] $	\|H(x) - H(\bar{x})+v\| =o(\|x-\bar{x}\|) \quad \text{for all }\; x \; \text{near }\; \bar{x}, \; v \in DH(x)(\bar{x}-x)$.
				\item[\bf (iv)] $H$ is not directionally differentiable at $\bar{x}$.
			\end{itemize}
			An explicitly constructed function $H$ of this type can be found in Hoheisel et al. \cite[Example~4.11]{HungBoris}. 
			It is clear that $\varphi$ is a differentiable convex function on $(-1,1)$ by (ii) and the fact that $\nabla\varphi(x) = H(x)$ for all $x \in (-1,1)$. Thus the metric regularity of the gradient mapping $\nabla\ph$ around $\bar{x}$ is equivalent to its strong metric regularity around $\bar{x}$ due to Arag\'on Artacho and Geoffroy \cite[Proposition~3.8]{ag}. This verifies the assumptions (H1) and (H2) in our paper. By Proposition~2.4 from Mordukhovich and Sarabi \cite{BorisEbrahim} and property  (iii), we conclude that $\nabla\varphi$ is semismooth$^*$ at $\bar{x}$. Meanwhile, property (iv) tells us that $\nabla\varphi$ is not directionally differentiable at $\bar{x}$.}
	\end{Remark}\vspace*{-0.05in}
	
	Next we consider a particular case of the gradient stationary
	equation \eqref{gra}, where $\ox$ is a {\em local minimizer} of
	$\ph$. Moreover, our attention is paid to the remarkable class of
	local minimizers exhibiting the property of {\em tilt stability}
	introduced by Poliquin and Rockafellar \cite{Poli}. In the case of
	tilt-stable minimizers, Algorithm~\ref{NM} was developed by
	Mordukhovich and Sarabi \cite{BorisEbrahim}. The results established
	below improve those from the latter paper. First we recall the
	notion of tilt-stable local minimizers for the general case of
	extended-real-valued functions
	$\ph\colon\R^n\to\oR$.\vspace*{-0.05in}
	
	\begin{Definition}[\bf tilt-stable local
		minimizers]\label{def:tilt} Given $\varphi\colon\R^n\to\oR$, a
		point $\ox\in\dom\varphi$ is a {\sc tilt-stable local minimizer} of
		$\varphi$ if there exists a number $\gamma>0$ such that the mapping
		\begin{equation}\label{tilt} M_\gamma\colon v \mapsto{\rm
				argmin}\big\{\varphi(x)-\langle v,x\rangle\;\big|\;x
			\in\B_\gamma(\ox)\big\} \end{equation} is single-valued and
		Lipschitz continuous on some neighborhood of $0\in\R^n$ with
		$M_\gamma(0)=\{\ox\}$. \end{Definition}\vspace*{-0.03in} This notion
	was largely investigated and characterized in second-order
	variational analysis with many applications to constrained
	optimization. Besides the seminal paper by Poliquin and Rockafellar
	\cite{Poli}, we refer the reader to Chieu et al. \cite{ChieuNghia},
	Drusvyatskiy and Lewis \cite{dl}, Drusvyatskiy et al. \cite{dmn},
	Gfrerer and Mordukhovich \cite{gm}, Mordukhovich \cite{Mor18},
	Mordukhovich and Nghia \cite{MorduNghia}, Mordukhovich and
	Rockafellar \cite{mr} and the bibliographies therein. Some of these
	characterizations are used in the following
	theorem.\vspace*{-0.05in}
	
	\begin{Theorem}[\bf Newton-type method for tilt-stable minimizers of ${\cal C}^{1,1}$ functions]\label{thm:tilt} Let $\varphi\colon\R^n\to\R$ be of class ${\cal C}^{1,1}$ around a given point $\ox$, which is a tilt-stable local minimizer of $\varphi$. Then there is a neighborhood $U$ of $\ox$ such that the following assertions hold:\\[1ex]
		{\bf(i)} For any $x\in U$ there exists a direction $d\in\R^n$
		satisfying the inclusion \begin{equation}\label{directiondescent}
			-\nabla\varphi(x)\in\partial\la d,\nabla\varphi\ra(x).
		\end{equation} Furthermore, we have that
		$\langle\nabla\varphi(x),d\rangle<0$ whenever $x\ne\ox$ and that for
		each $c\in(0,1)$ there is $\delta>0$ ensuring the fulfillment of the
		inequality \begin{equation}\label{linesearch}
			\varphi(x+td)\le\varphi(x)+ct\langle\nabla\varphi(x),d\rangle\;\mbox{
				for all }\;t\in(0,\delta). \end{equation} {\bf(ii)} If in addition
		the gradient mapping $\nabla\varphi$ is semismooth$^*$ at $(\ox,0)$,
		then Algorithm~{\rm\ref{NM}} is well-defined for any starting point
		$x^0\in U$ and generates a sequence $\{x^k\}$ that $Q$-superlinearly
		converges to $\ox$, while the sequence of the function values
		$\{\varphi(x^k)\}$ $Q$-superlinearly converges to $\ph(\ox)$, and
		the sequence of the gradient values $\{\nabla\varphi(x^k)\}$
		$Q$-superlinearly converges to $0$. \end{Theorem}\vspace*{-0.1in}
	\begin{proof} To verify (i), deduce from Drusvyatskiy and Lewis
		\cite[Proposition~3.1]{dl} that the imposed tilt stability of the
		local minimizer $\ox$ implies that the gradient mapping
		$\nabla\varphi$ is strongly metrically regular around $(\ox,0)$.
		Then Theorem~\ref{strongsol} tells us that there exists a
		neighborhood $U_1$ of $\bar{x}$ such that for all $x\in U_1$ we can
		find a direction $d\in\R^n$ satisfying \eqref{directiondescent}.
		Furthermore, it follows from Chieu et al.
		\cite[Theorem~4.7]{ChieuLee17} that there exists another
		neighborhood $U_2$ of $\ox$ such that $\varphi$ is strongly convex
		on $U_2$, and we have \begin{equation}\label{pd} \langle
			z,u\rangle>0\;\mbox{ for all }\;z\in\partial^2\varphi(x)(u)\;\mbox{
				and }\;x\in U_2,\;u\ne 0. \end{equation} Denote $U:=U_1\cap U_2$ and
		fix any $x\in U$ with $x\ne\ox$, which gives us $d\in\R^n$
		satisfying \eqref{directiondescent}. To show that $d\ne 0$, assume
		the contrary and then get from \eqref{directiondescent} that
		$\nabla\ph(x)=0$. Hence it follows from the strong convexity of
		$\ph$ that $x$ is a strict global minimizer of $\varphi$ on $U$,
		which clearly contradicts the tilt stability of $\ox$ by
		Definition~\ref{def:tilt}, and thus $d\ne 0$. Combining the latter
		with \eqref{directiondescent} and \eqref{pd}, we get $\langle\nabla
		\varphi(x),d\rangle<0$ and hence conclude that \eqref{linesearch}
		holds by using, e.g., Lemma~2.19 from Izmailov and Solodov
		\cite{Solo14}.\vspace*{0.03in}
		
		Next we verify assertion (ii). As follows from
		Theorem~\ref{solvability}(ii), the imposed strong metric regularity
		of $\nabla\ph$ around $\ox$ ensures the fulfillment of assumption
		(H1) and of course (H2). The additional semismoothness$^*$
		assumption on $\nabla\ph$ at $(\ox,0)$ is (H3) in this setting, and
		hence we deduce the well-posedness and local superlinear convergence
		of iterates $\{x^k\}$ in Algorithm~\ref{NM} from
		Theorem~\ref{localconverge}. To show that the sequence
		$\{\varphi(x^k)\}$ converges superlinearly to $\varphi(\ox)$, we
		conclude by the ${\cal C}^{1,1}$ property of $\ph$ around $\ox$ and
		Lemma~A.11 from Izmailov and Solodov \cite{Solo14} that there exists
		$\ell>0$ yielding \begin{equation*}
			|\varphi(x^{k+1})-\varphi(\ox)|\le\frac{\ell}{2}\|x^{k+1}-\ox\|^2\;\mbox{
				for sufficiently large }\;k\in\N. \end{equation*} Furthermore, the
		second-order growth condition that follows from tilt stability of
		$\ox$ (see, e.g., Mordukhovich and Nghia
		\cite[Theorem~3.2]{MorduNghia}) gives us $\kappa>0$ such that
		\begin{equation*}
			|\varphi(x^k)-\varphi(\ox)|\ge\varphi(x^k)-\varphi(\ox)\ge\frac{1}{2\kappa}\|x^k-\ox\|^2\;\mbox{
				for large }\;k\in\N. \end{equation*} Combining the two estimates
		above produces the inequality \begin{equation*}
			\frac{|\varphi(x^{k+1})-\varphi(\ox)|}{|\varphi(x^k)-\varphi(\ox)|}\le\ell\kappa\frac{\|x^{k+1}-\ox\|}{\|x^k-\ox\|},
		\end{equation*} which deduces the claimed superlinear convergence of
		$\{\ph(\ox)\}$ from the one established for
		$\{x^k\}$.\vspace*{0.03in}
		
		To finish the proof, it remains to to show that the sequence
		$\{\nabla\varphi(x^k)\}$ superlinearly converges to $0$. Indeed, the
		Lipschitz continuity of $\nabla\varphi$ around $\ox$ gives us a
		constant $\ell>0$ such that \begin{equation*}
			\|\nabla\varphi(x^{k+1})\|=\|\nabla\varphi(x^{k+1})-\nabla\varphi(\ox)\|\le\ell\|x^{k+1}-\ox\|\;\mbox{
				for large }\;k\in\N. \end{equation*} The strong local monotonicity
		of $\nabla\varphi$ around $(\bar{x},0)$  under the tilt stability of
		$\ox$ (see, e.g., Poliquin and Rockafellar \cite[Theorem~1.3]{Poli}
		for a more general result) tells us that there exists a constant
		$\kappa>0$ with \begin{equation*}
			\langle\nabla\varphi(x^k)-\nabla\varphi(\ox),x^k-\ox\rangle\ge\kappa\|x^k-\bar{x}\|^2\;\mbox{
				for large }\;k\in\N, \end{equation*} and hence
		$\|\nabla\varphi(x^k)\|\ge\kappa\|x^k-\ox\|$ for such $k$. Thus we
		arrive at the estimate \begin{equation*}
			\frac{\|\nabla\varphi(x^{k+1})\|}{\|\nabla\varphi(x^k)\|}\le\frac{\ell}{\kappa}\frac{\|x^{k+1}-\ox\|}{\|x^k-\ox\|},
		\end{equation*} which verifies that the gradient sequence
		$\{\nabla\ph(x^k)\}$ superlinearly converges to 0 as $k\to\infty$
		due to such a convergence of $x^k\to\ox$ obtained above. This
		completes the proof of the theorem. \end{proof}\vspace*{-0.03in}
	
	To conclude, we compare Theorem~\ref{thm:tilt} with the recent
	results from Mordukhovich and Sarabi
	\cite{BorisEbrahim}.\vspace*{-0.05in}
	
	\begin{Remark}[\bf comparison with known results under tilt
		stability]\label{comp-tilt} {\rm The aforementioned paper
			\cite{BorisEbrahim} developed Algorithm~\ref{NM}, written in an
			equivalent form, for computing tilt-stable minimizers $\ox$ of
			$\ph\in{\cal C}^{1,1}$. As follows from Drusvyatskiy and Lewis
			\cite[Theorem~3.3]{dl} (see also Drusvyatskiy et al.
			\cite[Proposition~4.5]{dmn} for a more precise statement), that the
			tilt-stability of $\ox$ for $\ph$ is equivalent to the strong metric
			regularity of $\nabla\ph$ around $(\ox,0)$ provided that $\ox$ is a
			local minimizer of $\ph$. The latter requirement is essential as
			trivially illustrated by the function $\ph(x):=-x^2$ on $\R$, where
			$\ox=0$ is not a tilt-stable local minimizer while $\nabla\ph$ is
			strongly metrically regular around $(\ox,0)$. Note also that the
			solvability/well-posedness of Algorithm~\ref{NM} holds under weaker
			assumptions than the strong metric regularity; see
			Section~\ref{sec:solvN}. The local superlinear convergence of
			$\{x^k\}$ to a tilt-stable minimizer $\ox$ in
			Theorem~\ref{thm:tilt}(ii) follows from Mordukhovich and Sarabi
			\cite[Theorem~4.3]{BorisEbrahim} under the semismoothness$^*$ of
			$\nabla\ph$ at $(\ox,0)$. Besides the local superlinear convergence
			of $\ph(x^k)\to\ph(\ox)$ and $\nabla\ph(x^k)\to 0$ in
			Theorem~\ref{thm:tilt}(ii), the new statements of
			Theorem~\ref{thm:tilt}(i) include the {\em descent property}
			$\la\nabla\ph(x^k),d^k\ra<0$ of the algorithm and the {\em
				backtracking line search} \eqref{linesearch} at each iteration.}
	\end{Remark}\vspace*{-0.2in}

\section{Generalized Newton Algorithm for Prox-Regular
Functions}\label{sec:prox}\vspace*{-0.05in}\setcounter{equation}{0}

	This section is devoted to the design and justification of a
	generalized Newton algorithm to solve subgradient inclusions of type
	\eqref{subgra-inc}, where $\ph\colon\R^n\to\oR$ is a continuously
	prox-regular function. We have already considered this remarkable
	class of functions in Section~\ref{sec:solvN} concerning solvability
	of the second-order subdifferential systems \eqref{newton-inc},
	which play a crucial role in the design of the generalized Newton
	algorithm to find a solution of \eqref{subgra-inc} in this section.
	The approach developed here is to reduce the subgradient inclusion
	\eqref{subgra-inc} to a gradient one of type \eqref{gra} with the
	replacement of $\ph$ from the class of continuously prox-regular
	functions by its Moreau envelope, which is proved to be of class
	${\cal C}^{1,1}$. This leads us to the well-defined and
	implementable generalized Newton algorithm expressed in terms of the
	second-order subdifferential of $\ph$ and the single-valued,
	monotone, and Lipschitz continuous proximal mapping associated with
	this function. We show that the proposed algorithm exhibits local
	superlinear convergence under the standing assumptions imposed and
	discussed above.\vspace*{0.03in}
	
	First we formulate the notions of Moreau envelopes and proximal
	mappings associated with extended-real-valued functions. Recall that
	$\ph\colon\R^n\to\oR$ is {\em proper} if
	$\dom\ph\ne\emp$.\vspace*{-0.05in}
	
	\begin{Definition}[\bf Moreau envelopes and proximal
		mappings]\label{def:moreau} Let $\varphi\colon\R^n\to\oR$ be proper
		and l.s.c., and let $\lambda>0$. The {\sc Moreau envelope}
		$e_\lambda\varphi$ and the {\sc proximal mapping} $\textit{\rm
			Prox}_\lambda\varphi$ are defined by \begin{equation}\label{Moreau}
			e_\lambda\varphi(x):=\inf\left\{\varphi(y)+\frac{1}{2\lambda}\|y-x\|^2\;\Big|\;y\in\R^n\right\},
		\end{equation} \begin{equation}\label{Prox} \textit{\rm
				Prox}_\lambda\varphi(x):={\rm
				argmin}\left\{\varphi(y)+\frac{1}{2\lambda}\|y-x\|^2\;\Big|\;y\in\R^n\right\}.
		\end{equation} If $\lambda=1$, we use the notation $e\varphi(x)$ and
		$\text{\rm Prox}_{\varphi}(x)$ in \eqref{Moreau} and \eqref{Prox},
		respectively. \end{Definition}
	
	Both Moreau envelopes and proximal mappings have been well
	recognized in variational analysis and optimization as efficient
	tools of regularization and approximation of nonsmooth functions.
	This has have done particularly for convex functions and more
	recently for continuously prox-regular functions; see Rockafellar
	and Wets \cite{Rockafellar98} and the references therein. Proximal
	mappings and the like have been also used in numerical algorithms of
	the various types revolving around computing proximal points; see,
	e.g., the book by Beck \cite{Beck} and the paper by Hare and
	Sagastiz\'abal \cite{Hare} among many other publications. In what
	follows we are going to use the proximal mapping \eqref{Prox} for
	designing a generalized Newton algorithm to solve subgradient
	inclusions \eqref{subgra-inc} for continuously  prox-regular
	functions with applications to regularized least square
	problems.\vspace*{0.03in}
	
	Here is our basic Newton-type algorithm to solve
	subgradient inclusions \eqref{subgra-inc} generated by prox-regular
	functions $\ph$. Note that this algorithm constructively describes
	the area of {\em choosing a starting point} that depends on the {\em
		constant of prox-regularity}. The {\em subproblem} of this algorithm
	at each iteration consists of finding a {\em unique} solution to the
	optimization problem in \eqref{Prox}, which is a {\em
		regularization} of $\ph$ by using {\em quadratic
		penalties}.\vspace*{-0.05in}
	
	\begin{Algorithm}[\bf Newton-type algorithm for subgradient inclusions]\label{NM3} {\rm Let $r>0$ be a constant of prox-regularity of $\ph\colon\R^n\to\oR$ from \eqref{prox}.\\[1ex]
			{\bf Step~0:} Pick any $\lambda\in(0,r^{-1})$, choose a starting
			point $x^0$ by \begin{equation}\label{start-point} x^0\in
				U_\lambda:=\text{rge}(I+\lambda\partial\varphi), \end{equation}
			and set $k:=0$.\\[1ex]
			{\bf Step~1:} If $0\in\partial\varphi(x^k)$, then stop. Otherwise
			compute \begin{equation}\label{subprob}
				v^k:=\frac{1}{\lambda}\Big(x^k-\text{\rm
					Prox}_\lambda\varphi(x^k)\Big). \end{equation} {\bf Step~2}: Choose
			$d^k\in\R^n$ such that \begin{equation}\label{prox-dir}
				-v^k\in\partial^2\varphi(x^k-\lambda v^k,v^k)(\lambda v^k+d^k).
			\end{equation} {\bf Step~3:} Compute $x^{k+1}$ by \begin{equation*}
				x^{k+1}:=x^k+d^k,\quad k=0,1,\ldots. \end{equation*} {\bf Step 4:}
			Increase $k$ by $1$ and go to Step~1.} \end{Algorithm}
	
	Note that Algorithm~\ref{NM3} does not include computing the Moreau
	envelope \eqref{Moreau} while requiring to solve subproblem
	\eqref{subprob} built upon the proximal mapping \eqref{Prox}. By
	definition of the second-order subdifferential \eqref{2nd} and the
	limiting coderivative \eqref{lim-cod}, the implicit inclusion
	\eqref{prox-dir} for $d^k$ can be rewritten explicitly as
	\begin{equation}\label{prox-dir1} (-v^k,-\lambda v^k-d^k)\in
		N_{\text{gph}\,\partial\varphi}(x^k-\lambda v^k,v^k). \end{equation}
	Observe that for convex functions $\varphi\colon\R^n\to\oR$ we can
	choose $\lm$ in Step~0 arbitrarily from $(0,\infty)$ with
	$U_\lm=\R^n$ in \eqref{start-point}. This is due a well-known result
	of convex analysis, which is reflected in the following lemma that
	plays a crucial role in the justification of Algorithm~\ref{NM3}.
	Recall that $I$ stands for the identity operator, and that $\ph$ is
	{\em prox-bounded} if it is bounded from below by a quadratic
	function on $\R^n$.\vspace*{-0.05in}
	
	\begin{Lemma}[\bf Moreau envelopes and proximal mappings for prox-regular functions]\label{rela} Let $\varphi\colon\R^n\to\oR$ be prox-bounded on $\R^n$ and continuously prox-regular at $\ox$ for
		$\ov\in\partial\varphi(\ox)$ with  {modulus $r>0$}. Then the following assertions hold for all $\lambda\in (0,r^{-1})$, where the parameter $\lm$ can be chosen arbitrarily from $(0,\infty)$ with $U_\lm=\R^n$ if $\ph$ is convex:\\[1ex]
		{\bf(i)} The Moreau envelope $e_\lambda\varphi$ is of class ${\cal C}^{1,1}$ on $U_\lambda$ taken from \eqref{start-point}, which contains a neighborhood of $\ox+\lambda\ov$. Furthermore, $\ox$ is a solution to \eqref{subgra-inc} if and only if $\nabla e_\lambda\varphi(\ox)=0$.\\[1ex]
		{\bf(ii)} The proximal mapping $P_\lm\ph$ is single-valued,
		monotone, and Lipschitz continuous on $U_\lm$ and satisfies the
		condition $P_\lm\ph(\ox+\lm\ov)=\ox$.\\[1ex]
		{\bf(iii)} The gradient of $e_\lm\ph$ is calculated by
		\begin{equation}\label{gradEnve}
			\nabla
			e_\lambda\varphi(x)=\frac{1}{\lambda}\Big(x-\text{\rm		Prox}_\lambda\varphi(x)\Big)=\big(\lambda
			I+\partial\varphi^{-1}\big)^{-1}(x)\;\mbox{ for all }\;x\in
			U_\lambda. 
		\end{equation} 
	\end{Lemma} \begin{proof} Denote
		$\varphi_{\ov}(x):=\varphi(x+\ox) - \varphi(\ox)-\langle\ov,x\rangle$ and observe that
		$\varphi_{\ov}$ satisfies the assumptions from Poliquin and Rockafellar \cite[Theorem~4.4]{Poliquin}. This yields assertions (i)
		and (ii). The results for convex functions $\ph$ follow from
		\cite[Theorem~2.26]{Rockafellar98}. Assertion (iii) is taken from
		Poliquin and Rockafellar \cite[Theorem~4.4]{Poliquin}.
	\end{proof}\vspace*{-0.03in}
	
	The next simple lemma is also needed in the proof of the main result
	of this section.\vspace*{-0.05in}
	
	\begin{Lemma}[\bf second-order subdifferential graph]\label{2gph}
		{In the setting of Lemma~{\rm\ref{rela}}}, for any
		$\lambda\in(0,r^{-1})$, $x\in U_\lambda$, and $v=\nabla
		e_\lambda\varphi(x)$ we have the equivalence \begin{equation*}
			(v^*,x^*)\in{\rm gph}\big(D^*\nabla
			e_\lambda\varphi\big)(x,v)\iff(v^*-\lambda
			x^*,x^*)\in\gph\partial^2\varphi (x-\lambda v,v). \end{equation*}
	\end{Lemma} 
	\begin{proof} The relationships in \eqref{inverse} and \eqref{gradEnve} tell
		us that $(v^*,x^*)\in\gph(D^*\nabla e_\lambda\varphi)(x,v)$ if and
		only if \begin{equation*} -v^*\in D^*(\lambda
			I+\partial\varphi^{-1})(v,x)(-x^*). \end{equation*} Elementary
		operations with the limiting coderivative yield
		\begin{equation}\label{inverseofMoreau} D^*(\lambda
			I+\partial\varphi^{-1})(v,x)(-x^*)=-\lambda
			x^*+(D^*\partial\varphi^{-1})(v,x-\lambda v)(-x^*). \end{equation}
		This ensures the equivalence of \eqref{inverseofMoreau} to the
		inclusion \begin{equation*} \lambda
			x^*-v^*\in(D^*\partial\varphi^{-1})(v,x-\lambda v)(-x^*),
		\end{equation*} which implies by \eqref{inverse} that
		$x^*\in\partial^2\varphi(x-\lambda v,v)(v^*-\lambda x^*)$ and thus
		completes the proof. 
	\end{proof}
	
	\begin{Remark}[\bf on the iterative sequence generated by Algorithm \ref{NM3}] \label{NMandNM3} \rm  Lemma \ref{rela} and Lemma \ref{2gph} allow us to show that Algorithm \ref{NM3} is a special case of the general scheme given in Algorithm \ref{NM}. Indeed,  we can equivalently rewrite the conditions in \eqref{subprob} and \eqref{prox-dir} as
		$$
		v^k= \nabla e_\lambda \varphi(x^k) \quad \text{and }\; -\nabla e_\lambda \varphi(x^k) \in \partial^2 e_\lambda \varphi(x^k)(d^k)= \partial \langle d^k, \nabla e_\lambda \varphi\rangle (x^k). 
		$$
		Therefore, Algorithm \ref{NM3} reduces to Algorithm \ref{NM} with $\varphi:=e_\lambda \varphi$.  
	\end{Remark}
	
	\medskip 
	Now we proceed with the formulation and proof of the major result of
	this section on the well-posedness and superlinear convergence of
	the proposed algorithm for prox-regular functions.\vspace*{-0.05in}

	\begin{Theorem}[\bf local superlinear convergence of
		Algorithm~\ref{NM3}]\label{localNM3} In addition to the standing
		assumption {\rm(H1)--(H3)} with $\ov=0$ therein, let
		$\varphi\colon\R^n\to\oR$ be prox-bounded on $\R^n$ and continuously
		prox-regular at $\ox$ for $0\in\partial\varphi(\ox)$ with constants
		$r,\ve>0$ from \eqref{prox}. Then there exists a neighborhood $U$ of
		$\ox$ such that for all starting points $x^0\in U$ we have that
		Algorithm~{\rm\ref{NM3}} is well-defined and generates a sequence of
		iterates $\{x^k\}$, which converges superlinearly to the solution
		$\bar{x}$ of \eqref{subgra-inc} as $k\to\infty$. 
	\end{Theorem}
	\begin{proof} It follows from Remark \ref{NMandNM3} that solving the
		subgradient inclusion \eqref{subgra-inc} for the class of
		prox-regular functions under consideration by Algorithm \ref{NM3} is equivalent to solving
		the gradient system $\nabla e_\lm\ph(x)=0$ with the ${\cal C}^{1,1}$
		function $e_\lm\ph$ under the indicated choice of parameters by Algorithm~\ref{NM}. To
		apply Algorithm~\ref{NM} to the equation $\nabla e_\lm\ph(x)=0$, we
		need to check that assumptions (H1)--(H3) imposed in this theorem on
		$\partial\ph$ are equivalent to the corresponding assumptions on
		$\nabla e_\lm\ph$ imposed in Theorem~\ref{localconverge}.
		
		Let us first verify that assumption (H1) of the theorem yields its
		fulfillment of its counterpart for $e_\lm\ph$. Since $U_\lm$ from
		\eqref{start-point} contains a neighborhood of $\ox$ by
		Lemma~\ref{rela}(i), it suffices to show that there is a
		neighborhood $\Tilde U$ of $\ox$ such that for each $x\in\Tilde U$
		there exists $d\in\R^n$ satisfying
		\begin{equation}\label{solforenve} -\nabla
			e_\lambda\varphi(x)\in\big(D^*\nabla e_\lambda\varphi\big)(x)(d).
		\end{equation} Indeed, assumption (H1) gives us neighborhoods $U$ of
		$\ox$ and $V$ of $\ov=0$ for which inclusion \eqref{newton-inc}
		holds. Since $\nabla e_\lm\ph$ and $I-\nabla e_\lm\ph$ are
		continuous around $\ox$  and since \begin{equation*}
			\big((I-\lm\nabla e_\lm\ph)(\ox),\nabla
			e_\lm\ph(\ox)\big)=(\ox,0)\in(U\cap U_\lm)\times V, \end{equation*}
		there exists a neighborhood $\Tilde U$ of $\ox$ such that $\Tilde
		U\subset(U\cap U_\lm)$ and that \begin{equation*} (I-\lm\nabla
			e_\lm\ph)(\Tilde U)\times\nabla e_\lm\ph(\Tilde U)\subset(U\cap
			U_\lm)\times V. \end{equation*} Fix now any $x\in\Tilde U$ and put
		$y:=\nabla e_\lm\ph(x)$. Performing elementary transformations
		brings us to  \begin{equation*} D^*\big(\nabla
			e_\lambda\varphi\big)^{-1}(y,x)(y)=D^*\big(\lambda
			I+\partial\varphi^{-1}\big)(y,x)(y)=\lambda
			y+\big(D^*\partial\varphi^{-1}\big)(y,x-\lambda y)(y).
		\end{equation*} It follows from \eqref{newton-inc} that there exists $\widetilde{d} \in \R^n$ such that $-y \in (D^*\partial \varphi) (x-\lambda y, y)(\widetilde{d})$, which implies that 
		$-\widetilde{d}\in(D^*\partial\varphi^{-1})(y,x-\lambda y)(y)$.
		Denoting $d:=-\lambda y+\widetilde{d}$, we arrive at
		\begin{equation*} -d\in D^*\big(\nabla
			e_\lambda\varphi\big)^{-1}(y,x)(y), \end{equation*} which tells us
		by \eqref{inverse} that $-\nabla e_\lambda\varphi(x)\in(D^*\nabla
		e_\lambda\varphi)(x)(d)$ and thus verifies
		\eqref{solforenve}.\vspace*{0.03in}
		
		Next we show that the the metric regularity assumption (H2) on
		$\partial\ph$ around $(\ox,0)$ is equivalent to the metric
		regularity of the Moreau envelope gradient $\nabla
		e_\lambda\varphi$. To proceed, pick any $\lambda\in(0,r^{-1})$ from
		Step~0 of the algorithm and then get by Lemma~\ref{2gph} that
		\begin{equation*} 0\in\big(D^*\partial\varphi\big)(\bar{x},0)(u)\iff
			0\in\big(D^*\nabla e_\lambda\varphi\big)(\bar{x},0)(u).
		\end{equation*} Then the claimed equivalence between the metric
		regularity properties follows directly from the Mordukhovich
		coderivative criterion \eqref{cod-cr} applied to both mappings
		$\partial\ph$ and $\nabla e_\lambda\varphi$.
		
		The equivalence between the semismoothness$^*$ of $\partial\varphi$
		at $(\ox,0)$  and of $\nabla e_\lambda\varphi$ around this point is
		verified in the proof of Theorem~6.2 in Mordukhovich and Sarabi
		\cite{BorisEbrahim}. Thus we can apply the above
		Theorem~\ref{localconverge} to the gradient system $\nabla
		e_\lambda\varphi(x)=0$ and complete the proof of this theorem by
		using Lemma~\ref{rela}. 
	\end{proof} 
	
	\medskip
	Next we present a consequence of Theorem~\ref{localNM3} for {\em
		tilt-stable local minimizers}.\vspace*{-0.05in}
	
	\begin{Corollary}[\bf computing tilt-stable minimizers of
		prox-regular functions]\label{Newtonfortilt} Let
		$\varphi\colon\R^n\to\oR$ be continuous prox-regular at $\ox$ for
		$\bar{v}:=0\in\partial\varphi(\bar{x})$ with $r>0$ taken from
		\eqref{prox}, and let $\ox$ be a tilt-stable local minimizer of
		$\varphi$. Assume that the subgradient mapping $\partial\varphi$ is
		semismooth$^*$ at $(\bar{x},\bar{v})$. Then whenever
		$\lambda\in(0,r^{-1})$ there exists a neighborhood $U$ of $\bar{x}$
		such that Algorithm~{\rm\ref{NM3}} is well-defined for all $x^0\in
		U$ and generates a sequence of iterates $\{x^k\}$, which
		superlinearly converges to $\bar{x}$ as $k\to\infty$.
	\end{Corollary}\vspace*{-0.1in} \begin{proof} Observe that, in the
		case of minimizing $\ph$ under consideration, the prox-boundedness
		assumption on $\ph$ imposed in Theorem~\ref{localNM3} can be
		dismissed without loss of generality. Indeed, we can always get this
		property by adding to $\ph$ the indicator function of some compact
		set containing a neighborhood of $\ox$, which makes $\ph$ to be
		prox-bounded without changing the local minimization. Furthermore,
		it follows from Drusvyatskiy and Lewis \cite[Theorem~3.3]{dl} that
		the mapping $\partial\ph$ is strongly metrically regular around
		$(\bar{x},\bar{v})$. Thus assumption (H2) holds automatically,
		{while (H1) follows from
			Theorem~\ref{solvability}(ii)}. The fulfillment of (H3) is assumed
		in this corollary, and so we deduce its conclusions from
		Theorem~\ref{localNM3}.
	\end{proof} 
	
	Note that the well-posedness and local superlinear convergence of
	the coderivative-based generalized Newton algorithm of computing
	tilt-stable local minimizers of the Moreau envelope $e_\lm\ph$ under
	the semismoothness$^*$ assumption on $\partial\ph$ has been recently
	obtained by Mordukhovich and Sarabi
	\cite[Theorem~6.2]{BorisEbrahim}. As follows from the discussions
	above, the latter algorithm is equivalent to Algorithm~\ref{NM3} in
	the case of tilt-stable local minimizers. However, the explicit form
	of Algorithm~\ref{NM3} seems to be more convenient for
	implementations, which is demonstrated below. Moreover, in
	Algorithm~\ref{NM3} we specify the area of starting points $x^0$ in
	Step~0 and also verify the choice of the prox-parameter $\lm$
	ensuring the best performance of the proposed
	algorithm.\vspace*{0.05in}
	
	Let us illustrate Algorithm~\ref{NM3} by the following example of
	solving the subgradient inclusion \eqref{subgra-inc} generated by a
	nonconvex, nonsmooth, and continuously prox-regular function
	$\ph\colon\R\to\oR$. This function is taken from Mordukhovich and
	Outrata \cite[Example~4.1]{BorisOutrata} and relates to the modeling
	of equilibria.\vspace*{-0.05in}
	
	\begin{Example}[\bf illustration of computing by
		Algorithm~\ref{NM3}] {\rm Consider the function \begin{equation*}
				\varphi(x):=|x|+\frac{1}{2}\big(\max\{x,0\}\big)^2-\frac{1}{2}\big(\max\{0,-x\}\big)^2+\delta_\Gamma(x),\quad
				x\in\R, \end{equation*} where $\dd_\Gamma$ is the indicator function
			of the set $\Gamma:=[-1,1] $. We can clearly write $\varphi$ in the
			form \begin{equation}\label{ph0}
				\varphi(x)=\vartheta(x)+\delta_\Gamma(x),\quad x\in\R,
			\end{equation} via the continuous, nonconvex, nonsmooth, piecewise
			quadratic function $\vt\colon\R\to\R$ given by \begin{equation*}
				\vartheta(x):=\begin{cases}
					-x-\frac{1}{2}x^2&\text{if}\quad x\in[-1,0],\\
					x+\frac{1}{2}x^2&\text{if}\quad x\in[0,1]. \end{cases}
			\end{equation*} Then we get by the direct calculations that
			\begin{equation*} \partial\varphi(x)=\begin{cases}
					(-\infty,0]&\text{if}\quad x=-1, \\
					\{-1-x\}&\text{if}\quad x\in(-1,0),\\
					[-1,1]&\text{if}\quad x=0,\\
					\{1 + x\}&\text{if}\quad x\in(0,1),\\
					[2,\infty)&\text{if}\quad x=1. \end{cases} \end{equation*} It is not
			hard to check that the mapping $\partial\varphi$ is strongly
			metrically regular around $(\bar{x},0)$ with $\ox=0$ and
			semismooth$^*$ at this point, and thus the assumptions of
			Theorem~\ref{localNM3} hold. Moreover, we get \begin{equation*}
				\varphi(x)\ge\varphi(u)+v(x-u)-\frac{1}{2}\|x-u\|^2\;\mbox{ for all
				}\;(u,v)\in\gph\partial\ph\cap(U\times\R), x \in U \end{equation*} where
			$U=[-1,1]$, which implies that $\varphi$ is continuously prox-regular on  $U$ with modulus $r=1$. Choosing $\lambda:=\frac{1}{2}\in(0,1)$ and
			$x^0:=\frac{1}{3}\in{\rm rge}(I+\frac{1}{2}\partial\ph)$, we run
			Algorithm~\ref{NM3} with \begin{equation*}
				\text{\rm Prox}_\lambda\varphi(x^0)=\text{\rm argmin}_{y\in\R}\big\{\varphi(y)+(y-x^0)^{2}\big\}=0.
			\end{equation*} Thus the vector $v^0$ in Step~1 of the algorithm is
			calculated by
			$v^0=\frac{1}{\lambda}(x^0-P_\lambda\varphi(x^0))=\frac{2}{3}$. To
			find $d^0\in\R$ in Step~2 of the algorithm, we have by
			\eqref{prox-dir1} that \begin{equation*} (-v^0,-\lambda v^0-d^0)\in
				N_{\text{gph}\,\partial\varphi}(x^0-\lambda
				v^0,v^0)=N_{\text{gph}\,\partial\varphi}\left(0,2/3\right).
			\end{equation*} The second-order calculations in Mordukhovich and
			Outrata [41, Equation (4.7)] yield \begin{equation*}
				N_{\text{gph}\,\partial\varphi}\left(0,2/3\right)=\big\{(\omega,z)\in\R^2\;\big|\;z=0\big\}.
			\end{equation*} This tells us that $-\lambda v^0-d^0=0$, and hence
			$d^0=-\frac{1}{3}$. Letting $x^1:=x^0+d^0=0$ by Step~3, we arrive at
			$0\in\partial\varphi(x^1)$, and thus solve the subgradient inclusion
			\eqref{subgra-inc} for  $\varphi$ from \eqref{ph0}.}
	\end{Example}\vspace*{-0.05in}
	
	We conclude this section with the brief discussion on some other
	Newton-type methods to solve set-valued inclusions involving the
	subgradient systems \eqref{subgra-inc}.\vspace*{-0.05in}
	
	\begin{Remark}[\bf comparison with other Newton-type methods to solve generalized equations]\label{comp-inc} {\rm Let us make the following observations:\\[1ex]
			{\bf(i)} Josephy \cite{josephy} was the first to propose a
			Newton-type algorithm to solve {\em generalized equations} in the
			sense of Robinson \cite{rob} written in the form
			\begin{equation}\label{ge} 0\in f(x)+F(x), \end{equation} where
			$f\colon\R^n\to\R^n$ is a smooth single-valued mapping, and where
			$F\colon\R^n\tto\R^n$ is a set-valued one, which originally was
			considered as the normal cone to a convex set while covering in this
			case classical variational inequalities and nonlinear
			complementarity problems. The Josephy-Newton method constructs the
			next iterate $x^{k+1}$ by solving the linearized generalized
			equation \begin{equation}\label{lin-ge} 0\in f(x^k)+\nabla
				f(x^k)(x-x^k)+F(x) \end{equation}
			with the same set-valued part $F(x)$ as in \eqref{ge}. Note that in this method we must have a nonzero function $f$ in \eqref{ge}; otherwise algorithm \eqref{lin-ge} stops at $x^1$. There are several results on the well-posedness and superlinear convergence of the Josephy-Newton method under appropriate assumptions; see, e.g., the books by Facchinei and Pang \cite{JPang} and by Izmailov and Solodov \cite{Solo14} with the references therein. The major assumption for \eqref{ge}, which is the most related to our paper, is the strong metric regularity of $f+F$ around $(\ox,0)$ imposed in Dontchev and Rockafellar \cite[Theorem~6C.1]{Donchev09}. In the case of subgradient systems \eqref{subgra-inc}, we actually have $f=\nabla\ph$ and $F=0$ in \eqref{ge}, which gives us the strong metric regularity of $\nabla\ph$. As discussed in Sections~\ref{sec:solvN} and \ref{sec:newtonC11}, the latter assumption is more restrictive that our standing ones imposed in Theorem~\ref{localconverge}. A similar strong metric regularity assumption is imposed in the semismooth Newton method of solving generalized equations (without changing $F$ in iterations) presented in Theorem~6F.1 of the aforementioned book by Dontchev and Rockafellar.\\[1ex]
			{\bf(ii)} Gfrerer and Outrata \cite{Helmut} have recently introduced the new {\em semismooth$^*$ Newton method} to solve the generalized equation $0\in F(x)$. In contrast to the Josephy-Newton and semismooth Newton methods for generalized equations, the new method approximates the set-valued part $F$ of \eqref{ge} by using a certain linear structure inside the limiting coderivative graph under the nonsingularity of matrices from the mentioned linear structure. Local superlinear convergence of the proposed algorithm is proved there under the additional semismooth$^*$ assumptions on the mapping $F$.\\[1ex]
			{\bf(iii)} Another Newton-type method to solve the inclusions $0\in
			F(x)$, with the verification of well-posedness and local superlinear
			convergence, was developed by Dias and Smirnov \cite{ds} based on
			{\em tangential approximations} of the graph of $F$. The major
			assumption therein is the metric regularity of $F$ and the main tool
			of analysis is the Mordukhovich criterion \eqref{cod-cr}. Observe
			that the well-posedness result of the suggested algorithm requires
			the Lipschitz continuity of $F$, which is rarely the case for
			subgradient mappings associated with nonsmooth functions.}
	\end{Remark} \vspace*{-0.2in}

\section{Generalized Newton Algorithm in Composite
Optimization}\label{sec:Newcomposite}\vspace*{-0.05in}\setcounter{equation}{0}

In this section we develop a generalized Newton algorithm to solve
	the subgradient inclusion $0\in\partial\varphi(x)$ with the cost
	function $\varphi:\R^n\to\oR$ given in the {\em composite
		optimization form} by \begin{equation}\label{noncvcomposite} \min \;
		\varphi(x):= f(x) + g(x) \quad \text{subject to }\; x \in \R^n,
	\end{equation} where the function $f:\R^n\to \R$ is
	$\mathcal{C}^2$-smooth while the {\em regularizer} $g:\R^n \to
	\overline{\R}$ is continuously prox-regular. Since the function
	$\ph$ in \eqref{noncvcomposite} is also continuously prox-regular,
	we can solve this problem by using the generalized Newton algorithm
	developed in Section~\ref{sec:prox}. However, one of the
	difficulties in implementing Algorithm~\ref{NM3} is to compute the
	proximal mapping of $\varphi$, which is the sum of two functions.
	This motivates us to design a generalized Newton algorithm to solve
	the subgradient inclusion $0\in\partial\varphi(x)$ associated with
	the composite optimization problem \eqref{noncvcomposite} by
	calculating the proximal mapping {\em only} for the regularizer $g$
	that is often easy to compute.
	
	To proceed, we employ the construction of {\em forward-backward
		envelopes} introduced by Patrinos and Bemporad \cite{pb} for convex
	composite functions and then largely used to design numerical
	methods in various composite frameworks of optimization; see, e.g.,
	Themelis et al. \cite{tsp} with the references therein. This leads
	us to developing a well-defined and implementable generalized Newton
	algorithm expressed in terms of the second-order subdifferential of
	the regularizer $g$ and the single-valued, monotone, and Lipschitz
	continuous proximal mapping associated with this function. We show
	that the proposed algorithm exhibits local superlinear convergence
	under the standing assumptions formulated in
	Section~\ref{sec:newtonC11}. \vspace*{-0.05in}
	
	\begin{Definition}[\bf forward-backward envelopes] Let
		$\varphi:=f+g$, where $f:\R^n\to \R$ is $\mathcal{C}^1$-smooth, and
		where $g:\R^n\to\overline{\R}$ is a proper l.s.c.\ function. The
		{\sc forward-backward envelope} $($FBE$)$ of $\varphi$ with the
		parameter value $\gamma>0$ is given by \begin{equation}\label{FBE}
			\varphi_\gamma(x):= \inf_{y \in \R^n} \left\{f(x) + \langle \nabla
			f(x),y -x\rangle + g(y) + \frac{1}{2\gamma}\|y-x\|^2 \right\}.
		\end{equation} 
	\end{Definition}\vspace*{-0.05in}
	By definition \eqref{Moreau} of the Moreau
	envelope, \eqref{FBE} is equivalent to \begin{equation}\label{FBE2}
		\varphi_\gamma(x)= f(x) - \frac{\gamma}{2}\|\nabla f(x)\|^2+
		e_\gamma g\big(x-\gamma\nabla f(x)\big). 
	\end{equation}
	Now we derive several properties of FBE  in our setting, that are of independent interest while being instrumental to design the
	aforementioned generalized Newton algorithm for
	\eqref{noncvcomposite} and justify its well-definiteness and
	superlinear convergence.\vspace*{0.03in}
	
	The first proposition verifies the ${\cal C}^1$-smoothness of FBE
	with computing its gradient and verify the desired properties of the
	prox-gradient mapping associated with
	\eqref{noncvcomposite}.\vspace*{-0.05in}
	
	\begin{Proposition}[\bf smoothness of FBE and related properties]
		\label{FBEdiff} Let $f$ be of class $\mathcal{C}^2$ around $\bar{x}$
		where $0\in\partial\varphi(\bar{x})$, and let $g$ be prox-bounded
		and continuously prox-regular at $\ox$ for $-\nabla f(\ox)$ with
		modulus $r>0$. Then for all $\gamma \in (0, r^{-1})$ there exists a
		neighborhood $U$ of $\ox$ on which the following properties hold:
		\begin{itemize} \item[\bf (i)] The mapping $x\mapsto
			\text{\rm Prox}_{\gamma g}(x-\gamma \nabla f(x))$ is single-valued
			and  Lipschitz continuous. \item[\bf (ii)] The FBE $\varphi_\gamma$
			is continuously differentiable and its gradient mapping is computed by \begin{equation}\label{gradFBE}
				\nabla\varphi_\gamma (x) = \gamma^{-1}\big(I-\gamma \nabla^2
				f(x)\big)\big(x- \text{\rm Prox}_{\gamma g}(x-\gamma \nabla
				f(x))\big). \end{equation} \end{itemize} 
		Moreover, if $I -\gamma
		\nabla^2 f(x)$  is nonsingular for  $x \in U$, we have the equivalence
		\begin{equation}\label{relationcomposite} \nabla
			\varphi_\gamma(x) = 0 \iff 0 \in \partial \varphi(x).
		\end{equation}
		If in addition the regularizer $g$ is
		convex, then $\gamma$ can be chosen in $(0,\infty)$.
	\end{Proposition} \begin{proof} By Lemma~\ref{rela}
		and the continuous prox-regularity of $g$, for each $\gamma\in
		(0,r^{-1})$ there is a neighborhood $U_\gamma$ of $\ox -
		\gamma\nabla f(\ox)$ such that the proximal mapping $\text{\rm
			Prox}_{\gamma g}$ is single-valued and Lipschitz continuous on
		$U_\gamma$. Since $f$ is $\mathcal{C}^2$-smooth around $\ox$, there is a neighborhood  $U$ of $\ox$ ensuring the Lipschitz continuous of the mapping $x \mapsto x -\gamma \nabla f(x)$ on $U$ and the fulfillment of the relationships
		$$
		x \in U \Longrightarrow x-\gamma\nabla f(x) \in U_\gamma,
		$$
		which therefore verifies the single-valuedness of the prox-gradient mapping
		$x\mapsto \text{\rm Prox}_{\gamma g}(x-\gamma \nabla f(x))$ on $U$. The Lipschitz continuity of this mapping follows from
		the Lipschitz continuity of the proximal mapping $\text{\rm Prox}_{\gamma g}$ on
		$U_\gamma$ and of the mapping $x \mapsto x -\gamma \nabla f(x)$ on $U$. This verifies assertion (i).
		
		The gradient calculation in \eqref{gradFBE} of assertion (ii) follows directly from \eqref{FBE2} and Lemma~\ref{rela}(iii). It remains to verify implication \eqref{relationcomposite} if $I -\gamma \nabla^2 f(x)$ is nonsingular for  $x \in U$. Indeed, due to the nonsingular of $I -\gamma \nabla^2 f(x)$, the gradient equation $\nabla \varphi_\gamma (x) = 0$ is equivalent to $x=
		\text{\rm Prox}_{\gamma g}(x-\gamma \nabla f(x))$. The latter is
		equivalent to the fact that $x = (I+\gamma \partial g)^{-1}(x-\gamma
		\nabla f(x))$ by Lemma~\ref{rela}. This means that $0 \in
		\nabla f(x) + \partial g(x)$. Using finally the limiting subdifferential sum
		rule with equalities from Mordukhovich \cite[Proposition~1.107]{Mordukhovich06}, we
		obtain the inclusion $0\in \partial\varphi(x)$. In the  case where $g$ is convex, we proceed similarly to the above to justifying that $\gamma$ can be arbitrary chosen from $(0,\infty)$,  and thus we complete the proof of this proposition.\end{proof} 
	
	Proposition~\ref{FBEdiff} tells us that using the forward-backward envelope \eqref{FBE} makes it possible to pass from the nonsmooth composite problem \eqref{noncvcomposite} to the
	unconstrained one \begin{equation}\label{FBEop} \min \quad
		\varphi_\gamma(x)\quad \text{subject to } \;x\in\R^n \end{equation}
	with a smooth cost function. Due to the explicit calculations of
	$\varphi_\gamma$ in \eqref{FBE2} and its gradient \eqref{gradFBE},
	we can extend Algorithm~\ref{NM}  to cover problem
	\eqref{noncvcomposite} by reducing it to \eqref{FBEop}. The
	implementation of this procedure requires revealing appropriate
	assumptions on $\varphi$ in \eqref{noncvcomposite}, which ensure the
	fulfillment of those for $\ph_\gg$ and thus allow us to apply the
	results of Section~\ref{sec:newtonC11} to the reduced problem \eqref{FBEop}.
	
	Note that \eqref{FBEop} is generally not a problem of ${\cal
		C}^{1,1}$ optimization, since Proposition~\ref{FBEdiff} does not guarantee the Lipschitz continuity of $\nabla\ph_\gg$. Nevertheless, in the case where $f$ is quadratic, the ${\cal C}^{1,1}$ property of $\ph_\gg$ follows from Proposition~\ref{FBEdiff} by formula \eqref{gradFBE}.
	From now on, we consider the problem \begin{equation}\label{QP} \text{minimize
		}\;\varphi(x):=\frac{1}{2}\langle Ax,x\rangle+\langle
		b,x\rangle+\alpha+g(x),\quad x\in\R^n, \end{equation} where
	$A\in\R^{n\times n}$ is a symmetric matrix, $b\in\R^n$,
	$\alpha\in\R$, and the regularizer $g:\R^n\to \overline{\R}$ is continuously
	prox-regular at $\ox\in\R^n$ for $-A\ox-b$ with $\ox$ satisfying the stationary condition $0 \in
	\partial\varphi(\ox)$. Clearly, \eqref{QP} is a subclass of the optimization problem \eqref{noncvcomposite} with $f(x):= \frac{1}{2}\langle Ax,x\rangle +\langle b,x\rangle +\alpha$.
	
	Let us highlight that problems of type \eqref{QP} are important in their own right, while they also arise frequently as {\em subproblems} in various numerical algorithms including
	{\em sequential quadratic programming methods} (SQP)
	\cite{Bonnans,Solo14}, {\em proximal Newton methods}
	\cite{kanzow21,lss,myzz}, etc. Observe furthermore that optimization problems
	of this type often appear in practical models related, e.g., to
	machine learning and statistics.\vspace*{0.03in}
	
	Now we start the procedure of designing and justifying a locally
	convergent generalized Newton algorithm to solve the inclusion $0
	\in \partial\varphi(x)$, where $\varphi$ is the cost function in the
	composite problem \eqref{QP}, by applying the corresponding results
	for the ${\cal C}^{1,1}$ gradient systems $\nabla\varphi_\gamma(x)
	=0$ obtained in Section~\ref{sec:newtonC11}. The first step is to
	express the generalized Hessian of the FBE $\ph_\gg$ from
	\eqref{FBEop} in terms of the given data of \eqref{QP}.\vspace*{-0.03in}
	
	\begin{Proposition}[\bf calculating the generalized Hessian of FBE]
		\label{calculatepsi} Let $\varphi=f+g$ be as in \eqref{QP} with a prox-bounded and continuously prox-regular $g\colon\R^n\to\oR$, and let $\gamma \in (0,r^{-1})$ be such that the matrix $B:=I -\gamma A$ is nonsingular, where $r>0$ is a constant of
		prox-regularity  of $g$. Then there exists a neighborhood $U$ of
		$\ox$ on which we have the second-order calculation formula
		\begin{equation}
			z \in \partial^2\varphi_\gamma (x)(w) \iff B^{-1} z - Aw \in \partial^2 g \left(\text{\rm Prox}_{\gamma g}(u),\frac{1}{\gamma}\big(u -\text{\rm Prox}_{\gamma g}(u)\big) \right)\big(w-\gamma B^{-1}z\big)
		\end{equation}
		for any $x \in U$, $w \in \R^n$, and $u:= x- \gamma(Ax+b)$.
	\end{Proposition}
	\begin{proof} Using the Moreau envelope, define the function $h:\R^n\to\overline{\R}$ by
		$$
		h(x):= e_{\gamma}g\big(x-\gamma(Ax+b)\big) \quad \text{for all }\; x \in \R^n.
		$$
		Due to Proposition~\ref{FBEdiff}, there exists a neighborhood $U$ of $\ox$ such that $\varphi_\gamma$ is $\mathcal{C}^1$-smooth and that the mapping $x \mapsto \text{\rm Prox}_{\gamma g}(x-\gamma(Ax+b))$ is continuous on $U$, which implies that $h$ is continuous differentiable around $\ox$ due to \eqref{FBE2}. Take any $x \in U$ and put $u:=x-\gamma (Ax+b)$.
		It follows from the standard chain rule that
		$$
		\nabla h(x) = (I-\gamma A)^* \nabla e_\gamma g (u) = B\nabla e_\gamma g(u).
		$$
		Using \eqref{FBE2} and the second-order sum rule from Mordukhovich \cite[Proposition~1.121]{Mordukhovich06}, for each $w\in\R^n$ we have
		\begin{equation}\label{cal1}
			\partial^2\varphi_\gamma(x)(w) = (A-\gamma A^*A)w + \partial^2 h(x)(w) = BAw + \partial^2 h(x)(w).
		\end{equation}
		Employing another second-order chain rule taken from \cite[Theorem~1.127]{Mordukhovich06} yields
		\begin{equation}\label{cal2}
			\partial^2h(x)(w)  = B\partial^2 e_\gamma g(u)(Bw)\;\mbox{ for all }\;w\in\R^n.
		\end{equation}
		Combining \eqref{cal1} and \eqref{cal2} brings us to the formula
		$$
		\partial^2\varphi_\gamma(x)(w) = BAw + B\partial^2 e_\gamma g(u)(Bw)
		$$
		valid for the same $w$. This verifies the equivalences
		$$
		z \in \partial^2\varphi_\gamma(x)(w) \iff z - BAw \in B\partial^2e_\gamma g(u)(Bw) \iff B^{-1}z - Aw \in \partial^2e_\gamma g(u)(Bw)
		$$
		for all $w\in \R^n$. Using finally Lemma~\ref{2gph}, we arrive at the inclusion
		$$
		B^{-1}z - Aw \in \partial^2 g \left(\text{\rm Prox}_{\gamma g}(u), \frac{1}{\gamma}\big(u-\text{\rm Prox}_{\gamma g}(u)\big) \right)(Bw - \gamma B^{-1}z +\gamma Aw)
		$$
		and thus complete the proof of the proposition due to $Bw + \gamma Aw = w$.
	\end{proof}
	
	The next proposition shows that the metric regularity and tilt stability of the original cost function $\ph$ in \eqref{QP} is equivalent to the corresponding properties of its FBE $\ph_\gamma$ in \eqref{FBEop}.\vspace*{-0.05in}
	
	\begin{Proposition}[\bf metric
		regularity and tilt-stability of FBE]\label{metrictilt} Let in the setting of Proposition~{\rm\ref{calculatepsi}} $\ox$ be such that
		$0\in\partial\varphi(\bar{x})$. Then we have the following assertions:
		\begin{itemize} 
			\item[\bf (i)] $\partial\varphi$ is metrically regular around $(\bar{x},0)$  if and only if $\nabla\varphi_\gamma$ is metrically regular
			around $\ox$.
			\item[\bf (ii)] If $\bar{x}$ is a tilt-stable local minimizer of $\varphi$, then $\bar{x}$ is a tilt-stable local minimizer of $\varphi_\gamma$ provided that the matrix $B=I-\gamma A$ is positive definite.
		\end{itemize}
	\end{Proposition}
	\begin{proof} It follows from Proposition~\ref{calculatepsi} that
		\begin{equation}\label{equi}
			z \in \partial^2\varphi_\gamma(\bar{x})(w) \iff B^{-1}z \in Aw + \partial^2 g\left(\text{\rm Prox}_{\gamma g}(\bar{u}),\frac{1}{\gamma}\big(\bar{u} -\text{\rm Prox}_{\gamma g}(\bar{u})\big) \right)\big(w-\gamma B^{-1}z\big)
		\end{equation}
		with $\bar{u}$ defined therein. Proposition~\ref{FBEdiff} gives us $\nabla\varphi_\gamma(\ox)=0$, which means that
		$$
		\gamma^{-1}\big(I-\gamma A\big)\big(x- \text{\rm Prox}_{\gamma g}(x-\gamma \nabla f(x))\big) =0
		$$
		and implies that $\ox =  \text{\rm Prox}_{\gamma g}(\ox-\gamma \nabla f(\ox))$ since $B = I-\gamma A$ is nonsingular. This tells us that $\bar{x} - \text{\rm Prox}_{\gamma g}(\bar{u}) =0$, and thus \eqref{equi} is equivalent to
		\begin{eqnarray}\label{equa1}
			B^{-1}z &\in& Aw + \partial^2 g(\bar{x}, -A\bar{x}-b)(w-\gamma B^{-1}z) \nonumber\\
			&=& A(w-\gamma B^{-1}z) + \partial^2 g(\bar{x}, -A\bar{x}-b)(w-\gamma B^{-1}z) + \gamma AB^{-1}z. 
		\end{eqnarray}
		The second-order subdifferential sum rule from Mordukhovich  \cite[Proposition~1.121]{Mordukhovich06} yields
		\begin{eqnarray}\label{equa2}
			A(w-\gamma B^{-1}z) + \partial^2 g(\bar{x},-A\bar{x}-b)(w-\gamma B^{-1}z)&=&\partial^2(f+ g)(\bar{x},0)(w-\gamma B^{-1}z)\nonumber \\
			&=&\partial^2\varphi(\bar{x},0)(w-\gamma B^{-1}z).
		\end{eqnarray}
		
		Combining \eqref{equi}, \eqref{equa1}, and \eqref{equa2}, we arrive at the equivalence
		\begin{equation}\label{equiFBE}
			\begin{array}{ll}
				z\in\partial^2\varphi_\gamma(\bar{x})(w)&\iff B^{-1}z - \gamma AB^{-1}z \in \partial^2\varphi(\bar{x},0)(w-\gamma B^{-1}z)\\&\iff  z \in \partial^2\varphi(\bar{x},0)(w-\gamma B^{-1}z).
			\end{array}
		\end{equation}
		It follows from the coderivative  criterion \eqref{Morcri} and the equivalence in \eqref{equiFBE} that $\partial\varphi$ is metrically regular around $(\bar{x},0)$  if and only if $\nabla\varphi_\gamma$ is metrically regular around $\ox$, which justifies assertion {\bf(i)}.\vspace*{0.05in}
		
		To clarify assertion (ii), suppose that $\bar{x}$ is a tilt-stable local minimizer of $\varphi$, which also implies that $\partial \varphi$ is metrically regular around $(\ox,0)$. Let $w \in \R^n, w \ne 0$ and $z \in \partial^2\varphi_\gamma(\ox)(w)$. By \eqref{equiFBE}, we deduce that $z \in \partial^2\varphi(\ox,0)(w-\gamma B^{-1}z)$, which ensures that $z \ne 0$ by the coderivative  criterion \eqref{Morcri}. It follows from the pointbased second-order characterization of tilt-stability by Poliquin and Rockafellar \cite[Theorem 1.3]{Poli} that $\langle z , w - \gamma B^{-1}z\rangle \geq 0$, and thus
		\begin{equation} \label{zwPD}
			\langle z , w  \rangle \geq  \gamma \langle B^{-1}z, z\rangle>0. 
		\end{equation}
		due to the positive definiteness of $B^{-1}$, which tells us that $\ox$ is a tilt-stable local minimizer of $\varphi_\gamma$ due to \cite[Theorem 1.3]{Poli}. This completes the proof  of the proposition.
	\end{proof} 
	
	\medskip 
	We need another proposition before formulating our generalized Newton algorithm to solve the composite optimization problem \eqref{QP} and justifying its superlinear convergence via the given data.\vspace*{-0.05in}
	
	\begin{Proposition}[\bf semismoothness$^*$ of FBE
		derivatives]\label{semidirect} In the setting of Proposition~{\rm\ref{metrictilt}}, we claim that the semismooth$^*$ property of $\partial g$ at $(\bar{x},\ov)$ with $\ov:=-A\bar{x}-b$  yields the semismoothness$^*$ of $\nabla\varphi_\gamma$ at $\bar{x}$.
	\end{Proposition}
	\begin{proof} Denote $h_\gamma(x):=\text{\rm Prox}_{\gamma g}(x-\gamma(Ax+b))$ on $\R^n$ and get by Proposition~\ref{FBEdiff} that
		\begin{equation}\label{FBEh}
			\nabla\varphi_\gamma(x) = \gamma^{-1}(I-\gamma A)\big(x-h_\gamma(x)\big) = \gamma^{-1}Bx-\gamma^{-1}Bh_\gamma(x)\;\mbox{ for all }\;x\in \R^n.
		\end{equation}
		Since $0 \in \partial\varphi(\ox)$, it follows from Proposition~\ref{FBEdiff} that $\nabla \varphi_\gamma(\ox)=0$. 
		Due to \eqref{FBEh} and the nonsingularity of $B =  I - \gamma A$, we have $\bar{x}= h_\gamma(\bar{x})$. Furthermore, it follows from  Lemma~\ref{rela}, we have
		$$
		\text{\rm Prox}_{\gamma g}(x) = (I+ \gamma \partial g)^{-1}(x) \quad \text{for all }\; x \; \text{near } \ox.
		$$
		The semismoothness$^*$ of $\partial g$ at $(\bar{x},\ov)$ yields this property for $I+ \gamma \partial g$ at $(\ox, \ox + \gamma \ov)$ due to Gfrerer and Outrata \cite[Proposition 3.6]{Helmut}, and hence $(I+ \gamma \partial g)^{-1}$ is semismooth$^*$ at $\ox + \gamma \ov$ (see \cite[p.~496]{Helmut}) meaning that $\text{\rm Prox}_{\gamma g}$ is semismooth$^*$ at $\bar{x}-\gamma(A\bar{x}+b)$. It follows therefore from Khanh et al. \cite[Lemma~7.3]{kmpt221}  that $h_\gamma$ is semismooth$^*$ at $\bar{x}$. Employing finally \eqref{FBEh} together with \cite[Proposition 3.6]{Helmut} tells us that the gradient mapping $\nabla\varphi_\gamma$ is semismooth$^*$ at $\bar{x}$.
	\end{proof}\vspace*{0.05in}
	
	The proposed generalized Newton algorithm to solve problem \eqref{QP} is formulated as follows.\vspace*{-0.05in}
	
	\begin{Algorithm}[\bf Newton-type algorithm in nonconvex composite optimization]\label{NM4} {\rm Let $\ph$ be taken from \eqref{QP}, and let $r>0$ be a constant of prox-regularity of $g\colon\R^n\to\oR$.\\[1ex]
			{\bf Step~0:} Pick any $\gamma\in(0, r^{-1})$ such that $I-\gamma A$  is nonsingular, choose a starting point $x^0$, and set $k:=0$.\\[1ex]
			{\bf Step~1:} If $0\in\partial\varphi(x^k)$, then stop. Otherwise compute
			\begin{equation}\label{subprob1}
				u^k:= x^k - \gamma(Ax^k+b), \; v^k:=\text{\rm Prox}_{\gamma g}(u^k).
			\end{equation}
			{\bf Step~2}: Choose $d^k\in\R^n$ such that
			\begin{equation}\label{prox-dir3}
				-\frac{1}{\gamma}(x^k-v^k)-Ad^k \in\partial^2 g\Big(v^k, \frac{1}{\gamma}(u^k-v^k)\Big)(x^k-v^k + d^k).
			\end{equation}
			{\bf Step~3:} Compute $x^{k+1}$ by
			\begin{equation*}
				x^{k+1}:=x^k+d^k,\quad k=0,1,\ldots.
			\end{equation*}
			{\bf Step 4:} Increase $k$ by $1$ and go to Step~1.}
	\end{Algorithm}\vspace*{-0.15in}
	
	\begin{Remark}[\bf on the iterative sequence generated by Algorithm \ref{NM4}] \label{NMandNM4} \rm  Propositions~\ref{FBEdiff} and \ref{calculatepsi}  allow us to show that Algorithm \ref{NM4} is a special case of the general scheme given in Algorithm \ref{NM}. Indeed, by Proposition~\ref{FBEdiff} and the constructions of $\{u^k\}$ and $\{v^k\}$ in Algorithm \ref{NM4}, we have 
		$$
		\nabla \varphi_\gamma(x^k) = \frac{1}{\gamma}B (x^k - v^k), \quad \text{where }\; B:= I-\gamma A.   
		$$
		Therefore, inclusion \eqref{prox-dir3} can be equivalently rewritten as
		$$
		-B^{-1}\nabla \varphi_\gamma (x^k)   - Ad^k \in \partial^2 g\left(\text{\rm Prox}_{\gamma g}(u^k), \frac{1}{\gamma}\big(u^k-\text{\rm Prox}_{\gamma g}(u^k)\big) \right)\big(d^k +\gamma B^{-1}\nabla \varphi_\gamma(x^k)\big),
		$$
		which means that $-\nabla\varphi_\gamma(x^k)\in \partial^2\varphi_\gamma (x^k)(d^k) = \partial \langle d^k, \nabla \varphi_\gamma\rangle (x^k)$ by using Proposition \ref{calculatepsi} and \eqref{scal}. 
		Therefore, Algorithm \ref{NM4} reduces to Algorithm \ref{NM} with $\varphi:=\varphi_\gamma$.  
	\end{Remark}
	
	The next theorem justifies the existence and local superlinear convergence of iterates in Algorithm~\ref{NM4} under basically the standing assumptions of this paper. Note that we do not assume the convexity of the regularizer $g$ as in problems of convex composite optimization.\vspace*{-0.05in}
	
	\begin{Theorem}{\bf(local superlinear convergence of
			Algorithm~\ref{NM4}).}\label{localNM4} Let $\varphi=f+g$ be as in
		\eqref{QP}, where the regularizer $g\colon\R^n\to\oR$ is continuously prox-regular at the stationary point $\ox$ for $-A\ox-b$. In addition to the standing assumption {\rm(H1)--(H3)} with $\ov=0$, suppose that $g$ is prox-bounded on $\R^n$. Then there exists a neighborhood $U$ of $\ox$ such
		that for any starting points $x^0\in U$ we have that
		Algorithm~{\rm\ref{NM4}} is well-defined and generates a sequence of
		iterates $\{x^k\}$, which converges $Q$-superlinearly to the solution
		$\bar{x}$ of the subgradient inclusion $0\in \partial\varphi(x)$. 
	\end{Theorem}
	\begin{proof} It follows from Proposition~\ref{FBEdiff} that
		solving the gradient system $\nabla \varphi_\gamma(x)=0$ gives us a solution to the original subgradient inclusion $0\in \partial\varphi(x)$ under the
		indicated choice of parameters.Furthermore, Remark~\ref{NMandNM4} gives us an opportunity to reduce solving the inclusion $0\in\partial\varphi(x)$ by Algorithm~\ref{NM4} to solving
		the equation $\nabla\varphi(x)=0$ for the function $\ph_\gg$ of class ${\cal C}^{1,1}$ by using Algorithm~\ref{NM}. To justify this reduction, we need to check that assumptions (H1)--(H3) imposed in
		this theorem on $\partial\ph$ implies the fulfillment of the corresponding assumptions
		on $\nabla\varphi_\gamma$ imposed in Theorem~\ref{localconverge}.
		
		To proceed, observe that Propositions~\ref{metrictilt} and \ref{semidirect} show that assumptions (H2) and (H3) of this theorem yield the fulfillment of their counterparts for $\varphi_\gamma$. Pick any $\gamma\in(0, r^{-1})$ such that $I-\gamma A$ is nonsingular. Since $f$ is a quadratic function, there exists a neighborhood $U_\gamma$ of $\ox$ on which the prox-gradient mapping $x\mapsto \text{\rm
			Prox}_{\gamma g}(x-\gamma\nabla f(x))$ is single-valued and Lipschitz
		continuous, while $\varphi_\gamma$ is of class $\mathcal{C}^{1,1}$ on
		$U_\gamma$ by Proposition~\ref{FBEdiff}. Now it suffices to find a neighborhood $\Tilde U$ of $\ox$ such that for each $x\in\Tilde U$ there exists $\Tilde{d}\in\R^n$ satisfying the second-order inclusion
		\begin{equation}\label{solforenvecomposite} -\nabla
			\varphi_\gamma(x)\in \partial^2\varphi_\gamma (x)(\Tilde{d}).
		\end{equation} 
		Assumption (H1) tells us that there are neighborhoods $U$ of $\ox$ and $V$ of $\ov=0$ such that for every $(x,v)\in\gph\partial\varphi\cap(U\times V)$ there exists a
		direction $d\in\R^n$ satisfying \eqref{newton-inc}. Consider the single-valued
		mapping $T: U \cap U_\gamma\to \R^n\times \R^n$ defined by $$ T(x):=
		\Big(\text{\rm Prox}_{\gamma g}(u),A\text{\rm Prox}_{\gamma g}(u)
		+b + \frac{1}{\gamma}\big(u-\text{\rm Prox}_{\gamma g}(u)\big)\Big)
		\quad \text{for all }\; x \in U \cap U_\gamma, \; u:= x -
		\gamma(Ax+b). $$ Since $T$ is continuous with
		$T(\ox)=(\ox,0)$, there exists a neighborhood $\widetilde{U}$ of
		$\ox$ such that $T(\widetilde{U})\subset (U\cap U_\gamma) \times V$. For any  $x \in \widetilde{U}$ with $u:= x -\gamma(Ax+b)$ we have
		\begin{equation}\label{Proxginsub}
			\frac{1}{\gamma}\Big(u-\text{\rm Prox}_{\gamma
				g}(u)\Big)\in\partial g\big(\text{\rm Prox}_{\gamma g}(u)\big).   
		\end{equation}
		Applying the subdifferential sum rule from Mordukhovich \cite[Proposition~1.107]{Mordukhovich06} gives us 
		\begin{equation}\label{sumruleProxg}
			\partial\varphi(\text{\rm Prox}_{\gamma g}(u)) = A\,\text{\rm Prox}_{\gamma g}(u)
			+b + \partial g(\text{\rm Prox}_{\gamma g}(u)).
		\end{equation}
		Combining \eqref{Proxginsub} and \eqref{sumruleProxg}, we obtain the inclusion
		$$
		\Big(\text{\rm Prox}_{\gamma g}(u), A\text{\rm Prox}_{\gamma g}(u)
		+b + \frac{1}{\gamma}\big(u-\text{\rm Prox}_{\gamma g}(u)\big)\Big)
		\in \text{\rm gph}\partial\varphi\cap (U\times V). 
		$$ 
		Therefore, there exists a direction $d \in \R^n$ satisfying 
		\begin{equation} \label{complicated1}
			-A\text{\rm Prox}_{\gamma g}(u) - b - \frac{1}{\gamma}\big(u-\text{\rm
				Prox}_{\gamma g}(u)\big) \in \partial^2\varphi\Big(\text{\rm
				Prox}_{\gamma g}(u),A\text{\rm Prox}_{\gamma g}(u) +b +
			\frac{1}{\gamma}\big(u-\text{\rm Prox}_{\gamma g}(u)\big)\Big)(d).
		\end{equation}
		Using the second-order subdifferential sum rule from the aforementioned book
		\cite[Proposition~1.121]{Mordukhovich06} leads us to the generalized Hessian representation
		\begin{equation*}
			\partial^2\varphi\Big(\text{\rm Prox}_{\gamma g}(u),A\text{\rm
				Prox}_{\gamma g}(u) +b + \frac{1}{\gamma}\big(u-\text{\rm Prox}_{\gamma
				g}(u)\big)\Big)(d) = Ad +\partial^2g\left(\text{\rm Prox}_{\gamma
				g}(u), \frac{1}{\gamma}\big(u-\text{\rm Prox}_{\gamma g}(u)\big)
			\right)(d). 
		\end{equation*} 
		Combining now the latter with \eqref{complicated1} leads us to
		\begin{equation} \label{solv}
			-A\text{\rm Prox}_{\gamma g}(u) - b - \frac{1}{\gamma}\big(u-\text{\rm
				Prox}_{\gamma g}(u)\big) - Ad \in \partial^2g\Big(\text{\rm
				Prox}_{\gamma g}(u), \frac{1}{\gamma}\big(u-\text{\rm Prox}_{\gamma
				g}(u)\big)\Big)(d). 
		\end{equation} 
		Moreover, we clearly have the equalities
		\begin{eqnarray}\label{proxgu}
			-A\,\text{\rm Prox}_{\gamma g}(u) - b - \frac{1}{\gamma}\big(u-\text{\rm
				Prox}_{\gamma g}(u)\big) &=& -A\,\text{\rm Prox}_{\gamma g}(u) - b - \frac{1}{\gamma}\big(x-\gamma (Ax-b) -\text{\rm
				Prox}_{\gamma g}(u)\big) \nonumber\\
			&=&  \gamma^{-1}(I-\gamma A)(\text{\rm Prox}_{\gamma g}(u) - x) = -\nabla \varphi_\gamma(x). 
		\end{eqnarray}
		Put $B:= I-\gamma A$ and
		$\widetilde{d}:= d -
		\gamma B^{-1}\nabla \varphi_\gamma (x)$. Using \eqref{solv} and \eqref{proxgu}, we deduce that
		$$
		-B^{-1}\nabla\varphi_\gamma (x) - A\widetilde{d} \in \partial^2 g
		\Big(\text{\rm Prox}_{\gamma g}(u),\frac{1}{\gamma}\big(u
		-\text{\rm Prox}_{\gamma g}(u)\big)\Big)\big(\widetilde{d}+\gamma
		B^{-1}\nabla\varphi_\gamma (x)\big), 
		$$ 
		which verifies \eqref{solforenvecomposite} due to Proposition \ref{calculatepsi}. This allows us to
		apply Theorem~\ref{localconverge} to the gradient equation $\nabla
		\varphi_\gamma(x)=0$ and thus completes the proof of this theorem.
	\end{proof} 
	
	\medskip 
	We have the following specification of Theorem~\ref{localNM4} in the case of of tilt-stable minimizers. 
	
	\begin{Corollary}[\bf computing tilt-stable minimizers in nonconvex composite optimization]\label{Newtonfortiltcomposite} Let $\ox$ be a tilt-stable local minimizer of the function $\ph$ in \eqref{QP}, where the regularizer $g\colon\R^n\to \oR$ is prox-bounded on $\R^n$ and continuously prox-regular at $\ox$ for $-A\ox-b$, and where the mapping $\partial g$ is semismooth$^*$ at $(\bar{x},-\nabla f(\ox))$ with $f(x):=\frac{1}{2}\la Ax,x\ra+\la b,x\ra+\al$.  Then there exists a
		neighborhood $U$ of $\bar{x}$ such that Algorithm~{\rm\ref{NM4}} is
		well-defined for all $x^0\in U$ and generates a sequence of iterates
		$\{x^k\}$, which $Q$-superlinearly converges to the local tilt-stable minimizer $\bar{x}$.
	\end{Corollary} 
	\begin{proof} By the tilt stability of $\ox$, the assumptions of Theorem~\ref{localNM4} are satisfied; see Remark~\ref{comp-tilt}.
	\end{proof} 
	
	\medskip 
	We conclude this section with two remarks. The first one concerns the nonsingularity assumption on the matrix $I-\gg A$ in Algorithm~\ref{NM4}.\vspace*{-0.05in}
	
	\begin{Remark}[\bf on nonsingularity of the shifted matrix in Algorithm~\ref{NM4}]\label{posit-dedf} {\rm Observe first that in this paper we have {\em never imposed the positive-definiteness} assumption on the generalized Hessian $\partial^2\ph$ of the cost functions. Furthermore, we do not assume that the quadratic matrix $A$ in the composite optimization problem \eqref{QP} is nonsingular. However, the nonsingularity of the shifted matrix $I-\gg A$ for some $\gg>0$ is required in the very construction of Algorithm~\ref{NM4}. Note that we can {\em always find} a number $\gg$ such that $I-\gamma A$ is positive definite, and thus it is also nonsingular due to the symmetry of the matrix $A$  and the openness
			of the set of positive definite matrices.}
	\end{Remark}\vspace*{-0.05in}
	
	The second remark discusses the comparison between our Algorithm~\ref{NM4} to solve \eqref{QP} and a version of the semismooth Newton method in composite optimization.\vspace*{-0.05in}
	
	\begin{Remark}[\bf comparison with the semismooth Newton method in composite optimization]\label{comp-inccomposite} \rm The thesis by Milzarek \cite{Milzarek} deals with the composite minimization problem \eqref{noncvcomposite}, where $f$ is twice continuously differentiable (possibly nonconvex), while the regularizer $g$ is a proper l.s.c.\ and convex function. The developed approach employs the semismooth Newton method to solve the equation
		\begin{equation*}
			F^{\Lambda}(x):=x-\text{\rm Prox}_{g}^\Lambda\big(x-\Lambda^{-1}\nabla f(x)\big)=0,
		\end{equation*}
		where $\Lambda$ is a symmetric, real, and positive-definite $n\times n$ matrix, and where
		\begin{equation}\label{proxi}
			\text{\rm Prox}_{g}^\Lambda(x) := {\rm argmin}\Big\{g(y)+\frac{1}{2\lambda} \langle \Lambda (x-y),x-y\rangle \;\Big|\;y\in\R^n\Big\}.
		\end{equation}
		The local superlinear convergence results are obtained in the thesis under the assumptions that the proximity operator \eqref{proxi} is semismooth and the generalized Jacobian of $F^{\Lambda}(x)$ is nonsingular. Meanwhile, our method addresses solving the subgradient inclusion $0
		\in \partial\varphi(x)$ with $\varphi$ taken from \eqref{QP}, where the regularizer $g$ is merely prox-regular under the metric regularity and semismooth$^*$ assumption on the subgradient mapping; see Remark~\ref{semialg} for the related discussions. \end{Remark} \vspace*{-0.2in}

\section{Applications to Regularized Least Square Problems}\label{lasso}\vspace*{-0.05in}\setcounter{equation}{0}

In this section we provide applications of our main Algorithm~\ref{NM4} to solving two practical models
	arsing in statistics, machine learning, image processing, etc. The first model is known as the {\em Lasso problem} (or as the {\em $\ell^1$-regularized least square optimization problem}). It was formulated by Tibshirani \cite{Tibshirani} for statistical applications. The second model we label here as the {\em $\ell^{1,2}$-regularized least square optimization problem}. Both of these problems can be described in the form 
	\begin{eqnarray}\label{regression}
		\text{minimize }\;\varphi(x):=\frac{1}{2}\|Ax-b\|_2^2+g(x),\quad
		x\in\R^n, 
	\end{eqnarray} 
	where $A$ is an $m\times n$ matrix, $b\in\R^m$, and the regularizer $g\colon\R^n\to\R$ is a nonsmooth function, which is convex in the first problem and nonconvex in the second one. Rewriting these problems in the composite optimization form \eqref{QP}, we'll constructively express in what follows all the elements and assumptions of Algorithm~\ref{NM4} entirely in terms of the given data for both problems under consideration.\vspace*{0.03in}
	
	Let is start with the {\em Lasso problem} formulated as
	\begin{eqnarray}\label{Lasso}
		\text{minimize }\;\varphi(x):=\frac{1}{2}\|Ax-b\|_2^2+\mu\|x\|_1,\quad x\in\R^n,
	\end{eqnarray}
	where $A$ is an $m\times n$ matrix, $\mu>0$, $b\in\R^m$, $\|\cdot\|_2$ and $\|\cdot\|_1$ stand for the standard Euclidean and sum norms, respectively. Note that \eqref{Lasso} is a subclass of
	\eqref{regression} with $g(x):=\mu\|x\|_1$ and
	\begin{equation}\label{fg}
		f(x):=\frac{1}{2}\langle\Tilde{A}x,x\rangle+\langle\Tilde{b},x\rangle+\Tilde{\alpha}\quad\text{with
		}\;\Tilde{A}:=A^*A,\;
		\Tilde{b}:=-A^*b,\;\mbox{ and }\;\Tilde{\alpha}:=\frac{1}{2}\|b\|^2, 
	\end{equation} 
	where the matrix $\Tilde{A}=A^*A$ is always {\em symmetric and
		positive-semidefinite}. The Lasso problem \eqref{Lasso} is convex and
	always admits an optimal solution, which is fully characterized by the
	subgradient inclusion 
	\begin{equation}\label{sumLasso}
		0\in\partial\varphi(x)\;\mbox{ with }\;\ph\;\mbox{ from
		}\;\eqref{Lasso}.
	\end{equation}
	To efficiently apply Algorithm~\ref{NM4} to solving the subgradient system
	\eqref{sumLasso}, we have to provide {\em explicit calculations} of $\partial\ph$, $\partial^2g$, and ${\rm Prox}_{\gg g}$ via the initial data of \eqref{Lasso}. Here they are. 
	
	\begin{Proposition}[\bf explicit calculations for Lasso]\label{lasso-calc} For problem \eqref{Lasso}, the following hold:
		\begin{equation}\label{proxofg}
			\big(\text{\rm Prox}_{\gamma g}(x)\big)_i=\begin{cases}
				x_i-\mu\gamma&\text{if}\quad x_i>\mu\gamma,\\
				0&\text{if}\quad-\mu\gamma\le x_i\le\mu\gamma,\\
				x_i+\mu\gamma&\text{if}\quad x_i<-\mu\gamma.
			\end{cases}
		\end{equation}
		\begin{equation}\label{norm1}
			\partial\varphi(x)=A^*(Ax-b)+\mu F(x)\;\mbox{ for all }\;x\in\R^n,
		\end{equation}
		where the set-valued mapping $F\colon\R^n\tto\R^n$ is computed by
		\begin{equation}\label{F}
			F(x)=\left\{v\in\R^n\;\bigg|\;
			\begin{array}{@{}cc@{}}
				v_j=\text{\rm sgn}(x_j),\;x_j\ne 0,\\
				v_j\in[-1,1],\;x_j=0
			\end{array}\right\}.
		\end{equation}
		For each $(x,y)\in\text{\rm gph}\,\partial g$ and $v=(v_1,\ldots,v_n)\in\R^n$, we have 
		\begin{equation}\label{second-order}
			\partial^2g(x,y)(v)=\Big\{w\in\R^n\;\Big|\;\Big(\frac{1}{\mu}w_i,-v_i\Big)\in G\Big(x_i,\frac{1}{\mu}y_i\Big),\;i=1,\ldots,n\Big\},
		\end{equation}
		where the mapping $G\colon\R^2\tto\R^2$ is given by
		\begin{equation}\label{G}
			G(t,p):=\begin{cases}
				\{0\}\times\R&\text{\rm if}\quad t\ne 0,\;p\in\{-1,1\},\\
				\R\times\{0\}&\text{\rm if}\quad t=0,\;p\in(-1,1),\\
				(\R_{+}\times\R_{-})\cup(\{0\}\times\R)\cup(\R\times\{0\})&\text{\rm if}\quad t=0,\;p=-1,\\
				(\R_{-}\times\R_{+})\cup(\{0\}\times\R)\cup(\R\times\{0\})&\text{\rm if}\quad t=0,\;p=1,\\
				\emp&\text{\rm otherwise}.
			\end{cases}
		\end{equation}
	\end{Proposition}
	\begin{proof}  The formula for the proximal mapping \eqref{proxofg} follows from definition \eqref{Prox}. Proceeding with the subdifferential calculation, we clearly get $\partial(\|\cdot\|_1)(x)= F(x)$ with $F(x)$ is computed by \eqref{F}, and thus
		\begin{equation}\label{first-order}
			\partial g(x)=\left\{v\in\R^n\;\bigg|\;
			\begin{array}{@{}cc@{}} v_j=\text{\rm sgn}(x_j),\;x_j\ne 0,\\
				v_j\in[-\mu,\mu],\;x_j=0
			\end{array}\right\} \quad\text{whenever }\; x\in\R^n.
		\end{equation}
		Then \eqref{norm1} follows from \eqref{first-order} and the subdifferential sum rule of convex analysis. 
		
		It remains to verify the second-order subdifferential formula \eqref{second-order} for $g$ at $(x,y)\in\gph\partial g$. Considering the auxiliary function $\psi(t):=|t|$ on $\R$, observe that
		\begin{equation*}
			\gph\partial\psi=\big((\R\setminus\{0\})\times\{-1,1\}\big)\cup\big(\{0\}\times[-1,1]\big)\;\mbox{ and }\;\|x\|_1=\sum_{i=1}^n\psi(x_i),\quad x\in\R^n.
		\end{equation*}
		Furthermore, we see that $N_{\text{gph}\,\partial\psi}=G$ for the mapping $G$ defined in \eqref{G}. It now follows from Mordukhovich and Outrata \cite[Theorem~4.3]{BorisOutrata} that
		\begin{equation*}
			\partial^2(\|\cdot\|_1)\Big(x,\frac{1}{\mu}y\Big)(v)=\Big\{u\in\R^n\;\Big|\;(u_i,-v_i)\in N_{\text{gph}\,\partial\psi}\Big(x_i,\frac{1}{\mu}y_i\Big) \Big\},
		\end{equation*}
		which therefore justifies the fulfillment of \eqref{second-order} and thus completes the proof of the proposition.
	\end{proof} 
	
	Note that by applying to \eqref{Lasso} the second-order subdifferential sum rule from Mordukhovich \cite[Theorem~1.21]{Mordukhovich06} and using \eqref{second-order}, we easily arrive at the calculation of $\partial^2\ph$, which is not needed to run Algorithm~\ref{NM4} for \eqref{Lasso}. However, to run our algorithm, we have to explicitly determine the sequences $\{v^k\}$ and $\{d^k\}$ taken from \eqref{subprob1} and \eqref{prox-dir3}, respectively. For the case of $v^k$, it follows from 
	\eqref{proxofg} that 
	$$
	\left(v^k\right)_i=\begin{cases}
		(u^k)_i-\mu\gamma&\text{if}\quad (u^k)_i>\mu\gamma,\\
		0&\text{if}\quad-\mu\gamma\le (u^k)_i\le\mu\gamma,\\
		(u^k)_i+\mu\gamma&\text{if}\quad (u^k)_i<-\mu\gamma. \end{cases} \quad \text{where }\; u^k= x^k - \gamma (A^*Ax^k+b).
	$$
	To find $d^k$, we substitute into \eqref{prox-dir3} the calculation of $\partial^2g$ from \eqref{second-order}, which brings us to the conditions
	\begin{equation*}
		\begin{cases}
			\big(-\frac{1}{\gamma}(x^k-v^k)-A^*Ad^k \big)_i=0&\text{if}\quad\left(v^k\right)_i\ne 0,\\
			\big(x^k-v^k + d^k\big)_i=0&\text{if}\quad\left(v^k\right)_i=0.
		\end{cases} 
	\end{equation*} 
	Thus $d^k$ can be computed by solving the linear equation $X^k d = v^k-x^k$, where
	\begin{equation}\label{X} (X^k)_i := \begin{cases}
			\gamma(A^*A)_i & \text{if} \quad (v^k)_i \ne 0,\\
			I_i  & \text{if}\quad (v^k)_i=0. \end{cases},
	\end{equation}
	where $(X^k)_i$, $(A^*A)_i$, $I_i$ are  the $i^{\text{th}}$  rows of the matrices $X^k$, $A^*A$, and $I$, respectively.
	
	\medskip
	The above calculations allow us to apply Algorithm~\ref{NM4} to solving the Lasso problem \eqref{Lasso}.\vspace*{-0.05in}
	
	\begin{Theorem}[\bf solving Lasso]\label{solveLasso} Let $\ox$ be an optimal solution to 
		\eqref{Lasso} such that $\partial\varphi$ is metrically regular around $(\ox,0)$. Then all the assumptions of Theorem~{\em\ref{localNM4}} are satisfied at $\ox$
		Thus Algorithm~{\rm\ref{NM4}}, with the
		parameter $\gamma>0$ for which $I-\gamma A^*A$ is nonsingular and the data
		$\partial\ph$, $\partial^2 g$, and ${\rm Prox}_{\gamma g}$
		explicitly computed in terms of $(A,b,\mu)$ in
		Proposition~{\rm\ref{lasso-calc}}, is well-defined around $\ox$, and the sequence of its
		iterates $\{x^k\}$ locally $Q$-superlinearly converges to $\ox$. Furthermore, $\ox$ is a
		tilt-stable local minimizer of $\varphi$. 
	\end{Theorem}
	\begin{proof} Observe first that the subdifferential mapping $\partial\ph$ enjoys the property of strong metric regularity around $(\ox,0)$. This follows from the convexity of $\ph$ and the result of Arag\'on Artacho and Geoffroy \cite[Proposition~3.8]{ag}. Furthermore, we get from the formula for
		the mapping $F$ in Proposition~\ref{lasso-calc} that the
		graph of $F$ is the union of finitely many closed convex polyhedra, and
		hence $F$ is semismooth$^*$ on its graph. Since
		the function $x\mapsto A^*(Ax-b)$ is continuously differentiable on
		$\R^n$, it follows from the subdifferential formula \eqref{norm1} and Proposition~3.6 from Gfrerer and Outrata \cite{Helmut} that the mapping $\partial\varphi$ is semismooth$^*$ at every point of its
		graph. Taking now into account the solvability results of
		Section~\ref{sec:solvN}, we see that all the assumptions of
		Theorem~\ref{localNM4} hold for the Lasso problem \eqref{Lasso}. Therefore, the solvability and convergence conclusions of the theorem follow from Theorem~\ref{localNM4}, where the obtained the specification of
		$\gamma$ is due to the convexity of $g$. Finally, the
		tilt-stability of $\ox$ is a consequence of Proposition~3.1 from Drusvyatskiy and Lewis \cite{dl}.
	\end{proof} 
	
	\medskip
	The next remark provides an essential clarification of our assumptions. 
	
	\begin{Remark}[\bf positive-definiteness and metric regularity]\label{pos-defMT} {\rm Note that when the matrix $A^*A$ is positive-definite, the metric regularity of $\partial\varphi$ around
			the reference point in Theorem~\ref{solveLasso} is automatically
			satisfied. Observe that the metric regularity assumption is {\em weaker} than the
			positive-definiteness of $A^*A$. A simple class of functions $\ph$
			illustrating this is given by $\ph=f+|x|$. In
			particular, for $f\equiv 0$ we have \begin{equation*}
				\partial^2\ph(0,0)(v)=\big\{w\in\R\;\big|\;(w,-v)\in\R\times\{0\}\big\},
			\end{equation*} and thus $\text{\rm ker}\,\partial^2\ph(0,0)=\{0\}$,
			which tells us by \eqref{Morcri} that $\partial\ph$ is metrically
			regular around $(0,0)$. Some sufficient conditions for the	metric regularity of the subdifferential operator of the objective function \eqref{Lasso} can be found in the recent preprint by Berk et al. \cite{bbh22}.} 
	\end{Remark}
	
	In the remainder of this section, we develop applications of Algorithm~\ref{NM4} to solving the subgradient inclusion $0\in\partial\varphi(x)$ corresponding to the stationary condition for the nonsmooth and {\em nonconvex} composite optimization problem formulated as follows:
	\begin{eqnarray}\label{nonconvexcomposite} 
		\text{minimize
		}\;\varphi(x):=\frac{1}{2}\|Ax-b\|_2^2+\mu_1\|x\|_1-\mu_2
		\|x\|_2^2,\quad x\in\R^n,
	\end{eqnarray} where $A$ is an $m\times n$ matrix, $b\in\R^m$, and $\mu_1,\mu_2>0$. We label \eqref{nonconvexcomposite} as the $\ell^{1,2}$-{\em regularized least square optimization problem}. 
	It is clear that \eqref{nonconvexcomposite} is a subclass of \eqref{regression} whose cost function is represented in the composite form $\varphi:=f+g$, where \begin{equation}\label{fg2}
		f(x):=\frac{1}{2}\langle\Tilde{A}x,x\rangle+\langle\Tilde{b},x\rangle+\Tilde{\alpha}\quad\text{and
		}\;g(x):=\mu_1\|x\|_1-\mu_2 \|x\|_2^2 
	\end{equation} with
	$\Tilde{A}:=A^*A$, $\Tilde{b}:=-A^*b$, and
	$\Tilde{\alpha}:=\frac{1}{2}\|b\|^2$, and where the matrix
	$\Tilde{A}=A^*A$ is symmetric and positive-semidefinite. Note that the regularizer $g$ is nonsmooth and nonconvex, while being prox-bounded and continuously prox-regular on $\R^n$
	with modulus $2\mu_2$ due to the convexity of  the quadratically shifted function $g + \mu_2 \|\cdot\|^2$. This allows us to implement Algorithm~\ref{NM4} to solving the subgradient
	system $0\in \partial\varphi(x)$, where $\varphi$ is taken from
	\eqref{nonconvexcomposite}. To proceed, we  provide {\em explicit
		calculations} of all the algorithm ingredients in terms of the initial data of problem \eqref{nonconvexcomposite}.\vspace*{-0.05in}
	
	\begin{Proposition}[\bf precise algorithmic calculations for
		$\ell^{1,2}$-regularizer]\label{lasso-calc2} For the cost function $\ph$ in \eqref{nonconvexcomposite}, we have the subdifferential formula
		\begin{equation}\label{norm1-2}
			\partial\varphi(x)=A^*(Ax-b)+\mu_1 F(x) - 2\mu_2x \;\mbox{ whenever }\;x\in\R^n,
		\end{equation}
		where the set-valued mapping $F\colon\R^n\tto\R^n$ is taken from \eqref{F}. The proximal mapping \eqref{Prox} associated with the regularizer $g(x)=\mu_1\|x\|_1-\mu_2\|x\|_2^2$ is computed by 
		\begin{equation}\label{proxofg2}
			\big(\text{\rm Prox}_{\gamma g}(x)\big)_i=\begin{cases}
				\disp\frac{x_i - \gamma\mu_1}{1-2\gamma\mu_2} &\text{if}\quad x_i>\mu_1\gamma,\\
				0&\text{if}\quad-\mu_1\gamma\le x_i\le\mu_1\gamma,\\
				\disp\frac{x_i + \gamma\mu_1}{1-2\gamma\mu_2}&\text{if}\quad x_i<-\mu_1\gamma
			\end{cases}
		\end{equation}
		provided that $\gamma\in \left(0,\frac{1}{2\mu_2} \right)$. For each $(x,y)\in\text{\rm gph}\,\partial g$ and $v=(v_1,\ldots,v_n)\in\R^n$, the generalized Hessian of $g$ is computed by the formula
		\begin{equation}\label{second-order2}
			\partial^2g(x,y)(v)=\Big\{w\in\R^n\;\Big|\;\Big(\frac{1}{\mu_1}(w_i+2\mu_2 v_i),-v_i\Big)\in G\Big(x_i,\frac{1}{\mu_1}(y_i +2\mu_2 x_i)\Big),\;i=1,\ldots,n\Big\},
		\end{equation}
		where the mapping $G\colon\R^2\tto\R^2$ is taken from \eqref{G}.
	\end{Proposition}
	\begin{proof} The calculation of ${\rm Prox}_{\gamma g}(x)$ in \eqref{proxofg2} follows from the definition. Also, we can easily get
		\begin{equation}\label{first-order2}
			\partial g(x)=\mu_1 F(x) - 2\mu_2x \quad\text{whenever }\; x\in\R^n,
		\end{equation}
		where $F\colon\R^n\tto\R^n$ is defined in  \eqref{F}. Thus the subdifferential formula \eqref{norm1-2} is a consequence of \eqref{first-order2} and the standard first-order subdifferential sum rule. Arguing similarly to the proof of  Proposition~\ref{lasso-calc}, we verify the second-order calculation formula \eqref{second-order2} and thus complete the proof.
	\end{proof} 
	
	The obtained calculations allow us to compute parameters $v^k$ and $d^k$ in Algorithm~\ref{NM4} to solve the stationary inclusion $0\in\partial\ph(x)$ for problem \eqref{nonconvexcomposite}. Indeed, it follows from \eqref{proxofg2} that
	$$
	\left(v^k\right)_i=\begin{cases}
		\disp\frac{(u^k)_i - \gamma\mu_1}{1-2\gamma\mu_2}&\text{if}\quad (u^k)_i>\mu_1\gamma,\\
		0&\text{if}\quad-\mu_1\gamma\le (u^k)_i\le\mu_1\gamma,\\
		\disp\frac{(u^k)_i + \gamma\mu_1}{1-2\gamma\mu_2}&\text{if}\quad
		(u^k)_i<-\mu_1\gamma,
	\end{cases} \quad \text{where }\; u^k= x^k - \gamma (A^*Ax^k+b).
	$$ 
	Using further the formulas in \eqref{second-order2} and
	\eqref{first-order2}, we get the relationships 
	\begin{equation*}
		\begin{cases}
			\big(-\frac{1}{\gamma}(x^k-v^k)-A^*Ad^k +2\mu_2(x^k-v^k+d^k)  \big)_i=0&\text{if}\quad\left(v^k\right)_i\ne 0,\\
			\big(x^k-v^k + d^k\big)_i=0&\text{if}\quad\left(v^k\right)_i=0
		\end{cases}
	\end{equation*} 
	for $d^k$ and hence find these directions by solving the equation $X^k d =(1-2\gamma\mu_2)(v^k-x^k)$, where 
	\begin{equation}\label{X2} (X^k)_i := \begin{cases}
			\gamma(A^*A)_i - 2\gamma\mu_2 I_i & \text{if} \quad (v^k)_i \ne 0,\\
			(1- 2\gamma\mu_2)I_i  & \text{if}\quad (v^k)_i=0. \end{cases}
	\end{equation}
	
	Employing the above calculations in the framework of Theorem~\ref{localNM4}, we arrive at the following efficient computational procedure to solve problem \eqref{nonconvexcomposite} with justifying its excellent performance.\vspace*{-0.05in}
	
	\begin{Theorem}\label{solvenonconvex}  Let $\ox$ be such that $0 \in
		\partial\varphi(\ox)$ with $\varphi$ defined in
		\eqref{nonconvexcomposite}, and let $\partial\varphi$ be metrically
		regular around $(\ox,0)$. Then all the assumptions of Theorem~{\rm\ref{localNM4}} are satisfied for Algorithm~{\rm\ref{NM4}} with the parameter $\gamma\in \left( 0, \frac{1}{2\mu_2} \right)$
		such that $I-\gamma A^*A$ is nonsingular and the data $\partial\ph$,
		$\partial^2 g$, and ${\rm Prox}_{\gamma g}$ explicitly computed in
		terms of $(A,b,\mu)$ by the formulas in
		Proposition~{\rm\ref{lasso-calc2}}. Thus Algorithm~{\rm\ref{NM4}} is well-defined around $\ox$ for this problem, and the sequence of its
		iterates $\{x^k\}$ $Q$-superlinearly converges to $\ox$. 
	\end{Theorem}
	\begin{proof}  According to Theorem~\ref{twiceepi}, the second-order subdifferential inclusion
		\eqref{newton-inc} is robustly solvable around $(\ox,0)$ for $\ph$ from \eqref{nonconvexcomposite} provided that the  mapping $\partial\ph$ is metrically regular around $(\ox,0)$ as assumed, and if  $\ph$ and twice epi-differentiable
		at any points near $\ox$. Let us check the fulfillment of the latter property on the whole space $\R^n$. Indeed, observe that $\varphi(x)=\varphi_1 + \varphi_2$, where
		$$
		\varphi_1(x):= \frac{1}{2}\|Ax-b\|^2 - \mu_2 \|x\|_2^2 \quad \text{and }\; \varphi_2(x):= \mu_1 \|x\|_1.
		$$
		Since $\varphi_2$ is a proper, convex, and piecewise
		linear-quadratic function on $\R^n$, it follows from Rockafellar and Wets
		\cite[Proposition~13.9]{Rockafellar98} that $\varphi_2$ is twice
		epi-differentiable on $\R^n$. Furthermore, the function $\ph_1$ also enjoys this property on $\R^n$, since it is $\mathcal{C}^2$-smooth of $\varphi_1$; see Exercise~13.18 in the aforementioned book \cite{Rockafellar98}. This clearly verifies the twice epi-differentiability of the sum function $\ph$.
		
		It remains to check that the mapping $\partial\ph$ is semismooth$^*$ at the reference point. But it can be done similarly to the proof of Theorem~\ref{solveLasso} due to the subdifferential expression in \eqref{norm1-2}.
	\end{proof}\vspace*{-0.25in}

\color{black}\section{Concluding Remarks}\label{conc}\vspace*{-0.05in}
	This paper proposes and develops a generalized Newton method to
	solve gradient and subgradient systems by using second-order
	subdifferentials of ${\cal C}^{1,1}$ and extended-real-valued
	prox-regular functions, respectively, together with the appropriate
	tools of second-order variational analysis. The suggested algorithms
	for gradient and subgradient systems are comprehensively
	investigated with establishing verifiable conditions for their
	well-posedness/solvability and local superlinear convergence.
	Applications to solving important classes of nonsmooth
	{regularized least square problems} are given with complete computations of algorithm ingredients and assumption verifications.
	
	The results of this paper concern far-going nonsmooth extensions of the basic Newton method for ${\cal C}^2$-smooth functions. Besides further implementations of the obtained results and their applications to practical models, in our future research we intend to develop other, more advanced versions of nonsmooth second-order algorithms of the Newton and quasi-Newton types with establishing their local and global convergence as well as efficient specifications for broad classes of variational inequalities and constrained optimization problems. \\[1ex]
{\bf Acknowledgments.} Research of the first author is funded by Ho Chi Minh City University of Education Foundation for Science and Technology under grant number CS.2022.19.04TD. Research of the second and third authors was
	partly supported by the US National Science Foundation under grants
	DMS-1808978 and DMS-2204519. The research of the
	second author was also supported by the Australian Research Council
	under Discovery Project DP-190100555.  The authors are very grateful to three anonymous referees for their helpful remarks and suggestions, which allowed us to significantly improve the original presentation. 


\vspace*{-0.1in}
 
\small \begin{thebibliography}{10} 

\bibitem{ag} Arag\'on Artacho FJ, Geoffroy MH (2008)
Characterizations of metric regularity of subdifferentials, J.
Convex Anal. 15:365--380.

\bibitem{Beck} Beck A (2017) First-Order Methods in Optimization
(SIAM, Philadelphia, PA).

\bibitem{bbh22} Berk A, Brugiapaglia S, Hoheisel T (2022) Lasso reloaded: a variational analysis perspective with applications to compressed sensing. arXiv preprint arXiv:2205.06872.

\bibitem{Bonnans} Bonnans JF (1994) Local analysis of Newton-type
methods for variational inequalities and nonlinear programming.
Appl. Math. Optim. 29: 161--186.

\bibitem{Burke} Burke JV, Qi L (1991) Weak directional closedness
and generalized subdifferentials. J. Math. Anal. Appl. 159:485--499.

\bibitem{ChieuNghia} Chieu NM, Hien LV, Nghia TTA (2018)
Characterization of tilt stability via subgradient graphical
derivative with applications to nonlinear programming. SIAM J.
Optim. 28:2246--2273.

\bibitem{ChieuLee17} Chieu NH, Lee GM, Yen ND (2017) Second-order
subdifferentials and optimality conditions for ${\cal C}^1$-smooth
optimization problems. Appl. Anal. Optim. 1:461--476.

\bibitem{cl} Clarke FH (1983) Optimization and Nonsmooth Analysis
(Wiley, New York).

\bibitem{chhm} Colombo G, Henrion R, Hoang ND, Mordukhovich BS
(2016) Optimal control of the sweeping process over polyhedral
controlled sets, J. Diff. Eqs. 260:3397--3447.

\bibitem{ds} Dias S, Smirnov G (2012) On the Newton method for
set-valued maps. Nonlinear Anal. TMA 75:1219--1230.

\bibitem{dsy} Ding C, Sun D, Ye JJ (2014) First-order optimality
conditions for mathematical programs with semidefinite cone
complementarity constraints. Math. Program. 147:539--379.

\bibitem{dr} Dontchev AL, Rockafellar RT (1996) Characterizations of
strong regularity for variational inequalities over polyhedral
convex sets. SIAM J. Optim. 6:1087--1105.

\bibitem{Donchev09} Dontchev AL, Rockafellar RT (2014) Implicit
Functions and Solution Mappings: A View from Variational Analysis
(2nd edition, Springer, New York).

\bibitem{dl} Drusvyatskiy D, Lewis AS (2013) Tilt stability, uniform
quadratic growth, and strong metric regularity of the
subdifferential. SIAM J. Optim. 23:256--267.

\bibitem{dmn} Drusvyatskiy D, Mordukhovich BS, Nghia TTA (2014)
Second-order growth, tilt stability, and metric regularity of the
subdifferential. J. Convex Anal. 21:1165--1192.

\bibitem{JPang} Facchinei F, Pang J-C (2003) Finite-Dimensional
Variational Inequalities and Complementarity Problems, Vol. II
(Springer, New York).

\bibitem{g} Gfrerer H (2013) On directional metric regularity,
subregularity and optimality conditions for nonsmooth mathematical
programs. Set-Valued Var. Anal. 21:151--176.

\bibitem{gm} Gfrerer H, Mordukhovich BS (2015) Complete
characterization of tilt stability in nonlinear programming under
weakest qualification conditions. SIAM J. Optim. 25:2081--2119.

\bibitem{Helmut} Gfrerer H, Outrata JV (2021) On a semismooth$^*$
Newton method for solving generalized equations. SIAM J. Optim. 31:489--517.

\bibitem{gin-mor} Ginchev I, Mordukhovich BS (2011) On directionally
dependent subdifferentials. C. R. Acad. Bulg. Sci. 64:497--508.

\bibitem{Hare} Hare LW, Sagastiz\'abal C (2009) Computing proximal
points of nonconvex functions. Math. Program. 116:221--258.

\bibitem{hmn} Henrion R, Mordukhovich BS, Nam NM (2010) Second-order
analysis of polyhedral systems in finite and infinite dimensions
with applications to robust stability of variational inequalities.
SIAM J. Optim. 20:2199--2227.

\bibitem{ho} Henrion R, Outrata JV (2001) A subdifferential
condition for calmness of multifunctions. J. Math. Anal. Appl.
258:110--130.

\bibitem{hos} Henrion R, Outrata JV, Surowiec T (2012) Analysis of
M-stationary points to an EPEC modeling ologopolistic competition in
an electricity spot market. ESAIM: Contr. Optim. Calc. Var.
18:295--317.

\bibitem{HungBoris} Hoheisel T, Kanzow C, Mordukhovich BS, Phan H
(2012) Generalized Newton's method for nonsmooth equations based on
graphical derivatives. Nonlinear Anal. 75:1324--1340.

\bibitem{defeng} Jiang H, Qi L, Chen X, Sun D (1996) Semismoothness
and superlinear convergence in nonsmooth optimization and nonsmooth
equations.  De Pillo G, Giannessi F, eds. Nonlinear Optimization and
Applications (Springer, New York), 197--212.


\bibitem{Solo14} Izmailov AF, Solodov MV (2014) Newton-Type Methods
for Optimization and Variational Problems (Springer, New York).

\bibitem{kanzow21} Kanzow C, Lechner T (2021) Globalized inexact proximal newton-type methods for nonconvex composite
functions. Comput Optim Appl. 78(2):377–410

\bibitem{josephy} Josephy NH (1979) Newton's method for generalized
equations. Technical Summary Report No.\ 1965. Mathematics Research
Center, University of Wisconsin (Madison, WI).

\bibitem{kender} Kenderov P (1975) Semi-continuity of set-valued
monotone mappings. Fundamenta Mathematicae 88:61--69.

\bibitem{kmpt221} Khanh PD, Mordukhovich B, Phat VT, Tran BD (2021) Globally convergent coderivative-based gener-
alized newton methods in nonsmooth optimization. arXiv preprint arXiv:2109.02093.

\bibitem{Klatte} Klatte D, Kummer B (2002) Nonsmooth Equations in
Optimization. Regularity, Calculus, Methods and Applications (Kluwer
Academic Publishers, Dordrecht, The Netherlands).



\color{black}\bibitem{Kummer} Kummer B (1988) Newton's method for
non-differentiable functions. Guddat J et al., eds. Advances in
Mathematical Optimization (Akademie-Verlag, Berlin), 114--124.

\bibitem{lss} Lee JD, Sun Y, Saunders MA (2014) Proximal Newton-type
methods for minimizing composite functions. SIAM J. Optim.
24:1420--1443.

\bibitem{msz} Meng F, Sun D, Zhao Z (2005) Semismoothness of
solutions to generalized equations and the Moreau-Yosida
regularization. Math. Program. 104:561--581.

\bibitem{Mifflin} Mifflin R (1977) Semismooth and semiconvex
functions in constrained optimization. SIAM J. Control Optim.
15:957--972.

\bibitem{Milzarek} Milzarek A (2016) Numerical methods and second
order theory for nonsmooth problems (doctoral thesis, advisor
Ulbrich M,  Technical University of Munich, Germany).

\color{black}\bibitem{mms} Mohammadi A, Mordukhovich BS, Sarabi ME (2022) Variational analysis of composite models with
applications to continuous optimization. Math. Oper. Res. 47(1):397–426.

\bibitem{mms1} Mohammadi A, Mordukhovich BS, Sarabi ME (2021)
Parabolic regularity in geometric variational analysis. Trans. Amer.
Math. Soc. 374:1711--1763.

\bibitem{ms} Mohammadi A, Sarabi ME (2020) Twice
epi-differentiability of extended-real-valued functions with
applications in composite optimization. SIAM J. Optim. 30: 2379--2409.

\bibitem{m92} Mordukhovich BS (1992) Sensitivity analysis in
nonsmooth optimization. Field DA, Komkov V, eds. Theoretical Aspects
of Industrial Design, SIAM Proc. Appl. Math. 58:32--46
(Philadelphia, PA).

\bibitem{Mordu93} Mordukhovich BS (1993) Complete characterizations
of openness, metric regularity, and Lipschitzian properties of
multifunctions. Trans. Amer. Math. Soc. 340:1--35.

\bibitem{Mordukhovich06} Mordukhovich BS (2006) Variational Analysis
and Generalized Differentiation, I: Basic Theory, II: Applications
(Springer, Berlin).

\bibitem{Mor18} Mordukhovich BS (2018) Variational Analysis and
Applications (Springer, Cham, Switzerland).

\bibitem{MorduNghia} Mordukhovich BS, Nghia TTA (2015) Second-order
characterizations of tilt stability with applications to nonlinear
programming. Math. Program. 149:83--104.

\bibitem{BorisOutrata} Mordukhovich BS, Outrata JV (2001) On
second-order subdifferentials and their applications, SIAM J. Optim.
12:139--169.

\bibitem{mr} Mordukhovich BS, Rockafellar RT (2012) Second-order
subdifferential calculus with applications to tilt stability in
optimization. SIAM J. Optim. 22:953--986.

\bibitem{BorisEbrahim} Mordukhovich BS, Sarabi ME (2021) Generalized
Newton algorithms for tilt-stable minimizers in nonsmooth
optimization. SIAM J. Optim. 31:1184--1214.

\bibitem{myzz} Mordukhovich BS, Yuan X, Zeng S, Zhang J (2022) A globally convergent proximal newton-type method
in nonsmooth convex optimization. Mathematical Programming 1–38.

\bibitem{os} Outrata JV, Sun D (2008) On the coderivative of the
projection operator onto the second-order cone. Set-Valued Anal.
16:999--1014.

\bibitem{p90} Pang J-S (1990) Newton's method for B-differentiable
equations. Math. Oper. Res. 15:311--341.

\bibitem{pb} Patrinos P, Bemporad A (2013) Proximal Newton methods
for convex composite optimization. Proc. IEEE Conf. Dec. Cont.
2358--2363 (Florence, Italy).

\bibitem{Poliquin} Poliquin RA, Rockafellar RT (1996) Prox-regular
functions in variational analysis. Trans. Amer. Math. Soc.
348:1805--1838.

\bibitem{Poli} Poliquin RA, Rockafellar RT (1998) Tilt stability of
a local minimum. SIAM J. Optim. 8:287--299.

\bibitem{LQi} Qi L, Sun J (1993) A nonsmooth version of Newton's
method. Math. Program. 58: 353--367.

\bibitem{rob} Robinson SM (1979) Generalized equations and their
solutions, I: basic theory. Math. Program. Study 10:128--141.

\bibitem{Rockafellar98} Rockafellar RT, Wets RJ-B (1998) Variational
Analysis (Springer, Berlin).

\bibitem{rock-zag} Rockafellar RT, Zagrodny D (1997) A
derivative-coderivative inclusion in second-order nonsmooth
analysis. Set-Valued Anal. 5:1--17.

\bibitem{Sun2001} Sun D  (2001) A further result on an implicit
function theorem for locally Lipschitz functions. Oper. Res. Lett.
28:193--198.

\bibitem{tsp} Themelis A, Stella L, Patrinos P (2918)
Forward-backward envelope for the sum of two nonconvex functions:
Further properties and nonmonotone linesearch algorithms. SIAM J.
Optim. 28:2274--2303.

\bibitem{Tibshirani} Tibshirani R (1996) Regression shrinkage and
selection via the Lasso. J. R. Stat. Soc. 58:267--288.

\bibitem{Ul} Ulbrich M (2011) Semismooth Newton Methods for
Variational Inequalities and Constrained Optimization Problems in
Function Spaces (SIAM, Philadelphia, PA).

\bibitem{wang} Wang X (2004) Subdifferentiability of real functions.
Real Anal. Exchange 30:137--172.

\bibitem{yy} Yao J-C, Yen ND (2009) Coderivative calculation related
to a parametric affine variational inequality. Part~1: Basic
calculation. Acta Math. Vietnam. 34:157--172. 
\end{thebibliography}
\end{document}